\documentclass[11pt]{amsart}

\usepackage[a4paper]{geometry}
\usepackage{latexsym, amsthm, amsfonts, amsmath,amssymb,graphicx,lmodern,mathrsfs}

\usepackage[latin1]{inputenc} 
\usepackage[T1]{fontenc} 
\usepackage[english]{babel} 

\newtheorem{theo}{Theorem}[section]
\newtheorem{lem}[theo]{Lemma}
\newtheorem{propo}[theo]{Proposition}
\newtheorem{coro}[theo]{Corollary}

\theoremstyle{definition}
\newtheorem{defi}[theo]{Definition}

\theoremstyle{remark}
\newtheorem{rem}[theo]{Remark}
\newtheorem{ex}[theo]{Example}
\newtheorem*{ex2}{Example}

\makeatletter

\@addtoreset{equation}{section}
\makeatother
\def\R{\mathbb{R}}
\def\Z{\mathbb{Z}}
\def\C{\mathbb{C}}
\def\N{\mathbb{N}}
\def\Q{\mathbb{Q}}

\def\n{\eta}
\def\r{\rho}
\def\s{\sigma}
\def\a{\alpha}
\def\e{\varepsilon}

\def\f{\varphi}

\def\b{\beta}

\def\n'{\nu}

\def\g{\gamma}

\def\k{\kappa}

\def\G{\Gamma}

\def\pz {\partial_{z}}
\def\pt {\partial_{t}}
\def\dt {\Delta_{t}}

\begin{document}
\sloppy
\title{A density theorem in parameterized differential Galois theory.}

\author{Thomas Dreyfus}
\address{Université Paris Diderot - Institut de Mathématiques de Jussieu,}
\curraddr{4, place Jussieu 75005 Paris.}
\email{tdreyfus@math.jussieu.fr.}
\thanks{Work partially supported by ANR, contract ANR-06-JCJC-0028}
\subjclass[2010]{Primary 34M15, 12H20, 34M03.}


\date{\today}


\begin{abstract}
We study parameterized linear differential equations with coefficients depending meromorphically  upon  the  parameters. As a main result, analogously to the unparameterized density theorem of Ramis, we show that the parameterized  monodromy,  the  parameterized exponential  torus  and  the  parameterized  Stokes  operators are topological generators in  Kolchin  topology, for  the parameterized differential Galois group introduced by Cassidy and Singer. We  prove an  analogous  result  for  the global  parameterized  differential  Galois  group, which generalizes a result by Mitschi and Singer. These authors give also a  necessary  condition  on  a  group  for being a global parameterized differential Galois group;
as a corollary of the density theorem, we prove that their condition is also sufficient. As an application, we  give  a  characterization  of completely integrable equations,  and  we  give  a  partial answer to a question of Sibuya about the transcendence properties of a given Stokes matrix. Moreover, using a parameterized Hukuhara-Turrittin theorem, we show that the Galois group descends to  a  smaller  field, whose  field  of  constants  is not differentially  closed. 
 \end{abstract} 

\maketitle

\tableofcontents

\pagebreak[3]
\section*{Introduction}
Let us consider a linear differential system of the form 
$$ \pz Y(z)=A(z)Y(z),$$
 where~$\pz=\frac{d}{dz}$, and~$A(z)$ is an~$m\times m$ matrix whose entries are germs of meromorphic functions in a neighborhood of a point, say~$0$ to fix ideas. The differential Galois group, which measures the algebraic dependencies among the solutions, can be viewed as an algebraic subgroup of~$\mathrm{GL}_{m}(\C)$ via the injective group morphism
$$\begin{array}{cccc}
\r_{U}: & Gal & \longrightarrow & \mathrm{GL}_{m}(\C) \\ 
 & \s  &\mapsto  & U(z)^{-1}\s (U(z)),
\end{array}~$$
where~$U(z)$ is some arbitrary fundamental solution, i.e., an invertible solution matrix. \par
Let~$U(z)$ be a fundamental solution contained in a Picard-Vessiot extension of the equation~${\pz Y(z)=A(z)Y(z)}$. The linear differential equation is said to be regular singular at~$0$ if there exists an invertible matrix~$P(z)$ whose entries are germs of meromorphic functions such that~$W(z)=P(z)U(z)$ satisfies ~$$\pz W(z)=\frac{A_{0}}{z}W(z),$$ where~$A_{0}$ is a matrix with constant complex entries. In this case,~$W(z)$ usually involves multivalued functions. 
Analytic continuation of~$W(z)$ along any simple loop~$\gamma$ around~$0$ yields another fundamental solution~$W(z)M_{\gamma}$. The matrix~$M_{\gamma}$, which is a monodromy matrix,  has complex entries, and does not depend on the choice of the homotopy class of~$\gamma$. The Schlesinger theorem says that the Zariski closure of the group generated by the monodromy matrix is the Galois group. In the general case, i.e., in presence of irregular singularities, the monodromy is no longer sufficient to provide a complete collection of topological generators. Ramis has shown that the group generated by the monodromy, the exponential torus and the Stokes operators, which is defined in a transcendental way as a subgroup of the differential Galois group, is dense in the latter in the Zariski topology.\\
\par
More recently, a Galois theory for parameterized linear differential equations of the form
$$ \pz Y(z,t)=A(z,t)Y(z,t),$$
where~$t=(t_{1},\dots,t_{n})$ are parameters, has been developed in \cite{CS}. See also \cite{HS,L,Rob,Um}.
Namely, the Galois group, which measures the~$(\partial_{t_{1}},\dots,\partial_{t_{n}})$-differential and algebraic dependencies among the solutions, can be seen as a differential group in the sense of Kolchin, that is a group of matrices whose entries lie in a differential field and satisfy a set of polynomial differential equations in the variables~$t_{1},\dots,t_{n}$. See \cite{C72,C89,Kol73,Kol85,MO}. To be applied, the theory from \cite{CS} requires the field of constants with respect to~$\pz$ to be of characteristic~$0$ and differentially closed (see~$\S \ref{3sec21}$). The drawback of this latter assumption is that a differentially closed field is a very big field, and cannot be interpreted as a field of functions.\par
There is a link between the parameterized differential Galois theory and isomonodromy for equations with only regular singular poles (see \cite{CS,MS12,MS13}). Let~${\mathcal{D}(t_{0},r)=\big\{(z_{1},\dots,z_{n})\in \C^{n}\big| \forall i\leq n, |z_{i}-t_{0,i}|<r\big\}}$ be an open polydisc in~$\C^{n}$, let~$\mathcal{D}$ be an open subset of~$\C$, and let~$A(z,t)$ be a matrix whose entries are analytic on~$\mathcal{D} \times \mathcal{D}(t_{0},r)$. We consider open disks~$D_{j}$ that cover~$\mathcal{D}$, and a solution~$U_{j}(z,t)$ of~$\pz Y(z,t)=A(z,t)Y(z,t)$ analytic on~$D_{j} \times \mathcal{D}(t_{0},r)$. If~${D_{i}\cap D_{j} \neq \varnothing}$, we define~${C_{i,j}(t)=U_{i}(z,t)^{-1}U_{j}(z,t)}$, the connection matrices. Following Definition~5.2 in \cite{CS} (see also \cite{Bo,M83}), the parameterized linear differential equation~${\pz Y(z,t)=A(z,t)Y(z,t)}$ is said to be isomonodromic\label{3p3} if, there is a choice of~$\left(D_{i}\right)$ covering ~$\mathcal{D}$, and of the solutions~$U_{i}(z,t)$ of~$\pz Y(z,t)=A(z,t)Y(z,t)$ analytic on~$D_{i} \times \mathcal{D}(t_{0},r)$ such that the connection matrices are independent of~$t$. In this case, the matrix of the monodromy is constant on the polydisc ~$\mathcal{D}(t_{0},r)$. When~$A(z,t)$ is of the form~$\sum_{i=1}^{s} \frac{A_{i}(t)}{z-u_{i}}$, such that all the~$A_{i}(t)$ have analytic entries on~$U$ and~$u_{i}\in \mathcal{D}$, the following statements are equivalent (see  \cite{CS}, Propositions 5.3 and 5.4).
\begin{itemize}
\item The Galois group is conjugate, over a differentially closed field (see Definition \ref{3defi4}), to a group of constant matrices.
\item The parameterized linear differential equation is isomonodromic in the above sense.
\item The parameterized linear differential equation is completely integrable (see Definition \ref{3defi2}).
\end{itemize}

\medskip
We are interested in the case where the parameterized linear differential equation may have irregular singularities, in a sense we are going to explain. The main result of this paper is a parameterized analogue of the density theorem of Ramis: we give topological generators for the Galois group in the Kolchin topology (in which closed sets are zero sets of differential algebraic polynomials). As an application of our main result, we improve Proposition 3.9 in \cite{CS} (see Remark \ref{3rem1}): a parameterized linear differential equation is completely integrable if and only if the topological generators for the Galois group just mentioned are conjugate to constant matrices over a field of meromorphic functions. Notice that the latter is not differentially closed.\\
\par
 The article is organized as follows. In the first section we study parameterized linear differential systems from an analytic point of view. The parameters will vary in~$U$, a non-empty polydisc of~$\C^{n}$. Let~$t=(t_{1},\dots,t_{n})\in U$ denote the multiparameter. Let~$\mathcal{M}_{U}$ be the field of meromorphic functions on~$U$ and let~$\hat{K}_{U}=\mathcal{M}_{U}[[z]][ z^{-1}]$ \label{3p1}.
 The Hukuhara-Turrittin theorem in this case gives the following result (see  Remark \ref{3rem2} for a discussion of a similar result present in \cite{Sch}): \\
\pagebreak[3]
\subparagraph{} \textbf{Proposition (see Proposition \ref{3propo4} below).}
\textit{Let~$\pz Y(z,t)=A(z,t)Y(z,t)$, with~$A(z,t)\in \mathrm{M}_{m}\Big( \hat{K}_{U}\Big)$ (that is a~$m\times  m$ matrix with entries in~$\hat{K}_{U}$). Then, there exist a non empty polydisc~$U'\subset U$ and ~$\nu\in \N^{*}$, such that we have a fundamental solution~$F(z,t)$ of the form:~$$F(z,t)=\hat{H}(z,t) z^{L(t)}e^{Q(z,t)},$$ 
where:
\begin{itemize}
 \item ~$\hat{H}(z,t)\in\mathrm{GL}_{m}\Big(\hat{K}_{U'}\big[z^{1/\nu}\big]\Big)$.
 \item ~$L(t)\in \mathrm{M}_{m}(\mathcal{M}_{U'})$.
 \item ~$e^{Q(z,t)}=\mathrm{Diag}(e^{q_{i}(z,t)})$, with~$q_{i}(z,t) \in z^{-1/\nu}\mathcal{M}_{U'} \big[z^{-1/\nu}\big]$.
 \item Moreover, we have~$ z^{L(t)}e^{Q(z,t)}=e^{Q(z,t)}z^{L(t)}$.\\
\end{itemize}}
\par 
See Remark \ref{3rem5} for a discussion about the uniqueness of a fundamental solution of~${\pz Y(z,t)=A(z,t)Y(z,t)}$ written in the above way.\par 
In~$\S \ref{3sec13}$, we briefly review the Stokes phenomenon in the unparameterized case. We have solutions, which are analytic in some sector and  Gevrey asymptotic to the formal part of the solution in the Hukuhara-Turrittin canonical form. The fact that various asymptotic solutions do not glue to a single solution on the Riemann surface of the logarithm is called the Stokes phenomenon. \par 
Let~$ U$ be a non empty polydisc of~$\C^{n}$ and let~$f(z,t)=\displaystyle \sum f_{i}(t)z^{i}\in \hat{K}_{U}$. We say that~$f(z,t)$ belongs to ~$\mathcal{O}_{U}(\{z\})$ \label{3p2} if for all~$t\in U$,~$z\mapsto \displaystyle \sum f_{i}(t)z^{i}$ is a germ of a meromorphic function at~$0$.
Remark that if~${f(z,t)\in\mathcal{O}_{U}(\{z\})\subset \mathcal{M}_{U}\big[\big[z\big]\big]\big[z^{-1}\big]=\hat{K}_{U}}$, then the~$z$-coefficients~$f_{i}(t)$ of~$f(z,t)$ are analytic on~$U$.\par 
In~$\S \ref{3sec14}$, we study the Stokes phenomenon for equations of the form~${\pz Y(z,t)=A(z,t)Y(z,t)}$, with~$A(z,t)\in \mathrm{M}_{m}( \mathcal{O}_{U}(\{z\}))$. In particular, we prove that the asymptotic solutions depend analytically (under mild conditions) upon the parameters.
\\ \par 
In the second section, we use the parameterized Hukuhara-Turrittin theorem to deduce some Galois theoretic properties of parameterized linear differential equations in coefficients in~$\mathcal{O}_{U}(\{z\})$. We first recall some facts from \cite{CS} about parameterized differential Galois theory. The problem is that the theory in \cite{CS} cannot be applied here, since~$\mathcal{M}_{U}$, our field of constants with respect to~$\pz$, is a field of functions that are meromorphic in~$t_{1},\dots,t_{n}$, and this field is not differentially closed (see~$\S  \ref{3sec21}$). In the papers \cite{GGO,Wib}, the authors prove the existence of parameterized Picard-Vessiot extensions under weaker assumptions than in \cite{CS}. See also \cite{CHS,N}. We do not use these latter results because we need a parameterized Hukuhara-Turrittin theorem, which proves directly that a parameterized Picard-Vessiot extension exists, not necessarily unique, in order to study the parameterized Stokes phenomenon. This allow us to define a group, we will call by abuse of language, see Remark \ref{3rem6}, the parameterized differential Galois group.
 In~$\S \ref{3sec23}$, we consider the local case~$\pz Y(z,t)=A(z,t)Y(z,t)$, with~${A(z,t)\in \mathrm{M}_{m}\big( \mathcal{O}_{U}(\{z\})\big)}$. We state and show the main result: \\ 
\pagebreak[3]
 \subparagraph{} \textbf{Parameterized analogue of the density theorem of Ramis  (Theorem \ref{3theo2}).}
\textit{The group generated by the parameterized monodromy, the parameterized exponential torus and the parameterized Stokes operators is dense in the parameterized differential Galois group for the Kolchin topology. \\}
\par
Then, we turn to the global case. We consider equations with coefficients in~$\mathcal{M}_{U}(z)$ and study their global Galois group. We prove a density theorem in this global setting, see Theorem \ref{3theo1}. The proof in the unparameterized case can be found in \cite{M}. In~$\S \ref{3sec25}$, we give various examples of calculations.
\\ \par
In the third section, we give three applications. First, we prove a criterion for the integrability of differential systems (see Definition \ref{3defi2}):\\
\pagebreak[3] \subparagraph{} \textbf{Proposition (see Proposition \ref{3propo6} below).}
\textit{Let~$A(z,t)\in \mathrm{M}_{m}(\mathcal{M}_{U}(z))$. The linear differential equation~$\pz Y(z,t)=A(z,t)Y(z,t)$ 
is completely integrable if and only if there exists a fundamental solution such that the matrices of the parameterized monodromy, the parameterized exponential torus and the parameterized Stokes operators for all the singularities are constant, i.e., do not depend on~$z$. \\}
\par
 As a second application, we give a partial answer to a question of Sibuya (see \cite{Si}), regarding the differential transcendence properties of a Stokes matrix  of the parameterized linear differential equation: 
~$$\begin{pmatrix}
\pz Y(z,t) \\ 
\pz^{2} Y(z,t)
\end{pmatrix} =\begin{pmatrix}
0 & 1 \\ 
z^{3}+t & 0
\end{pmatrix}\begin{pmatrix}
 Y(z,t) \\ 
\pz Y(z,t)
\end{pmatrix}.$$
Sibuya was asking whether an entry of a given Stokes matrix at infinity is~$\pt$-differentially transcendental, i.e., satisfies no differential polynomial equation. We prove that it is at least not~$\pt$-finite, i.e., that it satisfies no linear differential equations.\par
As a last application, we deal with the inverse problem. We prove that if~$G$ is the global parameterized differential Galois group of some equation having coefficients in~$k(z)$ (see~$\S \ref{3sec33}$), then~$G$ contains a finitely generated Kolchin dense subgroup. The converse of this latter assertion has been proved in Corollary 5.2 in \cite{MS12}, and we obtain a result on the inverse problem:\\
\pagebreak[3]
\subparagraph{}\textbf{Theorem (see Theorem \ref{3theo3} below).}
\textit{ ~$G$ is the global parameterized differential Galois group of some equation having coefficients in~$k(z)$ if and only if~$G$ contains a finitely generated Kolchin-dense subgroup. \\}
\subparagraph{}
In the appendix, we prove the following result.\\ 
\pagebreak[3]
\subparagraph{} \textbf{Theorem (see Theorem \ref{3theo4} below).}
\textit{Let us consider~$\pz Y(z,t)=A(z,t)Y(z,t)$, with~${A(z,t)\in \mathrm{M}_{m}\Big( \hat{K}_{U}\Big)}$. Then, there exists a non empty polydisc~$U'\subset U$, such that we have a fundamental solution~$F(z,t)$ of the form:~$$F(z,t)=\hat{P}(z,t) z^{C(t)}e^{Q(z,t)},$$ 
where:
\begin{itemize}
 \item ~$\hat{P}(z,t)\in\mathrm{GL}_{m}\Big(\hat{K}_{U'}\Big)$.
 \item ~$C(t)\in \mathrm{M}_{m}(\mathcal{M}_{U'})$.
 \item ~$e^{Q(z,t)}=\mathrm{Diag}\big(e^{q_{i}(z,t)}\big)$, with~$q_{i}(z,t) \in z^{-1/\nu}\mathcal{M}_{U'} \big[z^{-1/\nu}\big]$, for some~$\nu \in \N^{*}$.\\
\end{itemize}}
Remark that contrary to Proposition \ref{3propo4}, the entries of the formal part are not ramified. On the other hand,~$ z^{C(t)}$ and~$e^{Q(z,t)}$ do not commute anymore. This theorem is not necessary for the proof of the main result of the paper, this is the reason why we give the proof in the appendix. However, this result is important since it permits one to determine the equivalence classes (see \cite{VdPS}, Page 7) of parameterized linear differential systems in coefficients in~$\hat{K}_{U}$. \\
\subparagraph{}\textbf{Acknowledgments.}
This paper was prepared during my thesis, supported by the region Ile de France. I want to thank my advisor, Lucia Di Vizio for her helpful comments and the interesting discussions we had during the preparation of this paper. I also want to thank the organizers of the seminars that had made it possible for me to present the results contained in this paper. I want to thank Jean-Pierre Ramis, Guy Casale, Reinhard Schäfke, Daniel Bertrand and Michael F. Singer for pointing some mistakes and inaccuracies in this paper. Michael F. Singer in particular suggested the contribution of this paper to Theorem \ref{3theo3}. I certainly thank Carlos E. Arreche and Claude Mitschi for the read-through. Lastly, I heartily thank the anonymous referees of the two successive submissions and Jacques Sauloy who spent a great lot of time and effort to help me make the present paper readable.

\pagebreak[3]
\section{Local analytic linear differential systems depending upon parameters.}\label{3sec1}

In~$\S \ref{3sec11}$, we define the field to which the entries of the fundamental solution, in the Hukuhara-Turrittin canonical form, will belong. In~$\S \ref{3sec12}$, we prove a parameterized version of the Hukuhara-Turrittin theorem. In~$\S \ref{3sec13}$, we briefly review the Stokes phenomenon in the unparameterized case. In~$\S \ref{3sec14}$,  we study the Stokes phenomenon in the parameterized case. 

\pagebreak[3]
\subsection{Definition of the fields.}\label{3sec11}
Let us consider a linear differential system of the form~${\pz Y(z)=A(z)Y(z)}$, where~$A(z)$ is an~$m \times m$ matrix whose entries belongs to~$\C\big[\big[z\big]\big]\big[z^{-1}\big]$. We know we can find a formal fundamental solution in the Hukuhara-Turrittin canonical form~$\hat{H}(z)z^{L}e^{Q(z)}$, where:
\begin{itemize}
 \item ~$\hat{H}(z)$ is a matrix of formal power series in~$z^{1/\nu}$ for some ~$\nu\in \N^{*}$.
 \item~$L \in \mathrm{M}_{m}(\C)$. 
 \item~$Q(z)= \mathrm{Diag}(q_{i}(z))$, with~$q_{i}(z) \in z^{-1/\nu}\C \big[z^{-1/\nu}\big]$.
\item Moreover, we have~$z^{L}e^{Q(z)}=e^{Q(z)}z^{L}$.
\end{itemize}
Notice that this formulation is trivially equivalent to Theorem 3.1 in \cite{VdPS}. Let~$U$ be a non empty polydisc of~$\C^{n}$, let~$\hat{K}_{U}$ and~$\mathcal{M}_{U}$ defined in page \pageref{3p1}. We want to construct a field containing a fundamental set of solutions of 
$$ \partial_{z}Y(z,t)=A(z,t)Y(z,t),$$
 where~$A(z,t)\in \mathrm{M}_{m}\Big(\hat{K}_{U}\Big)$.  
 Let~$\dt=\{\partial_{t_{1}},\dots,\partial_{t_{n}}\}$ and let 
$$\textbf{E}_{U}=\displaystyle \bigcup_{\n' \in \N^{*}}  z^{\frac{-1}{\n'}} \mathcal{M}_{U} \Big[z^{\frac{-1}{\n'}}\Big].~$$
We define formally the~$(\pz,\dt)$-ring, i.e., a ring equipped with~$n+1$ derivations~$\pz,\partial_{t_{1}},\dots,\partial_{t_{n}}$ a priori not required to commute with each other,~$$R_{U}:=\hat{K}_{U}\left[\log, \left(z^{a(t)}\right)_{a(t) \in \mathcal{M}_{U}},
\Big(e(q(z,t))\Big)_{q(z,t) \in \textbf{E}_{U}} \right],$$
 with the following rules:
\begin{enumerate}
\item The symbols~$\log$,~$\Big(z^{a(t)}\Big)_{a(t) \in \mathcal{M}_{U}}$ and~$\Big(e(q(z,t))\Big)_{q(z,t) \in \textbf{E}_{U}}$ only satisfy the following relations:
$$\begin{array}{cclcclccl}
&&&z^{a(t)+b(t)}&=&z^{a(t)}z^{b(t)},&e(q_{1}(z,t)+q_{2}(z,t))&=&e(q_{1}(z,t))e(q_{2}(z,t)),\\

&&&z^{a}&=&z^{a}\in \hat{K}_{U} \hbox{ for }a\in \Z,&e(0)&=&1.
\end{array}~$$
\item The following rules of differentiation
$$\begin{array}{cclcclccl}
\pz \log&=&z^{-1},&\pz z^{a(t)}&=&\frac{a(t)}{z}z^{a(t)},&\pz e(q(z,t))&=&\pz (q(z,t)) e(q(z,t)),\\

\partial_{t_{i}} \log& =&0,&\partial_{t_{i}} z^{a(t)}&=&\partial_{t_{i}}(a(t))\log z^{a(t)},&\partial_{t_{i}}e(q(z,t))&=&\partial_{t_{i}} (q(z,t)) e(q(z,t)),
\end{array}~$$
equip the ring with a~$(\pz,\dt)$-differential structure, since these rules go to the quotient as can be readily checked.
\end{enumerate}
 The intuitive interpretations of these symbols are~$\log=\log(z)$,~$z^{a(t)}=e^{a(t)\log(z)}$ and~$e(q(z,t))=e^{q(z,t)}$. Let~$f(z,t)$ be one these latter functions. Then~$f(z,t)$ has a natural interpretation as an analytic function on~$ \widetilde{\C} \times U'$, where~$\widetilde{\C}$ is the Riemann surface of the logarithm and~$U'$ is some non empty polydisc contained in~$U$. We will use the analytic function instead of the symbol when we will consider asymptotic solutions (see~$\S  \ref{3sec13}$ and~$\S \ref{3sec14}$). For the time being, however, we see them only as symbols.\par 
Let~$\overline{\mathcal{M}_{U}}$ be the algebraic closure of~$\mathcal{M}_{U}$. In the same way as for~$R_{U}$, we construct the~$(\pz,\dt)$-ring 

$$\overline{R}_{U}:=\overline{\mathcal{M}_{U}}\big[\big[z\big]\big]\big[z^{-1}\big]\left[\log, \left(z^{a(t)}\right)_{a(t) \in \overline{\mathcal{M}_{U}}},\Big(e(q(z,t))\Big)_{q(z,t) \in \displaystyle \bigcup_{\n' \in \N^{*}}  z^{\frac{-1}{\n'}} \overline{\mathcal{M}_{U}}\left[z^{\frac{-1}{\n'}}\right]}\right].$$ 
 We can see (Proposition 3.22 in \cite{VdPS}) that this latter is an integral domain and its field of fractions has field of constants with respect to~$\pz$ equal to~$\overline{\mathcal{M}_{U}}$. Since~$R_{U}\subset \overline{R}_{U}$,~$R_{U}$ is also an integral domain. Therefore, we may consider the~$(\pz,\dt)$-fields:
~$$K_{F,U}=\mathcal{M}_{U}\left(\log, \left(z^{a(t)}\right)_{a(t) \in \mathcal{M}_{U} }\right),$$
~$$\hat{K}_{F,U}=\hat{K}_{U}\left(\log, \left(z^{a(t)}\right)_{a(t) \in \mathcal{M}_{U}}\right),$$
and
$$\widehat{\textbf{K}_{U}}=\hat{K}_{U}\left(\log, \left(z^{a(t)}\right)_{a(t) \in \mathcal{M}_{U}}, \Big(e(q(z,t))\Big)_{q(z,t) \in \textbf{E}_{U}}\right).$$
In the definition of the fields~$K_{F,U}$ and~$\hat{K}_{F,U}$, the subscript~$F$ stands for Fuchsian. Since~$\widehat{\textbf{K}_{U}}$ is contained in the field of fractions of~$\overline{R}_{U}$, it has field of constants with respect to~$\pz$ equal to~$\overline{\mathcal{M}_{U}}\cap \widehat{\textbf{K}_{U}}=\mathcal{M}_{U}$.\par 
We have defined~$(\pz,\dt)$-fields where all the derivations commute with each other. We have the following inclusions of ~$(\pz,\dt)$-fields:
$$\begin{array}{ccccccc}
&&K_{F,U}&&&&\\
&\nearrow&&\searrow&&&\\
\mathcal{M}_{U}&\rightarrow&\hat{K}_{U}&\rightarrow&\hat{K}_{F,U}&\rightarrow&\widehat{\textbf{K}_{U}}. 
\end{array}$$
\pagebreak[3]
\begin{rem}\label{3rem4}
Any algebraic function over~$\mathcal{M}_{U}$ can be seen as an element of~$\mathcal{M}_{U'}$, for some non-empty~${U'\subset U}$. Therefore, a finite extension of~$\mathcal{M}_{U}$ can be embedded in~$\mathcal{M}_{U'}$ for a convenient choice of~$U'\subset U$. We will use this fact in the rest of the paper.
\end{rem}
\pagebreak[3]
\begin{lem}\label{3lem8}
Let~$ U$ be a non empty polydisc of~$\C^{n}$ and let~$L(t) \in \mathrm{M}_{m}\Big(\overline{\mathcal{M}_{U}}\Big)$, where~$\overline{\mathcal{M}_{U}}$ is the algebraic closure of~$\mathcal{M}_{U}$. There exist a non empty polydisc~$U'\subset U$, and~${z^{L(t)}\in \mathrm{GL}_{m}(K_{F,U'})}$ satisfying~$$\pz z^{L(t)}=\frac{L(t)}{z}z^{L(t)}=z^{L(t)}\frac{L(t)}{z}.$$
\end{lem}

\begin{proof}
  Let~$L(t)=P(t)(D(t)+N(t))P^{-1}(t)$, with~$D(t)=\mathrm{Diag}(d_{i}(t))$,~$d_{i}(t)\in \overline{\mathcal{M}_{U}}$,~$N(t)$ nilpotent,~${D(t)N(t)=N(t)D(t)}$ and~$P(t) \in \mathrm{GL}_{m}\Big(\overline{\mathcal{M}_{U}}\Big)$ be the Jordan decomposition of~$L(t)$.\par 
Due to Remark \ref{3rem4}, there exists a non empty polydisc~$U'\subset U$, such that~$d_{i}(t)\in \mathcal{M}_{U'}$ and~${P(t) \in \mathrm{GL}_{m}(\mathcal{M}_{U'})}$. We may restrict~$U'$ and assume that~$N(t)$ does not depend upon~$t$ in~$U'$. Let us write~$N:=N(t)$.
Then, the matrix~$z^{L(t)}=P(t)\mathrm{Diag}(z^{d_{i}(t)})e^{N\log}P^{-1}(t)$ belongs to~$\mathrm{GL}_{m}(K_{F,U'})$ and~$z^{L(t)}$ satisfies~$$\pz z^{L(t)}=\frac{L(t)}{z}z^{L(t)}=z^{L(t)}\frac{L(t)}{z}.$$
\end{proof}
Let~$a(t)\in \mathcal{M}_{U}$ and let~$(a(t))\in \mathrm{M}_{1}(\mathcal{M}_{U})$ be the corresponding matrix. Then, we have~$ z^{a(t)}=z^{(a(t))}$.
\pagebreak[3]
\subsection{The Hukuhara-Turrittin theorem in the parameterized case.}\label{3sec12}

The goal of this subsection is to give the parameterized version of the Hukuhara-Turrittin theorem. In the appendix, we prove a slightly different result, which is not needed in the paper. See Theorem \ref{3theo4}.
\pagebreak[3]
\begin{propo}\label{3propo4}
Let~$ U$ be a non empty polydisc of~$\C^{n}$ and consider ~$$\partial_{z}Y(z,t)=A(z,t)Y(z,t),$$
 with~$A(z,t)\in \mathrm{M}_{m}\Big(\hat{K}_{U}\Big)$. There exists a non empty polydisc~$U'\subset U$ such that we have a fundamental solution~$F(z,t)\in \mathrm{GL}_{m}\Big(\widehat{\textbf{K}_{U'}}\Big)$ of the form
$$F(z,t)=\hat{H}(z,t) z^{L(t)}e\Big(Q(z,t)\Big),$$ 
where:
\begin{itemize}
 \item ~$\hat{H}(z,t)\in \mathrm{GL}_{m}\Big(\hat{K}_{U'}\left[z^{1/\nu}\right]\Big)$, for some~$\nu\in \N^{*}$.
 \item ~$L(t) \in \mathrm{M}_{m}(\mathcal{M}_{U'})$.
 \item ~$e\Big(Q(z,t)\Big)=\mathrm{Diag}\Big(e(q_{i}(z,t))\Big)$, with~$q_{i}(z,t) \in \textbf{E}_{U'}$.
 \item  Moreover, we have~$e\Big(Q(z,t)\Big)z^{L(t)}=z^{L(t)}e\Big(Q(z,t)\Big)$.
\end{itemize}
Furthermore, if~$A(z,t)\in \mathrm{M}_{m}\Big(\mathcal{O}_{U}(\{z\})\Big)$, there exists a non empty polydisc~$U''\subset U'$ such that we may assume that the~$z$-coefficients of~$\hat{H}(z,t)$ are all analytic on~$U''$.
\end{propo}

\pagebreak[3]
\begin{rem}\label{3rem5}
Remark that we have no uniqueness of the fundamental solution written in the same way as above, since for all~$\kappa\in \Z$,~$z^{\k}\hat{H}(z,t) z^{L(t)-\k}e^{Q(z,t)}$ is also a fundamental solution. However, because of the construction of~$\widehat{\textbf{K}_{U'}}$, we obtain that if for~${i\in \{1,2\}}$,~$\hat{H}_{i}(z,t)z^{L_{i}(t)}e\Big(Q_{i}(z,t)\Big)$ is a fundamental solution of~${\partial_{z}Y(z,t)=A(z,t)Y(z,t)}$ written in the same way as above, then, up to a permutation,~$Q_{1}$ and~$Q_{2}$ have the same entries.
\end{rem}

\pagebreak[3]
\begin{ex}[\cite{Sch}, Introduction]
If we consider~$$z^{2} \pz Y(z,t)=\begin{pmatrix}
t & 1 \\ 
z & 0
\end{pmatrix}  Y(z,t),$$ we get the solution

\begin{equation}\label{3eq3}
\begin{pmatrix}

\begin{pmatrix}
1 & 1 \\ 
0 & -t
\end{pmatrix}+O(z)\end{pmatrix}\begin{pmatrix}
z^{\frac{1}{t}}e^{\frac{-t}{z}} & 0 \\ 
0 & z^{\frac{-1}{t}}
\end{pmatrix},
\end{equation}
 for~$t\neq 0$ and the solution
$$\begin{pmatrix}
1 & 1 \\ 
z^{1/2} & -z^{1/2}
\end{pmatrix}\begin{pmatrix} \begin{pmatrix}
1& 0 \\ 
0 &1
\end{pmatrix}+O(z^{1/2})\end{pmatrix} \begin{pmatrix}
z^{\frac{1}{4}}e^{-z^{-1/2}} & 0 \\ 
0 & z^{\frac{1}{4}}e^{z^{-1/2}}
\end{pmatrix} ,$$
 for~$t=0$. The latter is not the specialization of (\ref{3eq3}) at~$t=0$. The problem is that the level of the unparameterized system (see~$\S \ref{3sec13}$ for the definition) at~$t=0$ is~$1$ and the level of the unparameterized system for~$t\neq 0$ is~$\frac{1}{2}$. This example shows that we cannot get a solution in the parameterized Hukuhara-Turrittin form, that remains valid for all values of the parameter~$t$. This is the reason why we have to restrict the subset of the parameter-space.
\end{ex} 
\pagebreak[3]
\begin{rem}\label{3rem2}
Similar results to Proposition \ref{3propo4} have been proved in Theorem 4.2 of \cite{Sch}. We now explain the result of Schäfke. Let~$U$ be an open connected subset of~$\C^{n}$ that contains~$0$ and let~${A(z,t)=\displaystyle\sum_{l=s}^{\infty}A_{l}(t)}$, with~$s\in \Z$, and~$A_{l}(t)$ analytic in~$U$. In particular, ~${A(z,t)\in \mathrm{M}_{m}\Big(\hat{K}_{U}\Big)}$. Let us consider~${\pz Y(z,t)=A(z,t)Y(z,t)}$ and assume that for all~${t\in U}$, there exists a solution~$\hat{H}_{t}(z) z^{L_{t}}e\Big(Q(z,t)\Big)$, given in the classical Hukuhara-Turrittin canonical form such that:
\begin{itemize}
\item The~$z$-coefficients of the~$q_{i}(z,t)$ are analytic functions in~$t\in U$.
\item The degree in~$z^{-1}$ of~$q_{i}(z,t)-q_{j}(z,t)$ is independent of~$t$ in~$U$.
\item If~$q_{i}(z,t)\not \equiv q_{j}(z,t)$, then~$ q_{i}(z,0)\neq q_{0}(z,0)$.
\end{itemize}
Under these assumptions, Schäfke concludes that, there exists an open neighborhood~${U'\subset U}$ of~$0$ in the~$t$-plane such that there exists a solution~${\hat{H}(z,t) z^{L(t)}e(Q(z,t))\in \mathrm{GL_{m}}(\widehat{\textbf{K}_{U'}})}$ with~$\hat{H}(z,t)=\sum_{l=0}^{\infty}\hat{H}_{l}(t)$ and~$t\mapsto \hat{H}_{l}(t),L(t)$ are analytic. Notice that Schäfke gives a necessary and sufficient condition, that can be algorithmically checked, for well behaved exponential part. See \cite{Sch}, Theorem~5.2. Using Schäfke's theorem, we can deduce Proposition \ref{3propo4} only in the particular case where~$A(z,t)$ has entries with~$z$-coefficients analytic in~$U$. Note that \cite{Sch} does not allow us to deduce the general case. See also \cite{BV},~$\S$ 10, Theorem 1, for another result of this nature.
\end{rem}

\begin{proof}[Proof of Proposition \ref{3propo4}]
Let~${K=C\big[\big[z\big]\big]\big[z^{-1}\big]}$, where~$C$ is an algebraically closed field of characteristic~$0$ equipped with a derivation~$\pz$ that acts trivially on~$C$ and with~${\pz(z)=1}$. The Hukuhara-Turrittin theorem (see Theorem 3.1 in \cite{VdPS}) is valid for  linear differential system with entries in~$K$. We apply it with~$C=\overline{\mathcal{M}_{U}}$, the algebraic closure of~$\mathcal{M}_{U}$. \par
Let us consider the matrices~$L(t) \in \mathrm{M}_{m}\Big(\overline{\mathcal{M}_{U}}\Big)$ and~$Q(z,t)=\mathrm{Diag}\Big(q_{i}(z,t)\Big)$, with~${q_{i}(z,t) \in z^{-1/\nu}\overline{\mathcal{M}_{U}}\left[z^{-1/\nu}\right]}$ for some~$\nu\in \N$. Because of Remark \ref{3rem4} and Lemma \ref{3lem8}, there exists a non empty polydisc~$ U'\subset U$, such that we may define~$z^{L(t)}\in \mathrm{GL}_{m}(K_{F,U'})$ satisfying~$\pz z^{L(t)}=\frac{L(t)}{z}z^{L(t)}=z^{L(t)}\frac{L(t)}{z}$,~$L(t) \in \mathrm{M}_{m}(\mathcal{M}_{U'})$ and~$q_{i}(z,t) \in \textbf{E}_{U'}$. Hence, there exists a non empty polydisc~$ U'\subset U$ such that the Hukuhara-Turrittin theorem gives a fundamental solution
$$F'(z,t)=\hat{H}'(z,t) z^{L(t)}e\Big(Q(z,t)\Big),$$ 
where:
\begin{itemize}
 \item ~$\hat{H}'(z,t)\in \mathrm{GL}_{m}\Big(\overline{\mathcal{M}_{U'}}\left[\left[z^{1/\nu}\right]\right]\left[z^{-1/\nu}\right]\Big)$, for some~$\nu\in \N$.
 \item ~$L(t) \in \mathrm{M}_{m}(\mathcal{M}_{U'})$.
 \item ~$e\Big(Q(z,t)\Big)=\mathrm{Diag}\Big(e(q_{i}(z,t))\Big)$, with~$q_{i}(z,t) \in \textbf{E}_{U'}$.
 \item  Moreover, we have~$e\Big(Q(z,t)\Big)z^{L(t)}=z^{L(t)}e\Big(Q(z,t)\Big)$.
\end{itemize}
 Let us prove now that we may find~$\hat{H}(z,t)\in \mathrm{GL}_{m}\Big(\hat{K}_{U'}\left[z^{1/\nu}\right]\Big)$, such that~${F(z,t)=\hat{H}(z,t) z^{L(t)}e\Big(Q(z,t)\Big)}$ is a fundamental solution. 
The matrix~$$F'(z,t)=\hat{H}'(z,t)z^{L(t)}e\Big(Q(z,t)\Big)$$ satisfies the parameterized linear differential equation
$$ \pz F'(z,t)=A(z,t)F'(z,t),$$
and the matrix~$z^{L(t)}e\Big(Q(z,t)\Big)$ satisfies parameterized linear differential equation:
$$ \pz z^{L(t)}e\Big(Q(z,t)\Big)=\left(z^{-1}L(t)+\pz Q(z,t)\right)z^{L(t)}e\Big(Q(z,t)\Big)=z^{L(t)}e\Big(Q(z,t)\Big)\left(z^{-1}L(t)+\pz Q(z,t)\right).$$
Hence,
$$\pz \hat{H}'(z,t)=A(z,t)\hat{H}'(z,t)-\hat{H}'(z,t)\left(z^{-1}L(t)+\pz Q(z,t)\right).$$
We write~$\hat{H}'(z,t)$ as a column vector~$\widetilde{H}'(z,t)$ of size~$m^{2}$. Let~$C(z,t)\in \mathrm{M}_{m^{2}}\Big(\hat{K}_{U'}\left[z^{1/\nu}\right]\Big)$, with~$\nu \in \N^{*}$ such that ~$\widetilde{H}'(z,t)$ satisfies  the parameterized linear differential system:
$$\pz \widetilde{H}'(z,t)=C(z,t)\widetilde{H}'(z,t).$$
Let us write~$\widetilde{H}'(z,t)=\displaystyle \sum_{i \geq N} \widetilde{H}'_{i}(t)z^{i/\nu}$ and~$C(z,t)=\displaystyle \sum_{i \geq M} C_{i}(t)z^{i/\nu}$, where~$M,N\in\Z$. Then, by identifying the coefficients of the ~$z^{i/\nu}$-terms of the power series in the equation~$\pz \widetilde{H}'(z,t)=C(z,t)\widetilde{H}'(z,t),$ we find that:
$$
\left(\frac{i}{\nu}+1\right)\widetilde{H}'_{i+\nu}(t)=\displaystyle \sum_{ 
l=N}^{i-M} C_{i-l}(t)\widetilde{H}'_{l}(t).
$$
We recall that because of the definition of~$\hat{K}_{U'}\left[z^{1/\nu}\right]$, every~$C_{i}(t)$ belongs to~$\mathrm{M}_{m}\left(\mathcal{M}_{U'}\right)$. The fact that there exists a fundamental solution~$\hat{H}(z,t) z^{L(t)}e\Big(Q(z,t)\Big)$, with 
${\hat{H}(z,t)\in \mathrm{GL}_{m}\Big(\hat{K}_{U'}\left[z^{1/\nu}\right]\Big)}$ is now clear.\par 
Assume now that ~$A(z,t)\in \mathrm{M}_{m}\Big(\mathcal{O}_{U}(\{z\})\Big)$. Let~$U''$ be a non empty polydisc with~${U''\subset U'}$ such that for~$z\neq 0$ fixed, the entries of the~$z$-coefficients of~$z^{-1}L(t)+\pz Q(z,t)$ are analytic on~$U''$. Then, the entries of the~$z$-coefficients of~$C(z,t)$ are all analytic on~$U''$. Hence, we may assume that the entries of the~$z$-coefficients of~$\hat{H}(z,t)$ are all analytic on~$U''$.
\end{proof}
\pagebreak[3]
\begin{rem}\label{3rem3}
If we take a smaller non empty polydisc~$ U$, we may assume that if we consider~${\partial_{z}Y(z,t)=A(z,t)Y(z,t)}$,  with~$A(z,t)\in \mathrm{M}_{m}\Big(\mathcal{O}_{U}(\{z\})\Big)$, then the fundamental solution of Proposition \ref{3propo4} belongs to~$\mathrm{GL}_{m}\Big(\widehat{\textbf{K}_{U}}\Big)$, and the entries of the~$z$-coefficients of~$\hat{H}(z,t)$ are all analytic on~$U$.
\end{rem}

\pagebreak[3]
\subsection{Review of the Stokes phenomenon in the unparameterized case.}\label{3sec13}
In this subsection we will briefly review the Stokes phenomenon in the unparameterized case. See \cite{CR,E,E81,LR90,LR94,LR95,LRR,M91,M95,MR,R80,R85,Ras,Re,RS,S09,W} and in particular Chapter 8 of \cite{VdPS} for more details. We will generalize some results concerning the summation of divergent series in the parameterized case in~$\S \ref{3sec14}$. First we treat the example of the Euler equation:
$$ z^{2}\pz Y(z)+Y(z)=z,$$
which admits as a solution the formal series:
$\hat{f}(z)=\displaystyle\sum_{n=0}^{\infty} (-1)^{n}n!z^{n+1}$. Classical methods of differential equations give another solution:
$$ f(z)=\int_{0}^{z} e^{1/z}e^{-1/t}\frac{dt}{t}=\int_{0}^{\infty} \frac{1}{1+u}e^{-u/z}du,$$
where~$1/t-1/z=u/z$. The solution~$\hat{f}(z)$ is divergent and the solution~$f(z)$ can be extended to an analytic function on the sector:
$$V=\left\{z \in \widetilde{\C}\Big|  \arg(z)\in \left]\frac{-3\pi}{2},\frac{+3\pi}{2}\right[\right\} .$$
 On this sector,~$f(z)$ is~$1$-Gevrey asymptotic to~$\hat{f}(z)$: for every closed subsector~$W$ of~$V$, there exist~${A_{W} \in \R}$ and ~$\e >0$ such that for all~$N$ and all~$z\in W$ with~$|z|<\e$,
$$\left|f(z)-\sum_{n= 0}^{N-1} (-1)^{n}n!z^{n+1}\right|\leq (A_{W})^{N+1}(N+1)! |z|^{N+1}.~$$    
We can also consider~$f(e^{2i\pi}z)$, which is an asymptotic solution on the sector:
$$V'= \left\{z\in \widetilde{\C} \Big|\arg(z)\in\left]\frac{\pi}{2},\frac{7\pi}{2}\right[ \right\} .$$
The two asymptotic solutions do not glue to a single asymptotic solution on~$V\cup V'$. In fact, the residue theorem implies that the difference in~$V\cap V'$ of the two asymptotic solutions is:
$$2i\pi e^{1/z}.$$
 The fact that various asymptotic solutions do not glue to a single analytic solution is called the Stokes phenomenon. 
\par 
More generally, let us consider a linear differential equation~$\pz Y(z)=A(z)Y(z)$ such that the entries of~$A(z)$ are germs of meromorphic functions in a neighborhood of~$0$. Let~$\hat{H}(z)z^{L}e\Big(Q(z)\Big)$, with~${Q(z)=\mathrm{Diag}\Big(q_{i}(z)\Big)}$, be a fundamental solution in the Hukuhara-Turrittin canonical form. Since for all~$k$ that belongs to~$\N$,~${\hat{H}(z)z^{L}e\Big(Q(z)\Big)=\hat{H}(z)\mathrm{Diag}(z^{k}) z^{L-k\mathrm{Id}}e\Big(Q(z)\Big)}$, we may assume that~$\hat{H}(z)$ has no pole at~$z=0$. The levels of~$\pz Y(z)=A(z)Y(z)$ are the degrees in~$z^{-1}$ of the~$q_{i}(z)-q_{j}(z)$ (the levels are positive rational numbers and are well defined because of Remark \ref{3rem5}). Consider~${q(z)=q_{k}z^{-k/\nu}+\dots+q_{1}z^{-1/\nu}\in z^{-1/\nu}\C\left[z^{-1/\nu}\right]}$ with~$\nu\in\N$. The real number~$d$ is called singular for~$q(z)$ if~$q_{k}e^{-idk/\nu}$ is a positive real number. These correspond to the arguments~$d$ such that~$r\mapsto e^{q\left(re^{id}\right)}$ increases fastest as~$r$ tends to~$0^{+}$.
The singular directions of~$\pz Y(z)=A(z)Y(z)$ (we will write singular directions when no confusion is likely to arise) are the real numbers that are singular for one of the ~$q_{i}(z)-q_{j}(z)$, with~$i\neq j$. Notice that the set of singular directions is finite modulo~$2\pi\nu$ for some~$\nu\in \N$. Let~$k_{1}<\dots<k_{r}$ be the levels of the linear differential equation. There exists a decomposition~$\hat{H}(z)=\hat{H}_{k_{1}}(z)+\dots+\hat{H}_{k_{r}}(z)$, such that for~$d$ not a singular direction, there exists an unique~$r$-tuple of matrices~$\Big(H_{k_{1}}^{d}(z),\dots,H_{k_{r}}^{d}(z)\Big)$, such that~$H_{k_{i}}^{d}(z)$ is analytic on the sector 
$$V_{d}=\left\{ z\in \widetilde{\C} \Big| \arg(z)\in \left]d-\frac{\pi}{2k_{i}},d+\frac{\pi}{2k_{i}}\right[\right\},$$
 and is~$k_{i}$-Gevrey asymptotic to~$\hat{H}_{k_{i}}(z)=\sum_{n\in \N} \hat{H}_{n,k_{i}}z^{n}$ on~$V_{d}$:
for every closed subsector~$W$ of~$V_{d}$, there exist~$A_{W} \in \R$ and ~$\e >0$ such that for all~$N$ and all~$z\in W$ with~$|z|<\e$,
$$\left|H_{k_{i}}^{d}(z)-\sum_{n=0}^{N-1} \hat{H}_{n,k_{i}}z^{n}\right|\leq (A_{W})^{N}\G\left(1+\frac{N}{k_{i}}\right) |z|^{N},~$$ 
where~$\G$ denotes the Gamma function.  Until the end of the paper, we will denote a fixed determination of the complex logarithm by~$\log (z)$. Furthermore, the matrix 
\begin{equation}\label{3eq2}
\Big(H_{k_{1}}^{d}(z)+\dots+H_{k_{r}}^{d}(z)\Big)e^{L\log(z)}e^{Q(z)}=H^{d}(z)e^{L\log(z)}e^{Q(z)},
\end{equation} which  is analytic on the sector~$\left\{ z\in \widetilde{\C} \Big| \arg(z)\in \left]d-\frac{\pi}{2k_{r}},d+\frac{\pi}{2k_{r}}\right[\right\}$, is a solution of~$\pz Y(z)=A(z)Y(z)$. As a matter of fact,~$H_{k_{i}}^{d}(z)$ is~$k_{i}$-Gevrey asymptotic to~$\hat{H}_{k_{i}}(z)$ on the larger sector:~$$\left\{ z\in \widetilde{\C} \Big| \arg(z)\in \left]d_{l}-\frac{\pi}{2k_{i}},d_{l+1}+\frac{\pi}{2k_{i}}\right[\right\},$$ where ~$d_{l},d_{l+1}$ are two singular directions and such that~$]d_{l},d_{l+1}[$ contains no singular directions. Therefore, we can construct an analytic solution on the sector
$\left\{ z\in \widetilde{\C} \Big| \arg(z)\in \left]d_{l}-\frac{\pi}{2k_{r}},d_{l+1}+\frac{\pi}{2k_{r}}\right[\right\}.$
 Let~$d\in \R$, and let:~$$d-\frac{\pi}{2k_{r}}<d^{-}<d<d^{+}<d+\frac{\pi}{2k_{r}},$$ such that there are no singular directions in~$[d^{-},d[\bigcup ]d,d^{+}]$.
We get two matrices~$H^{d^{+}}(z)e^{L\log(z)}e^{Q(z)}$ and~$H^{d^{-}}(z)e^{L\log(z)}e^{Q(z)}$ which are germs of analytic solutions on the sectors~$$\left\{z \in \widetilde{\C}\Big|\arg(z)\in \left]d^{-}-\frac{\pi}{2k_{r}},d+\frac{\pi}{2k_{r}}\right[\right\} \hbox{ and }\left\{z \in \widetilde{\C}\Big|\arg(z)\in \left]d-\frac{\pi}{2k_{r}},d^{+}+\frac{\pi}{2k_{r}}\right[\right\} .$$ The two matrices are in particular germs of solutions of ~$\pz Y(z)=A(z)Y(z)$ on the sector~$$\left\{z \in \widetilde{\C}\Big|\arg(z)\in \left]d-\frac{\pi}{2k_{r}},d+\frac{\pi}{2k_{r}}\right[\right\} .$$ A computation shows that there exists a matrix~$St_{d}\in \mathrm{GL}_{m}(\C)$, which we call the Stokes matrix in the direction~$d$,  such that: 
$$
H^{d^{+}}(z)e^{L\log(z)}e^{Q(z)}= H^{d^{-}}(z)e^{L\log(z)}e^{Q(z)}St_{d}.
$$
\pagebreak[3]
\begin{propo}\label{3propo10}
The following  statements are equivalent.
\begin{enumerate}
\item The entries of~$\hat{H}(z)$ converge.
\item~$St_{d}=\mathrm{Id}$ for all~$d\in \R$.
\item~$St_{d}=\mathrm{Id}$ for all singular directions.
\end{enumerate}
\end{propo}

\begin{proof}
From what is preceding, we deduce that if~$d$ is not a singular direction, then~${St_{d}=\mathrm{Id}}$. Therefore, the statements 2 and 3 are equivalent. If the entries of~$\hat{H}(z)$ converge, then, since~$\hat{H}(z)$ is Gevrey asymptotic to itself on every sector of $\widetilde{\C}$, for all~$d\in \R$, ~$H^{d}(z)=\hat{H}(z)$ and (2) holds. Assume now that~$St_{d}=\mathrm{Id}$ for all singular directions. From the proof of \cite{VdPS}, Theorem 8.10, we obtain that the entries of~$\hat{H}(z)$ converge.
\end{proof}

We can compute the asymptotic solutions using the Laplace and the Borel transformations. See Chapters 2 and 3 of \cite{B} for more details.
\pagebreak[3]
\begin{defi}\label{3defi3}
\begin{trivlist}
\item (1) Let~$k\in \Q$. The formal Borel transform~$\hat{\mathcal{B}}_{k}$ is the map that transforms the formal power series~$\sum a_{n}z^{n}$ into the formal power series:
$$\hat{\mathcal{B}}_{k}\left(\sum a_{n}z^{n}\right)=\sum \frac{a_{n}}{\G(1+\frac{n}{k})}z^{n}.$$
\item (2) Let~$d\in \R$,~$k\in \Q$,~$\e>0$ and let~$f$ analytic on the sector~$\left\{ z\in \widetilde{\C} \Big| \arg(z)\in ]d-\e,d+\e[ \right\}$. We assume that there exist~$A,B>0$ such that for~$\arg(z)=d$,
$$|f(z)|\leq Ae^{B|z|^{k}}.~$$ 
Then, the following integral is the germ of an analytic function on~$\left\{ z\in \widetilde{\C} \Big|\arg(z)\in \left]d-\frac{\pi}{2k},d+\frac{\pi}{2k}\right[\right\}$ (see \cite{B}, page 13 for a proof), and is called the Laplace transform of order~$k$ in the direction~$d$ of~$f$:
$$ \mathcal{L}_{k,d}(f)(z)=\displaystyle \int_{0}^{\infty e^{id}}f(u)e^{-\left(\frac{u}{z}\right)^{k}}d\left(\left(\frac{u}{z}\right)^{k}\right).$$
\end{trivlist}
\end{defi}
 
For a proof of the following proposition, see Section 7.2 of \cite{B}. 
\pagebreak[3]
\begin{propo}\label{3propo1}
Let~$k_{1}<\dots<k_{r}$ be the levels of~$\pz Y(z)=A(z)Y(z)$ and set~${k_{r+1}=+\infty}$. Suppose that~$d\in \R$ is not a singular direction, and let~$\hat{h}(z)$ be an entry of~$\hat{H}(z)$. Let~$(\kappa_{1},\dots,\kappa_{r})$ defined as:
$$\kappa_{i}^{-1}=k_{i}^{-1}-k_{i+1}^{-1}.$$ 
The series~$\hat{\mathcal{B}}_{\kappa_{r}}\circ\dots\circ\hat{\mathcal{B}}_{\kappa_{1}}(\hat{h})$ converges and there exist~$\e_{1},A_{1},B_{1}>0$ such that it has an analytic continuation~$h_{1}$ on the sector~$\left\{ z\in \widetilde{\C} \Big| \arg(z)\in ]d-\e_{1},d+\e_{1}[ \right\}$, and in this sector,
$$|h_{1}(z)|\leq A_{1}e^{B_{1}|z|^{\kappa_{1}}}.~$$ 
Moreover, for~$j=2$ (resp.~$j=3,\dots,j=r$), there exist~$\e_{j},A_{j},B_{j}>0$ such that the function~${h_{j+1}=\mathcal{L}_{\kappa_{j},d}(h_{j})}$ is analytic on the sector~$\left\{ z\in \widetilde{\C} \Big| \arg(z)\in ]d-\e_{j},d+\e_{j}[ \right\}$ and on this sector
$$|h_{j}(z)|\leq A_{j}e^{B_{j}|z|^{\kappa_{j}}}.~$$ 
Therefore, we may apply~$\mathcal{L}_{\kappa_{r},d}\circ\dots\circ\mathcal{L}_{\kappa_{1},d}\circ\hat{\mathcal{B}}_{\kappa_{r}}
\circ\dots\circ\hat{\mathcal{B}}_{\kappa_{1}}$ to every entry of~$\hat{H}(z)$.  We have the following equality:
$$ H^{d}(z)=\mathcal{L}_{\kappa_{r},d}\circ\dots\circ\mathcal{L}_{\kappa_{1},d}\circ
\hat{\mathcal{B}}_{\kappa_{r}}\circ\dots\circ\hat{\mathcal{B}}_{\kappa_{1}}\Big(\hat{H}\Big).$$
\end{propo}

\pagebreak[3]
\subsection{Stokes phenomenon in the parameterized case.}\label{3sec14}

Let~$\pz Y(z,t)=A(z,t)Y(z,t)$, with~${A(z,t)\in \mathrm{M}_{m}( \mathcal{O}_{U}\Big(\{z\})\Big)}$ (see page \pageref{3p2}), where~$ U$ is a non empty polydisc of~$\C^{n}$, and consider~${F(z,t)=\hat{H}(z,t) z^{L(t)}e\Big(Q(z,t)\Big)}$, with~$Q(z,t)=\mathrm{Diag}\Big(q_{i}(z,t)\Big)$, the fundamental solution of Proposition \ref{3propo4}. Since for all~$k\in \N$,~$F(z,t)$ is equal to~$\hat{H}(z,t)\mathrm{Diag}(z^{k}) z^{L(t)-k\mathrm{Id}}e\Big(Q(z,t)\Big)$, we may assume that~$\hat{H}(z,t)$ has no pole at~$z=0$. We define the levels of the system~$\pz Y(z,t)=A(z,t)Y(z,t)$ as the levels of the specialized system. The levels may depend upon~$t$, but they are invariant on the complementary of a closed set with empty interior. We want to extend the definition of the singular directions to the parameterized case. Consider~$q(z,t)=q_{k}(t)z^{-k/\nu}+\dots+q_{1}(t)z^{-1/\nu}\in \textbf{E}_{U}$. A continuous function~$d:U\rightarrow \R$ is called singular for~$q(z,t)$ if~$$\forall t\in U, \quad q_{k}(t)e^{-id(t)k/\nu}\in \R^{\geq 0}.$$
In general, if~$d(t)$ is a singular direction for~$q(z,t)$, the positive number~$q_{k}(t)e^{-id(t)k/\nu}$ depends on~$t$.
 The singular directions of~$\pz Y(z,t)=A(z,t)Y(z,t)$ (we will write singular directions when no confusion is likely to arise) are the directions that are singular for one of the ~$q_{i}(z,t)-q_{j}(z,t)$, with~$i\neq j$.
\pagebreak[3]
\begin{rem}
\begin{trivlist}
\item (1)
It may happen that for some~$t_{0}\in U$, the singular directions of~${\pz Y(z,t)=A(z,t)Y(z,t)}$ evaluated at~$t_{0}$ are not equal to the singular directions of the specialized system~${\pz Y(z,t_{0})=A(z,t_{0})Y(z,t_{0})}$. Take for example~${n=1}$,~${U=\C}$,~${t_{0}=0}$ and~$A(z,t)=\mathrm{Diag}\Big(-2tz^{-3}-z^{-2},2tz^{-3}+z^{-2}\Big)$. The two exponentials are~${e(q_{1}(z,t))=e(tz^{-2}+z^{-1})}$ and~$e(q_{2}(z,t))=e(-tz^{-2}-z^{-1})$.
However, there exists~$V\subset U$, a closed set with empty interior, such that for all~$t_{0}$ in~$U\setminus V$, the singular directions of~$\pz Y(z,t)=A(z,t)Y(z,t)$ evaluated at~$t_{0}$ are equal to the singular directions of the specialized system~$\pz Y(z,t_{0})=A(z,t_{0})Y(z,t_{0})$. \par 
\item (2)
Unfortunately, two different singular directions may be equal on a subset of~$U$. For example, for~${n=1}$,~$U=\C^{*}$, and~$A(z,t)=\mathrm{Diag}\Big(z^{-2},tz^{-2},-tz^{-2}\Big)$ we find three exponentials:~$e^{-1/z},e^{t/z}$ and~$e^{-t/z}$. For~$t\in \R^{>0}$, the singular directions of~$(2t)z^{-1}$ are the same as the singular directions of~$ (t+1)z^{-1}$.
\end{trivlist}
\end{rem}
Let~$(d_{i}(t))_{i\in \N}$ be the singular directions, and 
$$
\mathcal{D}=\Big\{t\in U \Big|\exists j, j'\in \N, \hbox{ such that } d_{j}\not\equiv  d_{j'} \hbox{ and }d_{j}(t)=d_{j'}(t)\Big\}.$$
\pagebreak[3]
\begin{lem}\label{3lem9}
~$\mathcal{D}$ is a closed subset of~$U$ with empty interior.
\end{lem}

\begin{proof}
 Assume that there exist a non empty polydisc~$ D\subset \mathcal{D}$, and two singular directions~$d_{j}(t),d_{j'}(t)$ such that~$d_{j}(t)=d_{j'}(t)$ on~$D$. Then, there exist a non empty polydisc~$D'\subset D$ and~$q(t),q'(t)\in \mathcal{M}_{D'}$ that do not vanish on~$D'$ such that~$q(t)/q'(t)$ has constant argument on~$D'$.
An analytic function with constant argument on a polydisc is constant. Hence, we deduce that~$d_{j}(t)=d_{j'}(t)$ on a polydisc, which implies that~$d_{j}(t)=d_{j'}(t)$ on~$U$. Since the set of singular directions is finite modulo~$2\pi \nu$ with~$\nu\in \N^{*}$,~$\mathcal{D}$ has empty interior.
\end{proof}

Thus, if we take a smaller non empty polydisc~$ U$, we may assume the following:
\begin{itemize}
\item~$\mathcal{D}=\varnothing$.
\item The levels of~$\pz Y(z,t)=A(z,t)Y(z,t)$ are independent of~$t$
\item For all~$t_{0}\in U$, the singular directions of~$\pz Y(z,t)=A(z,t)Y(z,t)$ evaluated at~$t_{0}$ are equal to the singular directions of the specialized system~${\pz Y(z,t_{0})=A(z,t_{0})Y(z,t_{0})}$.
\end{itemize}
 We still consider~$\pz Y(z,t)=A(z,t)Y(z,t)$ a parameterized linear differential system with~${A(z,t) \in \mathrm{M}_{m}\Big(\mathcal{O}_{U}(\{z\})\Big)}$ and ~$\hat{H}(z,t)z^{L(t)}e\Big(Q(z,t)\Big)\in \mathrm{GL}_{m}\left(\widehat{\textbf{K}_{U}}\right)$ the fundamental solution in the same form as in Proposition \ref{3propo4}. Let~$d(t)$ be a singular direction, and let~$k_{1}<\dots<k_{r}$ be the levels of~${\pz Y(z,t)=A(z,t)Y(z,t)}$. For~$t$ belonging to~$U$, we define the parameterized Stokes matrix~$St_{d(t)}$ (we will just call it the Stokes matrix when no confusion is likely to arise) as~$t\mapsto St_{d(t)}$, where~$St_{d(t)}$ is the Stokes matrix of the specialized system defined just before Proposition \ref{3propo10}.  
\pagebreak[3]
\begin{propo}\label{3propo3}
Let~$d(t)$ continuous in~$t$ such that for all~$t_{0}$ in~$U$,~$d(t_{0})$ is not a singular direction of the unparameterized linear differential equation~${\pz Y(z,t_{0})=A(z,t_{0})Y(z,t_{0})}$. We define~${t\mapsto H^{d(t)}(z,t)e^{L(t)\log(z)}e^{Q(z,t)}}$, as the solution (\ref{3eq2}), of the specialized system. Let~$d_{1}(t),d_{2}(t)$ be two singular directions such that for all~${t\in U}$, $d_{1}(t)<d(t)<d_{2}(t)$ and~$]d_{1}(t),d_{2}(t)[$ contains no singular directions. Then, there exists a map~${U\rightarrow \R^{>0}}$,${t\mapsto\e(t)}$, which is not necessarily continuous, such that~$H^{d(t)}(z,t)e^{L(t)\log(z)}e^{Q(z,t)}$ is meromorphic in~$(z,t)$ for
$$ (z,t)\in \left\{ z\in \widetilde{\C}\Big|\arg(z)\in \left]d_{1}(t)-\frac{\pi}{2k_{r}},d_{2}(t)+\frac{\pi}{2k_{r}} \right[ , \hbox{ and }0<|z|<\e(t)\right\} \times U.$$
\end{propo}

Notice that the existence of~$d(t)$ continuous in~$t$ such that for all~$t_{0}$ in~$U$,~$d(t_{0})$ is not a singular direction of the unparameterized linear differential equation~${\pz Y(z,t_{0})=A(z,t_{0})Y(z,t_{0})}$ is a direct consequence of the fact that~$\mathcal{D}=\varnothing$, and the fact that the singular directions are continuous in~$t$.

\begin{proof}
We recall that we have assumed that for all~$t_{0}\in U$, the singular directions of~${\pz Y(z,t)=A(z,t)Y(z,t)}$ evaluated at~$t_{0}$ are equal to the singular directions of the specialized system~${\pz Y(z,t_{0})=A(z,t_{0})Y(z,t_{0})}$. We have seen in~$\S \ref{3sec13}$, that for~$t$ fixed, the asymptotic solution is a germ of meromorphic function on the sector~$$\left\{z \in \widetilde{\C}\Big|  \arg(z)\in\left]d_{1}(t)-\frac{\pi}{2k_{r}},d_{2}(t)+\frac{\pi}{2k_{r}}\right[\right\}  .$$
We may replace~$d(t)$ by any function, possibly non continuous, such that for all~$t\in U$, ${d_{1}(t)<d(t)<d_{2}(t)}$. Since the singular directions are continuous in~$t$, we may assume that~$d(t)$ is locally constant. Since for~$z\neq 0$,~${t\mapsto e^{L(t)\log(z)}e^{Q(z,t)}}\in \mathcal{M}_{U}$,
 this is now a consequence of Proposition \ref{3propo1} and Lemma \ref{3lem3} below.
\end{proof}
\pagebreak[3]
\begin{lem}\label{3lem3}
We keep the same notation as in  Definition \ref{3defi3} and Proposition \ref{3propo1}. Let~$\hat{h}(z,t)$ be one of the entries of~$\hat{H}(z,t)$. Let~$ V \subset U$ be a non empty polydisc, and let~${d\in \R}$ such that for all~$t\in V$,~$d$ is not an unparameterized singular direction of~$\pz Y(z,t)=A(z,t)Y(z,t)$. Then, there exists a map ~$U\rightarrow \R^{>0}$,~$t\mapsto\e(t)$, which is not necessary continuous such that~$$ \mathcal{L}_{\kappa_{r},d}\circ\dots\circ\mathcal{L}_{\kappa_{1},d}\circ
\hat{\mathcal{B}}_{\kappa_{r}}\circ\dots\circ\hat{\mathcal{B}}_{\kappa_{1}}\Big(\hat{h}\Big)$$
is meromorphic in~$(z,t)$ on~$$(z,t)\in\left\{ z\in \widetilde{\C}\Big|\arg(z)\in \left]d-\frac{\pi}{2k_{r}},d+\frac{\pi}{2k_{r}}\right[, \hbox{ and }0<|z|<\e(t) \right\}\times V.$$ Moreover, for all~$j\leq n$:
$$\mathcal{L}_{\kappa_{r},d}\circ\dots\circ\mathcal{L}_{\kappa_{1},d}\circ
\hat{\mathcal{B}}_{\kappa_{r}}\circ\dots\circ\hat{\mathcal{B}}_{\kappa_{1}}\Big(\partial_{t_{j}}\hat{h}\Big) 
=\partial_{t_{j}}\left(\mathcal{L}_{\kappa_{r},d}\circ\dots\circ\mathcal{L}_{\kappa_{1},d}\circ
\hat{\mathcal{B}}_{\kappa_{r}}\circ\dots\circ\hat{\mathcal{B}}_{\kappa_{1}}\Big(\hat{h}\Big)\right).$$
\end{lem}

\begin{proof}
We will proceed in two steps.
\begin{trivlist}
\item (1) We recall that~$\hat{h}(z,t)\in \hat{K}_{U}\left[z^{1/\nu}\right]$ ($\nu\in \N^{*}$ has been defined in Proposition \ref{3propo4}) and (see Remark \ref{3rem3}) all the~$z$-coefficients are analytic on~$U$. Because of Proposition \ref{3propo1}, for~$t$ fixed,~$\hat{\mathcal{B}}_{\kappa_{r}}\circ\dots\circ\hat{\mathcal{B}}_{\kappa_{1}}\Big(\hat{h}\Big),$ is a germ of a meromorphic function. Therefore, it belongs to~$\mathcal{O}_{U}(\{z\})\left[z^{1/\nu}\right]$ .
Let~$h_{1}$ be the analytic continuation defined in Proposition \ref{3propo1}. In particular, for all~$z\in \widetilde{\C}$ with~$\arg(z)=d$,~$t\mapsto h_{1}(z,t)\in \mathcal{M}_{V}$. The fact that we have a meromorphic function allows us to differentiate termwise and for all~$j\leq n$,~$\partial_{t_{j}}h_{1}$ is equal to the analytic continuation of:
$$
\hat{\mathcal{B}}_{\kappa_{r}}\circ\dots\circ\hat{\mathcal{B}}_{\kappa_{1}}\Big(\partial_{t_{j}}\hat{h}\Big) .$$
\item 
(2) Let~$h_{2},\dots,h_{r}$ be the successive Laplace transforms defined in Proposition \ref{3propo1}. Let~$t_{0}\in V$, let~$W_{t_{0}}$ be a compact neighborhood of~$t_{0}$ in~$V$, let~$i\leq r$, and assume that for~$z\in \widetilde{\C}$ with~$\arg(z)=d$,~$t\mapsto h_{i}(z,t)$ is meromorphic on~$W_{t_{0}}$. It is sufficient to prove that for all~$z\in \widetilde{\C}$ with~$\arg(z)\in\left]d-\frac{\pi}{2\kappa_{i}},d+\frac{\pi}{2\kappa_{i}}\right[$ and~$|z|$ sufficiently small,~$t\mapsto h_{i+1}(z,t)$ is meromorphic on~$W_{t_{0}}$ and for all~$j\leq n$: 
$$\mathcal{L}_{\kappa_{i},d}\Big(\partial_{t_{j}}h_{i}\Big) 
=\partial_{t_{j}}\Big(\mathcal{L}_{\kappa_{i},d}(h_{i})\Big)=\partial_{t_{j}}h_{i+1}.$$
 The function
$\mathcal{L}_{\kappa_{i},d}\left(h_{i}\right)$ is an integral of a meromorphic function depending analytically upon parameters, and we just have to prove that it is possible to find a function~$f$ such that, for all~$t\in W_{t_{0}}$,~$|h_{i}(u,t)|<|f(u)|$ and for~$\arg(z)\in\left]d-\frac{\pi}{2\kappa_{i}},d+\frac{\pi}{2\kappa_{i}}\right[$, ~$|z|$ sufficiently small,~$\mathcal{L}_{\kappa_{i},d}(|f|)(z)<\infty$.
From Proposition \ref{3propo1}, we obtain the existence of~$A(t),B(t)>0$ such that for~$\arg(u)=d$,~$|h_{i}(u,t)|\leq A(t)e^{B(t)|u|^{\kappa_{i}}}$. Since~$h_{i}(u,t)$ is meromorphic, we may assume that~$A(t)$ and~$B(t)$ are continuous on~$W_{t_{0}}$.  The functions~$A(t)$ and~$B(t)$ admit a maximum~$A$ and~$B$ on the compact set~$W_{t_{0}}$. Finally for~$\arg(z)\in\left]d-\frac{\pi}{2\kappa_{i}},d+\frac{\pi}{2\kappa_{i}}\right[$ and~$|z|$ sufficiently small, 
$$\left| \mathcal{L}_{\kappa_{i},d}h_{i}\right| =\left| \displaystyle \int_{0}^{\infty e^{id}}h_{i}(u,t)e^{-\left(\frac{u}{z}\right)^{\kappa_{i}}}d\left(\left(\frac{u}{z}\right)^{\kappa_{i}}\right)\right|\leq  \displaystyle \int_{0}^{\infty } Ae^{B|u|^{\kappa_{i}}}\left|e^{-\left(\frac{u}{z}\right)^{\kappa_{i}}}\right|d\left(\left(\frac{u}{z}\right)^{\kappa_{i}}\right)<\infty.$$
\end{trivlist}
\end{proof}

\pagebreak[3]
\section{Parameterized differential Galois theory}

In this section we are interested in the parameterized differential Galois theory: this is a generalization of the differential Galois theory for parameterized linear differential equations. In~$\S \ref{3sec21}$, we review the parameterized differential Galois theory developed in \cite{CS}. In~$\S \ref{3sec22}$, we prove that some of the results of~$\S \ref{3sec21}$ stay valid without the assumption that the field of constants is differentially closed. This will help us in~$\S$ \ref{3sec22'} to prove that the local analytic parameterized differential Galois group descends to a smaller field, whose  field  of  constants  is not differentially  closed. In~$\S \ref{3sec23}$, we explain the main result of the paper: we show an analogue of the density theorem of Ramis in the parameterized case. In~$\S \ref{3sec24}$, we give a similar result for the global parameterized differential Galois group. We end by giving various examples of computation of parameterized differential Galois groups using the parameterized density theorem.

\pagebreak[3]
\subsection{Basic facts.}\label{3sec21}

We recall some facts from \cite{CS} about Galois theory of parameterized linear differential equations. Classical Galois theory of unparameterized linear differential equation is presented in some books such as  \cite{VdPS} and \cite{Ma}. \\ \par 
Let~$K$ be a differential field of characteristic~$0$ with~$n+1$ commuting derivations:~$\partial_{0},\dots,\partial_{n}$. We want to study differential equations of the form~$\partial_{0} Y=AY$, with~${A\in \mathrm{M}_{m}(K)}$. Let~$C_{K}$ be the field of constants with respect to~$\partial_{0}$. Since all the derivations commute with~$\partial_{0}$,~$(C_{K},\partial_{1},\dots,\partial_{n})$ is a differential field. By abuse, we will sometimes start from a~$(\partial_{1},\dots,\partial_{n})$-differential field~$C_{K}$ and build a~$(\partial_{0},\dots,\partial_{n})$-differential field extension~$K$ of~$C_{K}$, such that~$C_{K}$ is the field of constants with respect to~$\partial_{0}$. 
\pagebreak[3]
\begin{ex}
If~$K=\hat{K}_{U}$, then~$\partial_{0}=\pz$,~$\{ \partial_{1},\dots,\partial_{n}\}=\dt$, and~$C_{K}=\mathcal{M}_{U}$. 
\end{ex}

A parameterized Picard-Vessiot extension for the parameterized linear differential equation~$\partial_{0} Y=AY$ on~$K$ is a ($\partial_{0},\dots,\partial_{n}$)-differential field extension~$\widetilde{K}\Big|K$  with the following properties:
\begin{itemize}
 \item  There exists a fundamental solution for~$\partial_{0}Y=AY$ in~$\widetilde{K}$, i.e., an invertible matrix~$U=(u_{i,j})$, with entries in~$\widetilde{K}$,  such that~$\partial_{0} U=AU$.
\item~$\widetilde{K}=K\langle u_{i,j} \rangle_{\partial_{0},\dots,\partial{n}}$, i.e.,~$\widetilde{K}$ is the ($\partial_{0},\dots,\partial_{n}$)-differential field generated by~$K$ and the~$ u_{i,j}$. 
\item The field of constants of~$\widetilde{K}$ with respect to~$\partial_{0}$  is~$C_{K}$.
\end{itemize}
\par
Let~$L$ be a~$(\partial_{1},\dots,\partial_{n})$-field of characteristic~$0$ with commuting derivations. The~$(\partial_{1},\dots,\partial_{n})$-differential ring~$L\{y_{1},\dots, y_{k}\}_{\partial_{1},\dots,\partial_{n}}$ of differential polynomials in~$k$ indeterminates over~$L$ is the usual polynomial ring in the infinite set of variables
$$\{ \partial_{1}^{\nu_{1}}\dots\partial_{n}^{\nu_{n}}y_{j}\}^{\nu_{i}\in \N}_{j\leq k},~$$
and with derivations extending those in~$\{\partial_{1},\dots,\partial_{n}\}$ on~$L$, defined by:
$$\partial_{i}\left(\partial_{1}^{\nu_{1}}\dots\partial_{n}^{\nu_{n}}y_{j} \right) =\partial_{1}^{\nu_{1}}\dots\partial_{i}^{\nu_{i}+1}\dots\partial_{n}^{\nu_{n}}y_{j}.$$
\pagebreak[3]
\begin{defi}[\cite{CS}, Definition 3.2]\label{3defi4}
 We say that~$(C_{K},\partial_{1},\dots,\partial_{n})$ is differentially closed if it has the following property: 
For any~$k,l \in \N$ and for all~${P_{1},\dots,P_{k}\in C_{K}\{y_{1},\dots,y_{l}\}_{\partial_{1},\dots,\partial_{n}}}$, the system 
 \[
 \; \left \{
\begin{array}{ccl}
     P_{1}(\a_{1},\dots,\a_{l})&=&0  \\
      & \vdots     &           \\   
   P_{k-1}(\a_{1},\dots,\a_{l})&=&0 \\
  P_{k}(\a_{1},\dots,\a_{l})&\neq&0, \\
\end{array}
\right. 
\]  has a solution in~$C_{K}$ as soon as it has a solution in a~$(\partial_{1},\dots,\partial_{n})$-differential field containing~$C_{K}$.
\end{defi}
For the simplicity of the notation, we will say that~$C_{K}$ differentially closed rather than~$(C_{K},\partial_{1},\dots,\partial_{n})$ is differentially closed. Note that there exists a differentially closed extension of~$C_{K}$, see \cite{CS}, Section 9.1. By definition, a differentially closed field is algebraically closed.
\pagebreak[3]
\begin{propo}[\cite{CS}, Theorem 9.5]
Assume that~$C_{K}$ is differentially closed. Then, we have existence of the parameterized Picard-Vessiot extension for~$\partial_{0} Y=AY$. We have also the uniqueness of the  parameterized Picard-Vessiot extension for~$\partial_{0} Y=AY$, up to~$(\partial_{0},\dots,\partial_{n})$-differential isomorphism. 
\end{propo}
\textbf{Until the end of the subsection \ref{3sec21}, we assume that~$C_{K}$ is differentially closed.}\par
Let us consider~$\partial_{0} Y=AY$, with~${A\in \mathrm{M}_{m}(K)}$ and let~$\widetilde{K}\Big| K$ be a parameterized Picard-Vessiot extension. The parameterized differential Galois group~$Gal_{\partial_{0}}^{\partial_{1},\dots,\partial_{n}}\left(\widetilde{K}\Big| K\right)$ is the group of field automorphisms of~$\widetilde{K}$ which induce the identity on~$K$ and commute with all the derivations. This latter is independent of the choice of the parameterized Picard-Vessiot extension, since all the parameterized Picard-Vessiot extensions are~$(\partial_{0},\dots,\partial_{n})$-differentially isomorphic. In the  unparameterized case, the differential Galois group is an algebraic subgroup of~$\mathrm{GL}_{m}(C_{K})$. In the parameterized case, we find a linear differential algebraic subgroup:
\pagebreak[3]
\begin{defi}\label{3defi5}
Let us consider~$m^{2}$ indeterminates~$(X_{i,j})_{i,j\leq m}$. We say that a subgroup~$G$ of ~$\mathrm{GL}_{m}(C_{K})$ is a linear differential algebraic group if there exist ${P_{1},\dots,P_{k}\in C_{K}\{X_{i,j}\}_{\partial_{1},\dots,\partial_{n}}}$ such that for~${A=(a_{i,j}) \in \mathrm{GL}_{m}(C_{K})}$,~$$A \in G \Longleftrightarrow P_{1}(a_{i,j})=\dots=P_{k}(a_{i,j})=0.$$
\end{defi}
Let~$U$ be a fundamental solution of~$\partial_{0}Y=AY$. One proves directly that the map:\\
$$\begin{array}{cccc}
\r_{U}: & Gal_{\partial_{0}}^{\partial_{1},\dots,\partial_{n}}\left(\widetilde{K}\Big| K\right) & \longrightarrow & \mathrm{GL}_{m}(C_{K}) \\ 
 & \f &\longmapsto  & U^{-1}\f(U),
\end{array}$$
is an injective group morphism.
A fundamental fact is that~$$\hbox{Im } \r_{U}=\left\{ U^{-1}\f(U), \f \in Gal_{\partial_{0}}^{\partial_{1},\dots,\partial_{n}}\left(\widetilde{K}\Big| K\right) \right\}~$$ is a linear differential algebraic subgroup of~$\mathrm{GL}_{m}(C_{K})$ (see Theorem 9.5 in \cite{CS}). If we take a different fundamental solution in~$\widetilde{K}$, we obtain a conjugate linear differential algebraic subgroup of~$\mathrm{GL}_{m}(C_{K})$. We will identify~$Gal_{\partial_{0}}^{\partial_{1},\dots,\partial_{n}}\left(\widetilde{K}\Big| K\right)~$ with a linear differential algebraic subgroup of~$\mathrm{GL}_{m}(C_{K})$ for a chosen fundamental solution. We put a topology on~$\mathrm{GL}_{m}(C_{K})$, called Kolchin topology, for which the closed sets are defined as the zero loci of finite sets of differential polynomials with coefficients in~$C_{K}$. 
\pagebreak[3]
\begin{ex}\label{3ex3}{(Example 3.1 in \cite{CS})}
Let~$n=1$, let~$(C_{K},\pt)$ be a differentially closed~$\pt$-field that contains~$(\C(t),\pt)$, and let us consider~$K=C_{K}(z)$, the~$(\pz,\pt)$-differential field of rational functions in the indeterminate~$z$, with coefficients in~$C_{K}$, where~$z$ is a~$\pt$-constant with~$\pz z =1$,~$C_{K}$ is the field of constants with respect to~$\pz$, and~$\pz$ commutes with~$\pt$. Let us consider the parameterized differential equation
$$\pz Y(z,t)=\frac{t}{z}Y(z,t).$$ 
The fundamental solution is~$(z^{t})$ and~$K(z^{t},\log)$ is a Parameterized Picard-Vessiot extension (see~$\S \ref{3sec11}$ for the notations). Here, we have added~$\log$ because we want the extension to be closed under the derivations~$\pz$ and~$\pt$. Using the fact the Galois group commutes with~$\pz$ and~$\pt$, we find that the Galois group is given by:
$$\left\{  f\in C_{K}\big| f\neq 0 \; \hbox{and} \; f\pt^{2}f -(\pt f)^{2}=0\right\}.$$
We can see that if we take~$C_{K}=\C(t)$ or~$C_{K}=\mathcal{M}_{\C}$ (see page \pageref{3p1}), which are not differentially closed, then we find two different groups of differential automorphisms:
$$\left\{  f\in \C(t)\big| f\neq 0 \; \hbox{and} \; f\pt^{2}f -(\pt f)^{2}=0\right\}=\C^{*}$$
and 
$$\left\{  f\in \mathcal{M}_{\C}\big| f\neq 0 \; \hbox{and} \; f\pt^{2}f -(\pt f)^{2}=0\right\}=\left\{ ce^{bt} \big| b \in \C, c\in \C^{*} \right\} ,$$
which shows the importance of considering a Galois group defined over a differentially closed field.
See Example \ref{3ex1} for the resolution of this ambiguity using the parameterized density theorem.
\end{ex}

There is a Galois correspondence theorem for parameterized differential Galois theory, see Theorem~9.5 in \cite{CS}. For ~$G$ subgroup of~$Gal_{\partial_{0}}^{\partial_{1},\dots,\partial_{n}}\left(\widetilde{K}\Big| K\right)~$, let:~$$\widetilde{K}^{G}=\left\{ a \in \widetilde{K}\Big| \s (a)=a ,\; \forall \s \in G \right\} .$$ 
Then, the theorem says that the Kolchin closed subgroups of~$Gal_{\partial_{0}}^{\partial_{1},\dots,\partial_{n}}\left(\widetilde{K}\Big| K\right)~$ are in bijection with the~$(\partial_{0},\dots,\partial_{n})$-differential subfields of~$\widetilde{K}$ containing~$K$, via the map:~$$G \mapsto \widetilde{K}^{G}.$$
The inverse map is given by:
$$M \mapsto Gal_{\partial_{0}}^{\partial_{1},\dots,\partial_{n}}\left(\widetilde{K}\Big| M\right),$$
where~$Gal_{\partial_{0}}^{\partial_{1},\dots,\partial_{n}}\left(\widetilde{K}\Big| M\right)$ denotes the set of elements of~$Gal_{\partial_{0}}^{\partial_{1},\dots,\partial_{n}}\left(\widetilde{K}\Big| K\right)$ inducing identity on~$M$. In particular, we have the following corollary:
\pagebreak[3]
\begin{coro}\label{3coro3}
Let~$G$ be an arbitrary subgroup of~$Gal_{\partial_{0}}^{\partial_{1},\dots,\partial_{n}}\left(\widetilde{K}\Big| K\right)~$. Then,~$\widetilde{K}^{G}=K$ if and only if~$G$ is dense for Kolchin topology in~$Gal_{\partial_{0}}^{\partial_{1},\dots,\partial_{n}}\left(\widetilde{K}\Big| K\right)~$.
\end{coro} 

Let~$L | M | K$ be~$(\partial_{1},\dots,\partial_{n})$-differential field extensions. Notice that we do not exclude~${L=M=K}$. All the definitions we are going to give before the next Proposition come from \cite{HS},~$\S$ 6.2.3.\par
Given~$a_{1},\dots,a_{k}\in L$ and~$P\in M\{X_{1},\dots,X_{k}\}_{\partial_{1},\dots,\partial_{n}}$, we remark that~$P(a_{1},\dots,a_{n})$ is well defined.
Then, we may define the ~$(\partial_{1},\dots,\partial_{n})$-differential transcendence degree of~$L$ over~$M$ as the maximum number of elements~$a_{1},\dots,a_{k}$ of~$L$ such that:
$$ P(a_{1},\dots,a_{k})\neq 0,$$
for all non-zero~$(\partial_{1},\dots,\partial_{n})$-differential polynomials~$P$ with coefficients in~$M$. The~$(\partial_{1},\dots,\partial_{n})$-differential transcendence degree of an integral domain over another integral domain is defined to be the~$(\partial_{1},\dots,\partial_{n})$-differential transcendence degree of the fraction field of the first one over the fraction field of the second one.\par 
Let us consider~$m^{2}$ indeterminates~$(X_{i,j})_{i,j\leq m}$. Let~$(p)$ be a prime~$(\partial_{1},\dots,\partial_{n})$-differential ideal of~$C_{K}\{ X_{i,j}\}_{\partial_{1},\dots,\partial_{n}}$, i.e., a prime ideal stable under the derivations~$\partial_{1},\dots,\partial_{n}$. The~$(\partial_{1},\dots,\partial_{n})$-dimension of~$(p)$ over~$C_{K}$ is defined to be the ~$(\partial_{1},\dots,\partial_{n})$-differential transcendence degree of the quotient ring~$C_{K}\{ X_{i,j}\}_{\partial_{1},\dots,\partial_{n}}/(p)$ over~$C_{K}$.\par 
Let~$(r)$ be a radical~$(\partial_{1},\dots,\partial_{n})$-differential ideal of~$C_{K}\{ X_{i,j}\}_{\partial_{1},\dots,\partial_{n}}$, i.e., a radical ideal stable under the derivations~$\partial_{1},\dots,\partial_{n}$. Let ~$ (p_{1}),\dots,(p_{\nu})$ with~$\nu\in \N^{*}$ be the prime~$(\partial_{1},\dots,\partial_{n})$-differential ideals such that~$(r)=\displaystyle \bigcap_{k\leq \nu} (p_{k})$. The~$(\partial_{1},\dots,\partial_{n})$-dimension of~$(r)$ over~$C_{K}$ is defined to be the maximum in~$k$ of the~$(\partial_{1},\dots,\partial_{n})$-dimension of~$(p_{k})$ over~$C_{K}$.\par 
Assume that~$M\subset \widetilde{K}$. Let~$(q)$ be the radical~$(\partial_{1},\dots,\partial_{n})$-differential ideal of~$C_{K}\{ X_{i,j}\}_{\partial_{1},\dots,\partial_{n}}$ that defines~$Gal_{\partial_{0}}^{\partial_{1},\dots,\partial_{n}}\left(\widetilde{K}\Big| M\right)$ (see the proof of Proposition 9.10 in \cite{CS}). We define the~$(\partial_{1},\dots,\partial_{n})$-differential  dimension of 
$Gal_{\partial_{0}}^{\partial_{1},\dots,\partial_{n}}\left(\widetilde{K}\Big| M\right)$ over~$C_{K}$  as the~$(\partial_{1},\dots,\partial_{n})$-dimension of~$(q)$ over~$C_{K}$. 
\pagebreak[3]
\begin{propo}[\cite{HS}, Proposition 6.26]\label{3propo9}
The~$(\partial_{1},\dots,\partial_{n})$-differential transcendence degree of~$\widetilde{K}$ over~$M$ is equal to the 
$(\partial_{1},\dots,\partial_{n})$-differential  dimension  of 
$Gal_{\partial_{0}}^{\partial_{1},\dots,\partial_{n}}\left(\widetilde{K}\Big| M\right)$ over~$C_{K}$.
\end{propo}
\pagebreak[3]
\begin{ex2}\textit{\ref{3ex3} (bis).}
Let us keep the same notations as in Example \ref{3ex3}. The parameterized Picard-Vessiot extension is~$K(z^{t},\log)$ and the Galois group is:
$\left\{  f\in C_{K}\big| f\neq 0 \; \hbox{and} \; f\pt^{2}f -(\pt f)^{2}=0\right\}.$ We may directly check that the~$\pt$-differential  dimension of the Galois group is~$0$ and therefore,~$z^{t}$ satisfies a~$\pt$-differential polynomial equation with coefficients in~$C_{K}$.
\end{ex2}

\pagebreak[3]
\subsection{Parameterized differential Galois theory for a non-differentially closed field of constants.}\label{3sec22}
 Let~$K$ be a differential field of characteristic~$0$ with~$n+1$ commuting derivations:~$\partial_{0},\dots,\partial_{n}$. Let~$C_{K}$ be the field of constants with respect to~$\partial_{0}$. Note that we do not assume~$C_{K}$ to be differentially closed. Consider~$\partial_{0}Y=AY$, with~$A\in \mathrm{M}_{m}(K)$, and assume the existence of~$\widetilde{K}\Big|K$, a parameterized Picard-Vessiot extension for ~$\partial_{0}Y=AY$ (see~$\S$ \ref{3sec21}). This means in particular that the field of constants of~$\widetilde{K}$ with respect to~$\partial_{0}$ is~$C_{K}$. Let~$F=(F_{i,j})\in \mathrm{GL}_{m}\left(\widetilde{K}\right)$ be a fundamental solution such that~$\widetilde{K}= K\langle F_{i,j}\rangle_{\partial_{0},\dots,\partial{n}}$ (see~$\S \ref{3sec21}$ for the notation). Let ~$Aut_{\partial_{0}}^{\partial_{1},\dots,\partial_{n}}\left(\widetilde{K}\Big|K\right)$ be the group of~$(\partial_{0},\dots,\partial_{n})$-differential field automorphisms of~$\widetilde{K}$ letting~$K$ invariant

\pagebreak[3]
\begin{rem}\label{3rem6}
We avoid here the notation~$Gal_{\partial_{0}}^{\partial_{1},\dots,\partial_{n}}\left(\widetilde{K}\Big|K\right)$, because we have no theorem that guarantees the uniqueness of the parameterized Picard-Vessiot extension~$\widetilde{K}\Big|K$, since~$C_{K}$ is not differentially closed. However we will call it the parameterized differential Galois group, or Galois group, if no confusion is likely to arise.
\end{rem}

We extend Definition \ref{3defi5} for the field~$C_{K}$. Let us consider~$m^{2}$ indeterminates~$(X_{i,j})_{i,j\leq m}$. We say that a subgroup~$G$ of ~$\mathrm{GL}_{m}(C_{K})$ is a linear differential algebraic group if there exist~${P_{1},\dots,P_{k}\in C_{K}\{X_{i,j}\}_{\partial_{1},\dots,\partial_{n}}}$ such that for~$ A=(a_{i,j}) \in \mathrm{GL}_{m}(C_{K})$,~$$A \in G \Longleftrightarrow P_{1}(a_{i,j})=\dots=P_{k}(a_{i,j})=0.$$
The goal of the subsection is to prove:
\pagebreak[3]
\begin{propo}\label{3propo8}
\begin{trivlist}
\item (1) Let us consider the injective group morphism:~$$\begin{array}{cccc}
\r_{F}: & Aut_{\partial_{0}}^{\partial_{1},\dots,\partial_{n}}\left(\widetilde{K}\Big|K\right)& \longrightarrow & \mathrm{GL}_{m}(C_{K}) \\ 
 & \f &\longmapsto  & F^{-1}\f(F).
\end{array}~$$
Then,
$$\hbox{Im } \r_{F}=\left\{ F^{-1}\f(F), \f \in Aut_{\partial_{0}}^{\partial_{1},\dots,\partial_{n}}\left(\widetilde{K}\Big|K\right) \right\}$$ is a linear differential algebraic subgroup of~$\mathrm{GL}_{m} (C_{K})$. We will identify~$Aut_{\partial_{0}}^{\partial_{1},\dots,\partial_{n}}\left(\widetilde{K}\Big|K\right)$ with a linear differential algebraic subgroup of~$\mathrm{GL}_{m} (C_{K})$ for a chosen fundamental solution. The image is independent of this choice, up to conjugacy by an element of~$\mathrm{GL}_{m}(C_{K})$.
\item (2) Let~$G$ be a subgroup of~$Aut_{\partial_{0}}^{\partial_{1},\dots,\partial_{n}}\left(\widetilde{K}\Big|K\right)$. If~$\widetilde{K}^{G}=K$, then~$G$ is dense for Kolchin topology in~$Aut_{\partial_{0}}^{\partial_{1},\dots,\partial_{n}}\left(\widetilde{K}\Big|K\right)$. 
\end{trivlist}
\end{propo}
Remark that, contrary to Corollary \ref{3coro3}, the converse of (2) is false when~$C_{K}$ is not differentially closed. See \cite{CS}, Example 3.1.
Before showing the proposition, we point out two facts we will use in the proof. Let~$L|K$ be a~$(\partial_{0},\dots,\partial_{n})$-differential field extension and~$a_{1},\dots,a_{k}\in L$.
\begin{itemize}
\item As in the case where~$C_{K}$ is differentially closed (see~$\S \ref{3sec21}$), if~${P\in K\{ X_{1},\dots,X_{k} \}_{\partial_{1},\dots,\partial_{n}}}$, then~$P(a_{1},\dots,a_{k})$ is well defined.
\item The set~$\left\{ P(a_{1},\dots,a_{k})\big|P\in K\{ X_{1},\dots,X_{k} \}_{\partial_{1},\dots,\partial_{n}}\right\}$ is a~$(\partial_{0},\dots,\partial_{n})$-differential field extension we will denote by~$L\{ a_{1},\dots,a_{k} \}_{\partial_{1},\dots,\partial_{n}}\big|L$.
\end{itemize}
\begin{proof}[Proof of Proposition \ref{3propo8}]
\begin{trivlist}
\item (1)
We follow here the proof of Proposition 9.10 in \cite{CS}. We consider the differential polynomial ring:~$$R=K\{ X_{i,j},1/\det(X_{i,j}) \}_{\partial_{1},\dots,\partial_{n}},$$ and endow it with the~$\partial_{0}$-differential structure defined by~$\partial_{0} (X_{i,j})=A(X_{i,j})$. Let us consider:~$$S=K\{ F_{i,j},1/\det(F_{i,j}) \}_{\partial_{0},\dots,\partial_{n}},$$
the~$(\partial_{0},\dots,\partial_{n})$-differential subring of~$\widetilde{K}$ generated over~$K$ by the~$F_{i,j}$ and~$1/\det(F_{i,j})$. It is an integral domain. Let~$q$ be the obvious prime~$(\partial_{0},\dots,\partial_{n})$-differential ideal such that~$R/q\simeq S$. Let~$Z_{i,j}$ be the image of~$X_{i,j}$ in~$S\subset\widetilde{K}$, so that~$(Z_{i,j})$ is a fundamental solution for~$\partial_{0}Y=AY$ in~$S$.  Consider the following rings:
$$\begin{array}{clc}
\widetilde{K}\{ X_{i,j},1/\det (X_{i,j}) \}_{\partial_{1},\dots,\partial_{n}}&=&\widetilde{K}\{ Y_{i,j},1/\det (Y_{i,j}) \}_{\partial_{1},\dots,\partial_{n}}\\
\cup &&\cup\\
K\{ X_{i,j},1/\det (X_{i,j}) \}_{\partial_{1},\dots,\partial_{n}}&&
 C_{K}\{ Y_{i,j},1/\det (Y_{i,j})  \}_{\partial_{1},\dots,\partial_{n}},
\end{array}
$$
where the indeterminates~$Y_{i,j}$ are defined by~$(X_{i,j})=(Z_{i,j})(Y_{i,j})$. We remark that~$\partial_{0} (Y_{i,j})=0$.
Since we consider fields that are of characteristic~$0$, the differential ideal:~$$q\widetilde{K}\{ Y_{i,j},1/\det (Y_{i,j}) \}_{\partial_{1},\dots,\partial_{n}}\subset \widetilde{K}\{ X_{i,j},1/\det (X_{i,j}) \}_{\partial_{1},\dots,\partial_{n}}=\widetilde{K}\{ Y_{i,j},1/\det (Y_{i,j}) \}_{\partial_{1},\dots,\partial_{n}},$$ is a radical~$(\partial_{0},\dots,\partial_{n})$-differential ideal (see Corollary A.17 in \cite{VdPS}).  
The next lemma is an adaptation  of Lemma 9.8 in \cite{CS} without the assumption that the field of constants is differentially closed.
\pagebreak[3]
\begin{lem}\label{3lem2}
 The~$(\partial_{0},\dots,\partial_{n})$-ideal~$q\widetilde{K} \{ Y_{i,j},1/\det(Y_{i,j})\}_{\partial_{1},\dots,\partial_{n}}$ is generated by: 	$$I=q\widetilde{K} \{ Y_{i,j},1/\det(Y_{i,j})\}_{\partial_{1},\dots,\partial_{n}}\cap C_{K} \{ Y_{i,j},1/\det(Y_{i,j})\}_{\partial_{1},\dots,\partial_{n}}.$$
\end{lem}

\begin{proof}
Let~$(e_{i})_{i\in B}$ be a basis of~$C_{K}\{ Y_{i,j},1/\det(Y_{i,j})\}_{\partial_{1},\dots,\partial_{n}}$ over~$C_{K}$. Let:~$$f=\sum_{i=1}^{n} m_{i}e_{i}\in q\widetilde{K}\{ Y_{i,j},1/\det(Y_{i,j})\}_{\partial_{1},\dots,\partial_{n}},$$ with~$m_{i} \in \widetilde{K}.$ By induction on~$n$ we will show that~$f\in I$. If~$n=0$ or~$1$ there is nothing to prove. We assume that~$n>1$. We can suppose that~$m_{1}=1$ and~$m_{2}\notin C_{K}$. Then, because of the fact that the field of constants of~$\widetilde{K}$ with respect to~$\pz$ is~$C_{K}$:
$$\partial_{0}(f)=\sum_{i=2}^{n} \partial_{0}(m_{i})e_{i} \neq 0 \hbox{ and } f\in q\widetilde{K}\{ Y_{i,j},1/\det(Y_{i,j})\}_{\partial_{1},\dots,\partial_{n}}.$$
Then, by induction,~$\partial_{0}(f)\in I$. With the same argument:
$$\partial_{0}(m_{2}^{-1}f) \in I.$$ Then,~$\partial_{0}(m_{2}^{-1})f=\partial_{0}(m_{2}^{-1}f)-m_{2}^{-1}\partial_{0}f \in I$. Since~$\partial_{0}(m_{2}^{-1})\neq 0$, we obtain that~$f\in I$.
\end{proof}

By Lemma \ref{3lem2},~$q\widetilde{K}\{ X_{i,j},1/\det (X_{i,j}) \}_{\partial_{1},\dots,\partial_{n}}$ is generated by:
~$$I=q\widetilde{K}\{ X_{i,j},1/\det (X_{i,j}) \}_{\partial_{1},\dots,\partial_{n}}\cap C_{K} \{ Y_{i,j},1/\det(Y_{i,j})\}_{\partial_{1},\dots,\partial_{n}}.$$ Clearly~$I$ is a~$(\partial_{1},\dots,\partial_{n})$-radical ideal of~$C_{K} \{ Y_{i,j},1/\det(Y_{i,j})\}_{\partial_{1},\dots,\partial_{n}}$.
Let~${C=(C_{i,j}) \in \mathrm{GL}_{m} (C_{K})}$. The following statements are equivalent:
\begin{enumerate}
\item~$(C_{i,j}) \in Aut_{\partial_{0}}^{\partial_{1},\dots,\partial_{n}}\left(\widetilde{K}\Big|K\right)$.
\item The map~$K\{ X_{i,j},1/\det(X_{i,j})\}_{\partial_{1},\dots,\partial_{n}}\rightarrow K\{ X_{i,j},1/\det(X_{i,j})\}_{\partial_{1},\dots,\partial_{n}}$ defined by~${(X_{i,j})\mapsto (X_{i,j})(C_{i,j}):=(\displaystyle\sum_{k=1}^{m}X_{i,k}C_{k,j})}$ leaves~$q$ invariant.
\item  The map~$K\{ X_{i,j},1/\det(X_{i,j})\}_{\partial_{1},\dots,\partial_{n}}\rightarrow \widetilde{K}$ defined by ~$(X_{i,j})\mapsto (Z_{i,j})(C_{i,j})$ sends~$q$ to~$0$.
\item  The map~$\widetilde{K}\{ X_{i,j},1/\det(X_{i,j})\}_{\partial_{1},\dots,\partial_{n}}\rightarrow \widetilde{K}$ defined by~$(X_{i,j})\mapsto (Z_{i,j})(C_{i,j})$ sends 
\begin{center}
$q\widetilde{K}\{ X_{i,j},1/\det(X_{i,j})\}_{\partial_{1},\dots,\partial_{n}}=q\widetilde{K}\{ Y_{i,j},1/\det(Y_{i,j})\}_{\partial_{1},\dots,\partial_{n}}$ to~$0$.
\end{center}
\item  The map~$\widetilde{K}\{ Y_{i,j},1/\det(Y_{i,j})\}_{\partial_{1},\dots,\partial_{n}}\rightarrow \widetilde{K}$ defined by~$(Y_{i,j})\mapsto (C_{i,j})$ sends 
\begin{center}
   ~$q\widetilde{K}\{ Y_{i,j},1/\det(Y_{i,j})\}_{\partial_{1},\dots,\partial_{n}}$ to~$0$.
\end{center}
\end{enumerate}
The theorem is now a consequence of the fact that~$q\widetilde{K}\{ Y_{i,j},1/\det(Y_{i,j})\}_{\partial_{1},\dots,\partial_{n}}$ is generated by~$I$, a~$(\partial_{1},\dots,\partial_{n})$-radical ideal of~$C_{K} \{ Y_{i,j},1/\det(Y_{i,j})\}_{\partial_{1},\dots,\partial_{n}}$.
\item (2)
We follow the proof of Proposition 9.10 in \cite{CS}, and use the same notations as before. By construction, the ideal~$I$ of Lemma \ref{3lem2} above is the differential ideal that defines the Galois group.
Assume that the Kolchin closure of~$G$ is not the whole Galois group.  Then, there exists~${P\in C_{K}\{ Y_{i,j},1/\det(Y_{i,j}) \}_{\partial_{1},\dots,\partial_{n}}}$ such that~$P\notin I$ and~$P(g)= 0$ for all~$g\in G$. Lemma \ref{3lem2} implies that~$$P\notin J=q\widetilde{K} \{ Y_{i,j},1/\det(Y_{i,j})\}_{\partial_{1},\dots,\partial_{n}}.$$ Let~$$T=\Big\{Q\in  \widetilde{K}\{ X_{i,j},1/\det(X_{i,j}) \}_{\partial_{1},\dots,\partial_{n}}\Big|Q\notin J \hbox{ and } Q\Big((Z_{i,j})(g_{i,j})\Big)=0, \forall g=(g_{i,j})\in G\Big\}.$$ Since~$P\in T$,~$T\neq \{0\}$. An element~$Q\in T$ can be written as:
$$Q=f_{1}Q_{1}+\dots+f_{\nu}Q_{\nu},$$
where~$f_{i}\in \widetilde{K}$ and~$Q_{i}\in K\{ X_{i,j},1/\det(X_{i,j}) \}_{\partial_{1},\dots,\partial_{n}}$. Let~$Q=f_{1}Q_{1}+\dots+f_{\nu}Q_{\nu}\in T$ such that:
\begin{itemize}
\item~$f_{1}=1$.
\item All the~$f_{i}$ are non-zero.
\item~$\nu$ is minimal.
\end{itemize} 
For all~$g\in G$, let~$Q^{g}=f_{1}^{g}Q_{1}+\dots+f_{\nu}^{g}Q_{\nu}\in T$.
Let~$g\in G$. Since~$Q-Q^{g}$ is shorter than~$Q$, and satisfies~$\left(Q-Q^{g}\right)((Z_{i,j})(g_{i,j}))=0$, we have~$Q-Q^{g}\in J$. If~$Q-Q^{g}\neq 0$, there exists~$l\in\widetilde{K}$ such that~$Q-l(Q-Q^{g})$ is shorter than~$Q$. Since~${Q-l(Q-Q^{g})\in T}$, this is not possible unless~$Q-Q^{g}= 0$. Therefore,~$Q=Q^{g}$, for all~$g\in G$, and so~${Q\in K\{ X_{i,j},1/\det(X_{i,j}) \}_{\partial_{1},\dots,\partial_{n}}}$. Since~$Q(Z_{i,j})=0$, we have~$Q\in J$. This yields the result. 
\end{trivlist}
\end{proof}

\pagebreak
\subsection{A result of descent for the local analytic parameterized differential Galois group.}\label{3sec22'}
We keep the notations of Section \ref{3sec1}. Let~$\pz Y(z,t)=A(z,t)Y(z,t)$, with~${A(z,t)\in \mathrm{M}_{m}\Big(\mathcal{O}_{U}(\{z\})\Big)}$, where~$U$ is a non empty polydisc of~$\C^{n}$ and~$\mathcal{O}_{U}(\{z\})$ has been defined in Page~\pageref{3p2}. 
\pagebreak[3]  
\begin{rem}
Note that~$\mathcal{O}_{U}(\{z\})$ is a ring but not a field in general.
 For example, if~${n=1}$,~${(z-t)^{-1}\notin \mathcal{O}_{U}(\{z\})}$. However, we have~${(z-t)^{-1}\in \mathcal{O}_{\C^{*}}(\{z\})}$. More generally let~${\a(z,t) \in   \mathcal{O}_{U}(\{z\})}$. For~$t\in U$, let~$R(t)$ minimal such that~$|\a(z,t)|\neq 0$ for~$0<|z|<R(t)$. There exists a non empty polydisc~$U'$ such that there exists~$\e>0$ with~$R(t)>\e$ on~$U'$. In particular, we have~$\a(z,t)^{-1}\in \mathcal{O}_{U'}(\{z\})$.
\end{rem}
Since~$\mathcal{O}_{U}(\{z\})\subset\hat{K}_{U}$, which is a field,~$\mathcal{O}_{U}(\{z\})$ is an integral domain, and we can define~$K_{U}$ as the fraction field of~$\mathcal{O}_{U}(\{z\})$. We have~$$\{ a \in K_{U} | \pz a =0 \}=\{ a \in \hat{K}_{U} | \pz a =0 \}=\mathcal{M}_{U}.$$ 
Let: ~$$F(z,t)=(F_{i,j})=\hat{H}(z,t) z^{L(t)}e\Big(Q(z,t)\Big)\in \mathrm{GL}_{m}\Big(\widehat{\textbf{K}_{U}}\Big), (\hbox{see } \S \ref{3sec11})~$$
be the fundamental solution given in Proposition \ref{3propo4}.
Let us denote~$K_{U}\langle F_{i,j} \rangle_{\pz,\dt}=\widetilde{K_{U}}$, which is a~$(\pz,\dt)$-differential subfield of~$\widehat{\textbf{K}_{U}}$. We have seen in~$\S$ \ref{3sec11}, that~$\widehat{\textbf{K}_{U}}$ has field of constants with respect to~$\pz$ equal to~$\mathcal{M}_{U}$. Then, we deduce that~$\widetilde{K_{U}}\Big|K_{U}$ is a parameterized Picard-Vessiot extension. Therefore, the results of~$\S \ref{3sec22}$ may be applied here; and we can define a parameterized differential Galois group~$Aut_{\pz}^{\dt}\left(\widetilde{K_{U}}\Big|K_{U}\right)$, which will be identified with a linear differential algebraic subgroup of~$\mathrm{GL}_{m}(\mathcal{M}_{U})$. We want to prove now that it is the ``same'' as the one of~$\S \ref{3sec21}$. \\ \par
Let~$C$ be a ($\dt$)-differentially closed field that contains~$\mathcal{M}_{U}$. Let us define~$C[[z]][z^{-1}]$, the~$(\pz,\dt)$-differential field, where~$z$ is a~$(\dt)$-constant with~$\pz z =1$,~$C$ is the field of constants with respect to~$\pz$, and~$\pz$ commutes with all the derivations. We define the ring~$K_{U}\otimes_{\mathcal{M}_{U}}C$ with the differential structure given by:
$$\forall a\in K_{U}, \forall c\in C, \forall \partial\in \{\pz,\dt\},\quad \partial (a\otimes_{\mathcal{M}_{U}}c)=\partial a\otimes_{\mathcal{M}_{U}}c+ a\otimes_{\mathcal{M}_{U}}\partial c.~$$
This~$(\pz,\dt)$-differential ring can be naturally embedded into~$C[[z]][z^{-1}]$, which implies that it is an integral domain. Therefore, we may define~$\mathcal{K}_{C,U}$, the field of fractions of~$K_{U}\otimes_{\mathcal{M}_{U}}C$. We see now~$\mathcal{K}_{C,U}$ (resp.~$K_{U}\otimes_{\mathcal{M}_{U}}C$) as a subfield (resp. subring) of~$C[[z]][z^{-1}]$. 
\pagebreak[3]
\begin{propo}\label{3propo7}
Let us keep the same notations. Let~$\pz Y(z,t)=A(z,t)Y(z,t)$, with~${A(z,t)\in \mathrm{M}_{m}\Big(\mathcal{O}_{U}(\{z\})\Big)}$. The extension field~$\mathcal{K}_{C,U}\langle F_{i,j} \rangle_{\pz,\dt}\Big|\mathcal{K}_{C,U}=\widetilde{\mathcal{K}_{C,U}}\Big| \mathcal{K}_{C,U}$ is a parameterized Picard-Vessiot extension for~$\pz Y(z,t)=A(z,t)Y(z,t)$. Moreover, there exist~$P_{1},\dots,P_{k}\in \mathcal{M}_{U}\{X_{i,j}\}_{\dt}$ such that the image of the representation of~$Gal_{\pz}^{\dt}\left(\widetilde{\mathcal{K}_{C,U}}\Big| \mathcal{K}_{C,U}\right)$ (resp.~$Aut_{\pz}^{\dt}\left(\widetilde{K_{U}}\Big|K_{U}\right)$) associated to~$F(z,t)$ is the set of~$C$-rational points (resp.~$\mathcal{M}_{U}$-rational points) of the linear differential algebraic subgroup of~$\mathrm{GL}_{m}(C)$ (resp.~$\mathrm{GL}_{m}(\mathcal{M}_{U})$) defined by~$P_{1},\dots,P_{k}$. More explicitly:
$$\begin{array}{ll}
&\left\{ F^{-1}\f(F), \f \in Gal_{\pz}^{\dt}\left(\widetilde{\mathcal{K}_{C,U}}\Big| \mathcal{K}_{C,U}\right)\right\}\\
=&\left\{ A=(a_{i,j}) \in \mathrm{GL}_{m}(C)\Big| P_{1}(a_{i,j})=\dots=P_{k}(a_{i,j})=0\right\}\\\\
&\left\{ F^{-1}\f(F), \f \in Aut_{\pz}^{\dt}\left(\widetilde{K_{U}}\Big|K_{U}\right)\right\}\\
=&\left\{ A=(a_{i,j}) \in \mathrm{GL}_{m}(\mathcal{M}_{U})\Big| P_{1}(a_{i,j})=\dots=P_{k}(a_{i,j})=0\right\}.
\end{array}~$$
\end{propo} 

\begin{proof}
We follow the proof of \cite{MS12}, Proposition 3.3. Let~$(d_{k})$ be an~$\mathcal{M}_{U}$-basis of~$C$. Let us prove that the~$d_{k}$ are linearly independent over~$\widetilde{K_{U}}$. Write~$\sum_{k\leq \kappa} d_{k}P_{k}=0$ with~${0\neq P_{k}\in\widetilde{K_{U}}}$,~$\kappa\geq 2$ minimal and~$P_{\kappa}=1$.  We have~$\sum_{k\leq \kappa-1} d_{k}\pz P_{k}=0$. If~$\kappa=2$, then~$\pz P_{1}=0$. If~$\kappa>2$, we have that for all~$k$,~$\pz P_{k}=0$, because of the minimality of~$\kappa$. Since~$\widetilde{K_{U}}\Big|K_{U}$ is a parameterized Picard-Vessiot extension, for all~$k$,~$P_{k}\in \mathcal{M}_{U}$, and the~$d_{k}$ are linearly independent over~$\widetilde{K_{U}}$.\par
Now, we prove that~$\mathcal{K}_{C,U}\langle F_{i,j} \rangle_{\pz,\dt}\Big|\mathcal{K}_{C,U}$ is a parameterized Picard-Vessiot extension for~${\pz Y(z,t)=A(z,t)Y(z,t)}$.  Let~$\a\in \mathcal{K}_{C,U}\langle F_{i,j} \rangle_{\pz,\dt}$ with~$\pz\a=0$. We may assume that~$\a=\sum d_{k}P_{k}$, where~$P_{k}\in \widetilde{K_{U}}$. We have~$\pz \a=\sum d_{k}\pz P_{k}= 0.$ Since the~$d_{k}$ are linearly independent over~$\widetilde{K_{U}}$, we find~$\pz P_{k}=0$. Hence,~$P_{k}\in \mathcal{M}_{U}$, because~$\widetilde{K_{U}}\Big|K_{U}$ is a parameterized Picard-Vessiot extension. 
  Therefore,~$\a\in C$ and~$\mathcal{K}_{C,U}\langle F_{i,j}\rangle_{\pz,\dt}\Big|\mathcal{K}_{C,U}$ is a parameterized Picard-Vessiot extension for~$\pz Y(z,t)=A(z,t)Y(z,t)$.\par
Let~$Y_{i,j}$ be a set of~$m^{2}$ indeterminates and let~$I_{0}$,~$I_{1}$  be~$(\pz,\dt)$-differential ideals such that:  
$$\begin{array}{ccccc}
R_{0}&=&K_{U}\{F_{i,j}\}_{\pz,\dt}&=&K_{U}\{Y_{i,j}\}_{\pz,\dt}/ I_{0} \\
R_{1}&=&\mathcal{K}_{C,U}\{F_{i,j}\}_{\pz,\dt}&=&\mathcal{K}_{C,U}\{Y_{i,j}\}_{\pz,\dt}/ I_{1}.
\end{array}$$
The group~$Aut_{\pz}^{\dt}\left(\widetilde{K_{U}}\Big|K_{U}\right)$ (resp.~$Gal_{\pz}^{\dt}\left(\widetilde{\mathcal{K}_{C,U}}\Big| \mathcal{K}_{C,U}\right)$) is the set of~$B\in \mathrm{GL}_{m}(\mathcal{M}_{U})$ (resp.~${B\in \mathrm{GL}_{m}(C)}$) such that~$(F_{i,j})B$ is again a zero of~$I_{0}$ (resp.~$I_{1}$).
We just have to prove that~$I_{1}=CI_{0}$. The inclusion~${CI_{0}\subset I_{1}}$ is clear. Let us prove the other inclusion. Let~$P\in I_{1}$. Without loss of generality, we may assume that~$P\in \left(K_{U}\otimes_{\mathcal{M}_{U}}C\right)[Y_{i,j}]$. Let  us write~$P=\sum d_{k}P_{k}$, where~$P_{k}\in K_{U}[Y_{i,j}]$. One finds that:
$$P(F_{i,j})=\sum d_{k}P_{k}(F_{i,j})=0.$$
Since the~$d_{k}$ are linearly independent over~$\widetilde{K_{U}}$, one finds, 
$P_{k}(F_{i,j})=0$, and therefore~${I_{1}=CI_{0}}$.\par 
\end{proof}

\pagebreak[3]
\subsection{An analogue of the density theorem in the parameterized case.}\label{3sec23}
Let us consider~${\pz Y(z,t)=A(z,t)Y(z,t)}$, with~$A(z,t)\in \mathrm{M}_{m}\Big(\mathcal{O}_{U}(\{z\})\Big)$, where~$U$ is a non empty polydisc of~$\C^{n}$.  We want to find topological generators for~$Aut_{\pz}^{\dt}\left(\widetilde{K_{U}}\Big|K_{U}\right)$ for Kolchin topology. \par
We now define the parameterized monodromy. The notion of monodromy in the unparameterized case is well explained in \cite{VdPS}. For more details about parameterized monodromy, see \cite{CS,MS12,MS13,Si90}.
\pagebreak[3]
\begin{defi}\label{3defi1}
The notations are introduced in~$\S$ \ref{3sec11}. We define~$\hat{m}$, the formal parameterized monodromy, as follows: 
\begin{itemize}
\item~$\forall \hat{H}(z,t) \in \hat{K}_{U},\hat{m}\Big(\hat{H}(z,t)\Big)=\hat{H}(z,t)~$.
\item~$\forall a(t) \in \mathcal{M}_{U}, \hat{m}(z^{a(t)})=e^{2i\pi a (t)}z^{a (t)}$.
\item~$\hat{m}(\log)=2i\pi + \log$.
\item For all~$ q(z,t)=\sum a_{n}z^{-n} \in \textbf{E}_{U}=\displaystyle \bigcup_{\n' \in \Q^{>0}}  z^{\frac{-1}{\n'}} \mathcal{M}_{U} \left[z^{\frac{-1}{\n'}}\right]$, we define~$$\hat{m}\Big(e(q(z,t))\Big)=e\Big(\sum a_{n}e^{-2i\pi n}z^{-n}\Big).$$
\end{itemize}

From the construction of~$\hat{K}_{U}\Big[\log, \left(z^{a(t)}\right)_{a(t) \in \mathcal{M}_{U}} \Big(e(q(z,t))\Big)_{q(z,t) \in \textbf{E}_{U}}\Big]$, it is easy to check that~$\hat{m}$ induces a well defined~$(\pz,\dt)$-differential ring automorphism of~$\hat{K}_{U}\Big[\log, \left(z^{a(t)}\right)_{a(t) \in \mathcal{M}_{U}} \Big(e(q(z,t))\Big)_{q(z,t) \in \textbf{E}_{U}}\Big]$, and then it can be extended as a~$(\pz,\dt)$-differential field automorphism of~$\widehat{\textbf{K}_{U}}$ letting~$K_{U}$ invariant. Since~$\widetilde{K_{U}}\subset \widehat{\textbf{K}_{U}}$, and since~$\widetilde{K_{U}}$ is stable under~$\hat{m}$,~$\hat{m}$ induces an element of~$Aut_{\pz}^{\dt}\left(\widetilde{K_{U}}\Big|K_{U}\right)$.
\end{defi}
\pagebreak[3]
\begin{rem}
 In the regular singular case with one singularity at~$0$, the definition  
of formal parameterized monodromy restricts to the definition given in \cite{MS12}.
\end{rem}

We now introduce the parameterized exponential torus, which is a subgroup of~$Aut_{\pz}^{\dt}\left(\widetilde{K_{U}}\Big|K_{U}\right)$ consisting of elements that act on the~$e\Big(q(z,t)\Big)$, with~$q(z,t)\in \textbf{E}_{U}$.
\pagebreak[3]
\begin{defi}
 Let~$\a$ be a character of~$\textbf{E}_{U}$. We define~$\tau_{\a}$ as the map
\begin{itemize}
 \item~$\tau_{\a}~$ is the identity on~$\hat{K}_{F,U}$.
 \item~$\forall q(z,t) \in \textbf{E}_{U},\tau_{\a} (e(q(z,t)))=\a(q(z,t))e(q(z,t))$.
\end{itemize}
From the construction of~$\hat{K}_{U}\Big[\log, \left(z^{a(t)}\right)_{a(t) \in \mathcal{M}_{U}} \Big(e(q(z,t))\Big)_{q(z,t) \in \textbf{E}_{U}}\Big]$, it is easy to check that~$\tau_{\a}$ induces a well defined~$(\pz,\dt)$-differential ring automorphism of~$\hat{K}_{U}\Big[\log, \left(z^{a(t)}\right)_{a(t) \in \mathcal{M}_{U}} \Big(e(q(z,t))\Big)_{q(z,t) \in \textbf{E}_{U}}\Big]$, and then it can be extended to a~$(\pz,\dt)$-differential field automorphism of~$\widehat{\textbf{K}_{U}}$ letting~$K_{U}$ invariant. Since~$\widetilde{K_{U}}\subset \widehat{\textbf{K}_{U}}$, and since~$\widetilde{K_{U}}$ is stable under~$\tau_{\a}$,~$\tau_{\a}$ induces an element of~$Aut_{\pz}^{\dt}\left(\widetilde{K_{U}}\Big|K_{U}\right)$. \par 
 The  parameterized exponential torus (or simply, the exponential torus) is the subgroup of~$Aut_{\pz}^{\dt}\left(\widetilde{K_{U}}\Big|K_{U}\right)$ consisting of the~$\tau_{\a}$, where~$\a$ is a character of~$\textbf{E}_{U}$. Notice that the matrices of the exponential torus belongs to~$\mathrm{GL}_{m}(\C)$, while the coefficients of the matrix of~$\hat{m}$ depend upon~$t$.
\end{defi}
\pagebreak[3]
\begin{ex}\label{3ex4}
Let~$t=(t_{1},t_{2})$ and let us consider~$$ \pz\begin{pmatrix}
 Y_{1}(z,t)   \\ 
 Y_{2}(z,t)
\end{pmatrix}=
\begin{pmatrix}
-t_{1}z^{-2} & 0 \\ 
0 &-t_{2}z^{-2} 
                              \end{pmatrix} 
\begin{pmatrix}
 Y_{1}(z,t)   \\ 
 Y_{2}(z,t)
\end{pmatrix},$$ 
which admits~$\begin{pmatrix}
e^{t_{1}/z} & 0 \\ 
0 &e^{t_{2}/z}
                              \end{pmatrix}~$
as fundamental solution. The parameterized exponential torus and the parameterized differential Galois group are both equal to 
$$\left\{ \begin{pmatrix}
\a & 0 \\ 
0 &\b \end{pmatrix}, \hbox{ where } \a,\b \in \C^{*}\right\}.$$
Remark that the unparameterized exponential torus (see p.80 of \cite{VdPS}) and the unparameterized differential Galois group are isomorphic to~$\left(\C^{*}\right)^{2}$ if and only if~$t_{1}$ and~$t_{2}$ are linearly independant over $\Q$. In particular, the matrices of the parameterized exponential torus evaluated at a specialized value~$(u,v)$ of the parameter are not always equal to the matrices of the unparameterized exponential torus of the system~$$\pz \begin{pmatrix}
 Y_{1}(z,u,v)   \\ 
 Y_{2}(z,u,v)
\end{pmatrix}=
\begin{pmatrix}
-uz^{-2} & 0 \\ 
0 &-vz^{-2} 
                              \end{pmatrix} 
\begin{pmatrix}
 Y_{1}(z,u,v)   \\ 
 Y_{2}(z,u,v)
\end{pmatrix}.$$ This is a difference between the exponential torus and the two other generators of the parameterized differential Galois group: the monodromy 	and the Stokes operators (see Definition \ref{3defi6} below).
\end{ex}
\pagebreak[3]
\begin{lem}\label{3lem5}
Let~$d(t)$ be a singular direction of~$\pz Y(z,t)=A(z,t)Y(z,t)$ (see~$\S \ref{3sec14}$). The Stokes matrix~$St_{d(t)}$ induces an element of~$Aut_{\pz}^{\dt}\left(\widetilde{K_{U}}\Big|K_{U}\right)$.
\end{lem}

\begin{proof} 
Let us recall the construction of the Stokes matrices. Let~$d(t)$ be a singular direction and let~$k_{r}$ be the biggest level of~$\pz Y(z,t)=A(z,t)Y(z,t)$.  The assumption we have made on~$\mathcal{D}$ (see~$\S \ref{3sec14}$) tells us that there exists~$ t\mapsto d^{\pm}(t)$ continuous in~$t$ such that ~$$d(t)-\frac{\pi}{2k_{r}}<d^{-}(t)<d(t)<d^{+}(t)<d(t)+\frac{\pi}{2k_{r}},$$
with no singular directions in~$[d^{-}(t),d(t)[\bigcup ]d(t),d^{+}(t)]$.
From the construction of~$St_{d(t)}$, and~$\S \ref{3sec13}$, we know that 
$$H^{d^{+}(t)}(z,t)e^{L(t)\log(z)}e^{Q(z,t)}= H^{d^{-}(t)}(z)e^{L(t)\log(z)}e^{Q(z,t)}St_{d(t)}.$$
By construction, the Stokes matrix induces identity on~$K_{U}$. 
To prove that the Stokes matrices are elements of~$Aut_{\pz}^{\dt}\left(\widetilde{K_{U}}\Big|K_{U}\right)$, we have to prove that the maps~$i^{\pm}$, that send~$\hat{H}(z,t)z^{L(t)}e\Big(Q(z,t)\Big)$ to~$H^{d^{\pm}(t)}(z,t)e^{L(t)\log(z)}e^{Q(z,t)}$, induce~$(\pz,\dt)$-field isomorphisms. From the unparameterized case (see Theorem 2, Section 6.4 of \cite{B}), and the relations satisfied by the symbols~$\log$,~$\left(z^{a(t)}\right)_{a(t) \in \mathcal{M}_{U}}$ and~$\Big(e(q(z,t))\Big)_{q(z,t) \in \textbf{E}_{U}}$ (see~$\S \ref{3sec11}$),~$i^{\pm}$ induce~$\pz$-field isomorphisms.\par
We want now to prove that if~$\hat{H}(z,t)$ admits,~$H^{d^{\pm}(t)}(z,t)$ as asymptotic sum in the direction~$d^{\pm}(t)$, then~$\partial_{t_{i}}\hat{H}(z,t)$ admits~$\partial_{t_{i}}H^{d^{\pm}(t)}(z,t)$ as asymptotic sum in the direction~$d^{\pm}(t)$, for all~$i\leq n$. This is a consequence of Lemma \ref{3lem3} and the fact that we may assume that the~$d^{\pm}(t)$ are locally constant. Hence~$i^{\pm}$ commute with~$\partial_{t_{i}}$ and~$i^{\pm}$ induce~$(\pz,\dt)$-field isomorphisms.
\end{proof}
\pagebreak[3]
\begin{defi}\label{3defi6}
Let~$d(t)$ be a singular direction of~$\pz Y(z,t)=A(z,t)Y(z,t)$. The element of~$Aut_{\pz}^{\dt}\left(\widetilde{K_{U}}\Big|K_{U}\right)$ induced by the Stokes matrix in the direction~$d(t)$ is the Stokes operator in the direction~$d(t)$. For simplicity of notation, we write~$St_{d(t)}$ for both the Stokes operator and the Stokes matrix in the direction~$d(t)$.
\end{defi}
\pagebreak[3]
\begin{propo}\label{3propo5}
 If~$g(z,t) \in  \widetilde{K_{U}}$ is fixed by all the Stokes operators~$St_{d(t)}$, the monodromy and the exponential torus, then~$g(z,t)\in K_{U}$.
\end{propo}

\begin{proof}
Let~$\overline{\mathcal{M}_{U}}$ be the algebraic closure of~$\mathcal{M}_{U}$.  Proposition 3.25 of \cite{VdPS} implies that if~${g(z,t)\in \widehat{\textbf{K}_{U}}}$ is fixed by the monodromy and the exponential torus, then~${g(z,t)\in \widehat{\textbf{K}_{U}}\cap \overline{\mathcal{M}_{U}}\big[\big[z\big]\big]\big[z^{-1}\big]=\hat{K}_{U}}$. Since~${\widetilde{K_{U}}\subset\widehat{\textbf{K}_{U}}}$, we have to prove that if~$g(z,t)\in \widetilde{K_{U}}\cap \hat{K}_{U}$ is fixed by all the Stokes operators, then~$g(z,t)\in K_{U}$. Let~${g(z,t)\in \widetilde{K_{U}}\cap \hat{K}_{U}}$ be fixed by all the Stokes operators.  Let~${F(z,t)=\hat{H}(z,t) z^{L(t)}e\Big(Q(z,t)\Big)}$ be the fundamental solution defined in Proposition~\ref{3propo4} and let~$\Big(\hat{H}_{i,j}\Big)$ be the entries of the matrix~$\hat{H}(z,t)$. There exists~$P\in K_{U}\langle X_{i,j}\rangle_{\pz,\dt}$ such that~$P(\hat{H}_{i,j})=g(z,t)$.  
Let~$d(t)$ that satisfies the same properties as in Proposition \ref{3propo3}.  Because of Proposition \ref{3propo3}, there exists a map ~$U\rightarrow \R^{>0}$,~$t\mapsto\e(t)$, which is not necessarily continuous such that~$P\Big(H^{d(t)}_{i,j}\Big)$ is meromorphic in~$(z,t)$ for 
$$ (z,t)\in \left\{ z\in \widetilde{\C}\Big|\arg(z)\in \left]d_{1}(t)-\frac{\pi}{2k_{r}},d_{2}(t)+\frac{\pi}{2k_{r}}\right[\hbox{ and } 0<|z|<\e(t) \right\}\times U,$$
where~$d_{1}(t),d_{2}(t)$ are two singular directions. Since~$P\Big(\hat{H}_{i,j}\Big)$ is fixed by all the Stokes operators,~$P\Big(H^{d(t)}_{i,j}\Big)$  is meromorphic in~$(z,t)$ for~$0<|z|<\e(t)$ and~${(z,t)\in  \widetilde{\C} \times U}$.
Moreover,~${P\Big(H^{d(t)}_{i,j}\Big)(z,t)=P\Big(H^{d(t)}_{i,j}\Big)(e^{2i\pi}z,t)}$ on its domain of definition, which means that~$P\Big(H^{d(t)}_{i,j}\Big)$  is meromorphic in~$(z,t)$ for~$0<|z|<\e(t)$ and~${(z,t)\in  \C \times U}$. We recall that~$K_{U}$ consists of elements~${f(z,t)\in \hat{K}_{U}}$ such that for~${0<|z|<\e(t)}$,~${t\mapsto f(z,t)\in \mathcal{M}_{U}}$. We have,~$P(H^{d(t)}_{i,j})\in K_{U}$. We have seen in Lemma~\ref{3lem5} that the map that sends~$\hat{H}(z,t)z^{L(t)}e\Big(Q(z,t)\Big)$ to~$H^{d(t)}(z,t)e^{L(t)\log(z)}e^{Q(z,t)}$ induces a~$(\pz,\dt)$-field isomorphism. Since the above map leaves~$K_{U}$ invariant, this implies that~${P\Big(\hat{H}_{i,j}\Big)=g(z,t)\in K_{U}}$.
\end{proof}

We can now prove the main theorem of this paper. We recall some notations. Let~${\pz Y(z,t)=A(z,t)Y(z,t)}$, with~${A(z,t) \in \mathrm{M}_{m}\Big(\mathcal{O}_{U}(\{z\})\Big)}$ (see page~$ \pageref{3p2}$), let~$K_{U}$ be the fraction field of~$\mathcal{O}_{U}(\{z\})$, and let~$\widetilde{K_{U}}\Big|K_{U}$ be the parameterized Picard-Vessiot extension defined in the beginning of~$\S \ref{3sec23}$. Let~$Aut_{\pz}^{\dt}\left(\widetilde{K_{U}}\Big|K_{U}\right)$ be the field automorphisms of~$\widetilde{K_{U}}$ which commute with all the derivations and leave~$K_{U}$ invariant.
\pagebreak[3]
\begin{theo}[Parameterized analogue of the density theorem of Ramis]\label{3theo2}
 The group generated by the monodromy, the exponential torus and the Stokes operators is dense for the Kolchin topology in~$Aut_{\pz}^{\dt}\left(\widetilde{K_{U}}\Big|K_{U}\right)$.
\end{theo}

\begin{proof}
First of all, we have already pointed out that the monodromy, the exponential torus and the Stokes operators are elements of~$Aut_{\pz}^{\dt}\left(\widetilde{K_{U}}\Big|K_{U}\right)$. Using Proposition \ref{3propo8}, we just have to prove that if~$\a(z,t)\in \widetilde{K_{U}}$ is fixed by the monodromy, the exponential torus and the Stokes operators, then it belongs to~$K_{U}$. This is exactly Proposition  \ref{3propo5}.
\end{proof}
\pagebreak[3]
\begin{rem}
\begin{trivlist}
\item (1)
Let~$\C(t)\{ z\}$ be the subset of~$\mathcal{O}_{U}(\{z\})$ consisting of elements of the form~$\sum_{i>N} a_{i}(t)z^{i}$, with~$a_{i}(t)\in \C(t)$ and~$N\in \Z$. Let us consider~$\pz Y(z,t)=A(z,t)Y(z,t)$, with~${A(z,t) \in \mathrm{M}_{m}\Big(\C(t)\{ z\}\Big)}$. Even if we were able to define a parameterized Picard-Vessiot extension over~$\C(t)\{ z\}$, we would not have a parameterized analogue of the density theorem of Ramis, because the monodromy is not defined in this case. In general, we have~$$\hat{m}(z^{\a (t)})=e^{2i\pi \a(t)}z^{\a (t)} \notin \C(t)\{ z\}(z^{\a (t)}).$$
 This is why we take a larger field of constants with respect to~$\pz$. 
\item  (2) Similarly, we can prove that the group generated by the monodromy and the exponential torus is dense for Kolchin topology in~$Aut_{\pz}^{\dt}\left( \widetilde{K_{U}}\Big|\hat{K}_{U}\cap \widetilde{K_{U}}\right)$.
\end{trivlist}
\end{rem}
\pagebreak[3]
\begin{coro}\label{3coro2}
$Aut_{\pz}^{\dt}\left(\widetilde{K_{U}}\Big|K_{U}\right)$ contains a finitely generated Kolchin-dense subgroup.
\end{coro}

\begin{proof}
Let~$q_{1} (z,t),\dots,q_{\b}(z,t) \in \textbf{E}_{U}$ be~$\Q$-linearly independent such that 
$$\widetilde{K_{U}}\subset \hat{K}_{F,U}\Big(e(q_{1}(z,t)),\dots,e(q_{\b}(z,t))\Big).$$
Let~$\tau_{i}$ be an element of the exponential torus that fixes the~$e(q_{j}(z,t))$ for~$j\neq i$, and that sends~$e(q_{i}(z,t))$ to~$a e(q_{i}(z,t))$, with~$a$ not a root of unity. \par 
By the definition of the singular directions (see~$\S \ref{3sec14}$), there exists~$\nu\in \N^{*}$ such that the singular directions modulo~$2 \nu \pi$ are in finite number. Let~$d_{1}(t),\dots,d_{k}(t)$ be continuous singular directions such that, if~$d(t)$ is a singular direction, then~$d(t)$ is equal to one of the~$d_{i}(t)$ modulo~$2 \nu \pi$.
Let~$g(z,t)\in \widetilde{K_{U}}$ be fixed by the monodromy,~$\tau_{1},\dots,\tau_{\b}$, and~$St_{d_{1}(t)},\dots,St_{d_{k}(t)}$. 
Using (2) of Proposition \ref{3propo8}, it is sufficient to prove that~${g(z,t)\in K_{U}}$.\par
We can write~$g(z,t)$ as an element of:
$$\hat{K}_{F,U}\Big(e(q_{1}(z,t)),\dots,e(q_{\b-1}(z,t)))(e(q_{\b}(z,t))\Big).$$
Since the~$q_{i}(z,t)\in \textbf{E}_{U}$ are~$\Q$-linearly independent, we know by construction that the ~$e(Nq_{\b}(z,t))$, with~$N\in \Z$, are~$\C$-linearly independent over~$\hat{K}_{F,U}\Big(e(q_{1}(z,t)),\dots,e(q_{\b-1}(z,t))\Big)$. If we add the fact that~$g (z,t)$ is fixed by~$\tau_{\b}$, we obtain:~$$ g(z,t) \in\hat{K}_{F,U}\Big(e(q_{1}(z,t)),\dots,e(q_{\b-1}(z,t))\Big).$$
We apply the same argument~$\b$ times to conclude that~$g (z,t) \in  \hat{K}_{F,U}\cap \widetilde{K_{U}}$.
By the construction of the Stokes operators, we have that~$St_{d(t)}=\mathrm{Id}$ if and only if~$St_{2\nu \pi +d(t)}=\mathrm{Id}$, where~$\nu\in \N^{*}$ has been defined in the proof. Proposition \ref{3propo5} allows us to conclude that~$g(z,t)\in K_{U}$.
\end{proof}

\pagebreak[3]
\subsection{The density theorem for the global parameterized differential Galois group.}\label{3sec24}

 In this subsection, we consider parameterized linear  differential equation of the form:
$$
\pz Y(z,t)=A(z,t)Y(z,t),
$$
 with~$A(z,t)\in \mathrm{M}_{m}\Big(\mathcal{M}_{U}(z)\Big)$. We want to prove a density theorem for the global parameterized differential Galois group. The result in the unparameterized case is due to Ramis and a proof can be found for instance in \cite{M}, Proposition 1.3. The parameterized singularities of~$\pz Y(z,t)=A(z,t)Y(z,t)$ (that is the poles, including maybe~$\infty$, of~$A(z,t)$, as a rational function in~$z$) belong to the algebraic closure of~$\mathcal{M}_{U}$. Because of Remark~\ref{3rem4}, after taking a smaller non empty polydisc~$ U$, we may assume that the set of parameterized singularities belongs to~$\mathcal{M}_{U}$. We will write singularity instead of  parameterized singularity when no confusion is likely to arise. 
Let~$S=\{ \a_{1}(t),\dots,\a_{k}(t)\}\subset \mathbb{P}_{1}(\mathcal{M}_{U})$ be the set of the singularities of~$\pz Y(z,t)=A(z,t)Y(z,t)$. 
  For any singularity~$\a(t)$ of~${\pz Y(z,t)=A(z,t)Y(z,t)}$, we may define the levels and the set of singular directions of~$\a(t)$ by considering~$$\pz Y(z-\a(t),t)=A(z-\a(t),t)Y(z-\a(t),t) \hbox{ if }\infty \not\equiv \a(t)\in S$$ and~$$\pz Y(z^{-1},t)=A(z^{-1},t)Y(z^{-1},t) \hbox{ if } \infty \equiv \a(t)\in S.$$ Let~$(d_{i,j}(t))$ be the singular directions~$\a_{i}(t)$. As in~$\S \ref{3sec14}$, we define:
$$\mathcal{D}_{\a_{i}(t)}=\left\{t\in U \Big|\exists j, j'\in \N, \hbox{ such that } d_{i,j} \not\equiv d_{i,j'}\hbox{ and }d_{i,j}(t)=d_{i,j'}(t)\right\}.~$$
From Lemma \ref{3lem9}, all the~$\mathcal{D}_{\a_{i}(t)}$ are closed set with empty interior.
After taking a smaller non empty polydisc~$U$, we may assume that: 
\begin{itemize} 
\item
There exists~$\e >0$ such that for all~$t\in U$ and for all~$i\neq j$:~$$|\a_{i}(t)-\a_{j}(t)|>\e.$$ 
\item~$\mathcal{D}_{\a_{i}(t)}=\varnothing$ for all~$i\leq k$.
\item For all the singularities of~$\pz Y(z,t)=A(z,t)Y(z,t)$, the levels are independent of~$t$.
\item For all~$t_{0}\in U$, for all the singularities~$\infty\not \equiv \a(t)\in S$ (resp. for the singularity~$\infty$), the singular directions of~$\pz Y(z-\a(t),t)=A(z-\a(t),t)Y(z-\a(t),t)$ (resp.~${\pz Y(z^{-1},t_{0})=A(z^{-1},t_{0})Y(z^{-1},t_{0})}$) evaluated at~$t_{0}$ are equal to the singular directions of the specialized system~$\pz Y(z-\a(t),t_{0})=A(z-\a(t),t_{0})Y(z-\a(t),t_{0})$ (resp.~${\pz Y(z^{-1},t_{0})=A(z^{-1},t_{0})Y(z^{-1},t_{0})}$).
\item Every entry of every~$z$-coefficient of~$A(z,t)$ is analytic on~$U$.
\end{itemize}
Let~$x_{0}(t)\in \mathcal{M}_{U}$ and let~$\e>0$ such that:
$$ \forall t\in U, \forall i<j\leq k,\quad |x_{0}(t)-\a_{j}(t)|>\e \hbox{ and }|\a_{i}(t)-\a_{j}(t)|>\e.$$
For all~$i\leq k$ and all~$t\in U$, we define~$U_{\a_{i}(t)}$, the polydisc in the~$z$-plane with center~$\a_{i}(t)$ and with radius~$\e$. Let~$d_{\a_{i}}(t)$ be a continuous ray from~$\a_{i}(t)$ in~$U_{\a_{i}(t)}$,~$b_{\a_{i}}(t)$ be the continuous point of~$d_{\a_{i}}(t)$ with~$|b_{\a_{i}}(t)-\a_{i}(t)|=\e$ and~$\g_{\a_{i}}(t)$ be a continuous path in~$\mathbb{P}_{1}(\mathcal{M}_{U})$ from~$x_{0}(t)$ to~$b_{\a_{i}}(t)$ such that for all~$t\in U$ and all~$j\leq k$,~$|\g_{\a_{i}}(t)-\a_{j}(t)|>\e/2$. Analytic continuation of~$F(z,t)=(F_{i,j})$, a germ of solution at~$x_{0}(t)$ with the path~$\g_{\a_{i}}(t)$ and~$d_{\a_{i}}(t)$ provides a fundamental solution~$F^{d_{\a_{i}}(t)}(z,t)$ on a germ of open sector with vertex~$\a_{i}(t)$ bisected by~$d_{\a_{i}}(t)$. \par
Let~$\widetilde{\mathcal{M}_{U}(z)}=\mathcal{M}_{U}(X)\langle F_{i,j} \rangle_{\pz,\dt}$. From the assumptions we have made on~$x_{0}(t)$, we deduce that this field has a field of constants with respect to~$\pz$ equal to~$\mathcal{M}_{U}$. Therefore, we deduce that~$\widetilde{\mathcal{M}_{U}(z)}\Big| \mathcal{M}_{U}(z)$ is a parameterized Picard-Vessiot extension. The results of~$\S \ref{3sec22}$ may be applied here and we can define a parameterized differential Galois group~$Aut_{\pz}^{\dt}\left(\widetilde{\mathcal{M}_{U}(z)}\Big| \mathcal{M}_{U}(z)\right)$, which will be identified with a linear differential algebraic subgroup of~$\mathrm{GL}_{m}(\mathcal{M}_{U})$. We will make the same abuse of language as in the local case (see Remark \ref{3rem6}) and call it the parameterized linear differential Galois group, or Galois group, if no confusion is likely to arise. As in Proposition \ref{3propo7}, we want to prove now that it is the ``same'' as the one of~$\S \ref{3sec21}$. \par
Let~$C$ be a ($\dt$)-differentially closed field that contains~$\mathcal{M}_{U}$ and let~$C(z)$ denote the~$(\pz,\dt)$-differential field of rational functions in the indeterminate~$z$, with coefficients in~$C$, where~$z$ is a~$(\dt)$-constant with~$\pz z =1$,~$C$ is the field of constants with respect to~$\pz$, and~$\pz$ commutes with all the derivations. The next proposition is the analogue in the global case of  Proposition \ref{3propo7}.
\pagebreak[3]
\begin{propo}\label{3propo2}
Let us keep the same notations. Let~$\pz Y(z,t)=A(z,t)Y(z,t)$, with~${A(z,t)\in \mathrm{M}_{m}(\mathcal{M}_{U}(z))}$. The extension field~$C(z)\langle F_{i,j} \rangle_{\pz,\dt}\Big|C(z):=\widetilde{C(z)}\Big|C(z)$ is a parameterized Picard-Vessiot extension for~$\pz Y(z,t)=A(z,t)Y(z,t)$. Moreover, there exist~$P_{1},\dots,P_{k}\in \mathcal{M}_{U}\{X_{i,j}\}_{\dt}$ such that the image of the representation of~$Gal_{\pz}^{\dt}\left(\widetilde{C(z)}\Big| C(z)\right)$ (resp.~$Aut_{\pz}^{\dt}\left(\widetilde{\mathcal{M}_{U}(z)}\Big| \mathcal{M}_{U}(z)\right)$) associated to~$F(z,t)$ is the set of~$C$-rational points (resp.~$\mathcal{M}_{U}$-rational points) of the linear differential algebraic subgroup of~$\mathrm{GL}_{m}(C)$ (resp.~$\mathrm{GL}_{m}(\mathcal{M}_{U})$) defined by~$P_{1},\dots,P_{k}$. More explicitly:
$$\begin{array}{ll}
&\left\{ F^{-1}\f(F), \f \in Gal_{\pz}^{\dt}\left(\widetilde{C(z)}\Big| C(z)\right)\right\}\\
=&\left\{ A=(a_{i,j}) \in \mathrm{GL}_{m}(C)\Big| P_{1}(a_{i,j})=\dots=P_{k}(a_{i,j})=0\right\}\\\\
&\left\{ F^{-1}\f(F), \f \in Aut_{\pz}^{\dt}\left(\widetilde{\mathcal{M}_{U}(z)}\Big| \mathcal{M}_{U}(z)\right)\right\}\\
=&\left\{ A=(a_{i,j}) \in \mathrm{GL}_{m}(\mathcal{M}_{U})\Big| P_{1}(a_{i,j})=\dots=P_{k}(a_{i,j})=0\right\}.
\end{array}~$$
\end{propo}

\begin{proof}
This is exactly the same reasoning as in Proposition \ref{3propo7}.
\end{proof}

We want to find topological generators for~$Aut_{\pz}^{\dt}\left(\widetilde{\mathcal{M}_{U}(z)}\Big| \mathcal{M}_{U}(z)\right)$ for the Kolchin topology. \par
For~$\a(t)\in \mathcal{M}_{U}$, let~$$K_{U,\a(t)}=\{ f(z-\a(t),t) \; | \; f(z,t)\in K_{U} \},$$  and let~$$K_{U,\infty}=\{ f(z^{-1},t) \; | \; f(z,t)\in K_{U} \}.$$
Let~$\a(t)\in S$ and let~$Aut_{\pz}^{\dt}\left(\widetilde{\mathcal{M}_{U}(z)}\Big|K_{U,\a(t)}\cap\widetilde{\mathcal{M}_{U}(z)}\right)$ be the local Galois group for the fundamental solution~$F^{d_{\a}(t)}(z,t)$ described above. If we conjugate~$Aut_{\pz}^{\dt}\left(\widetilde{\mathcal{M}_{U}(z)}\Big|K_{U,\a(t)}\cap \widetilde{\mathcal{M}_{U}(z)}\right)$ by the differential isomorphism defined by analytic continuation of~$F(z,t)$ described above, we get an injective morphism of linear differential algebraic groups:~$$Aut_{\pz}^{\dt}\left(\widetilde{\mathcal{M}_{U}(z)}\Big|K_{U,\a(t)}\cap\widetilde{\mathcal{M}_{U}(z)}\right)\hookrightarrow Aut_{\pz}^{\dt}\left(\widetilde{ \mathcal{M}_{U}(z)}\Big| \mathcal{M}_{U}(z)\right).$$ 
Using the maps~$i^{\pm}$ defined in the proof of Lemma \ref{3lem5} and the injection above, we can define the monodromy, the exponential torus, and the Stokes operators for any singularities in ~$S$, as elements of~$$Aut_{\pz}^{\dt}\left(\widetilde{\mathcal{M}_{U}(z)}\Big|\mathcal{M}_{U}(z)\right).$$
\pagebreak[3]
\begin{theo}[Global parameterized analogue of the density theorem of Ramis]\label{3theo1}
Let ~${\pz Y(z,t)=A(z,t)Y(z,t)}$, where~${A(z,t)\in \mathrm{M}_{m}(\mathcal{M}_{U}(z))}$. For~$\a(t)\in S$, let~$G_{\a(t)}$ be the subgroup of:
$$Aut_{\pz}^{\dt}\left(\widetilde{\mathcal{M}_{U}(z)}\Big|K_{U,\a(t)}\cap\widetilde{\mathcal{M}_{U}(z)}\right),$$ 
generated by the monodromy, the exponential torus and the Stokes operators. Let~$G$ be the subgroup of~$Aut_{\pz}^{\dt}\left(\widetilde{\mathcal{M}_{U}(z)}\Big|\mathcal{M}_{U}(z)\right)$ generated by the~$G_{\a(t)}$, with~$\a (t)\in S$. Then~$G$ in dense for Kolchin topology in~$$Aut_{\pz}^{\dt}\left(\widetilde{\mathcal{M}_{U}(z)}\Big|\mathcal{M}_{U}(z)\right).$$ 
\end{theo}

\begin{proof}
We use (2) of Proposition \ref{3propo8}. We have to prove that the subfield of~$\widetilde{\mathcal{M}_{U}(z)}$ fixed by~$G$ is~$\mathcal{M}_{U}(z)$. Let~$f(z,t)\in \widetilde{\mathcal{M}_{U}(z)}$ be fixed by~$G$. Then, by the same reasoning as in Proposition \ref{3propo5}, it follows that~$f(z,t)$ belongs to~$K_{U,\a(t)}$, for~$\a(t)\in S$. Therefore, we deduce that~$ f(z,t)$ is meromorphic in~$(z,t)$ on~$\mathbb{P}_{1}(\C)\times U$, and has a finite number of poles in the~$z$-plane for~$t$ fixed. Hence, 
~$f(z,t)\in \mathcal{M}_{U}(z)~$.
\end{proof}

In particular, this generalizes Theorem 4.2 in \cite{MS12} which says that, if the equation has only regular singular poles, then the group generated by the monodromy at each pole is dense for Kolchin topology in~$Aut_{\pz}^{\dt}\left(\widetilde{\mathcal{M}_{U}(z)}\Big|\mathcal{M}_{U}(z)\right)$.  
\pagebreak[3]
\begin{coro}\label{3coro1}
$Aut_{\pz}^{\dt}\left(\widetilde{\mathcal{M}_{U}(z)}\Big|\mathcal{M}_{U}(z)\right)$  contains a finitely generated Kolchin-dense subgroup.
\end{coro}

\begin{proof}
In the proof of Theorem \ref{3theo1}, we see that the global parameterized differential Galois group is generated by all local parameterized differential Galois groups. Since there is a finite number of singularities, this is a consequence of Corollary \ref{3coro2}. 
\end{proof}

\pagebreak
\subsection{Examples.}\label{3sec25}

In all the examples, we will compute the global parameterized differential Galois group. This means that the base field is~$\mathcal{M}_{U}(z)$.
\pagebreak[3]
\begin{ex}\label{3ex1}
 Let us consider~$\pz Y(z,t)=\frac{t}{z}Y(z,t)$. This example was considered by direct computations in Example \ref{3ex3} but we will compute here~$Aut_{\pz}^{\dt}\left(\widetilde{\mathcal{M}_{U}(z)}\Big|\mathcal{M}_{U}(z)\right)$ using the parameterized density theorem. The fundamental solution is~$(z^{t})$ and the parameterized Picard-Vessiot extension over~$\mathcal{M}_{U}(z)$ is~$\mathcal{M}_{U}(z,z^{t},\log)$ (we want the extension to be closed under the derivations~$\pz$ and~$\pt$). The exponential torus and the Stokes matrices are trivial. The monodromy sends~$z^{t}$ to~$e^{2i\pi t}z^{t}$. The element~$e^{2i\pi t}$ satisfies the differential equation
$$\pt \left(\frac{\pt e^{2i\pi t}}{e^{2i\pi t}}\right)=0.$$
Therefore, the Kolchin closure of the monodromy is contained in:
$$\left\{ a\in \mathcal{M}_{U} \Big|\pt \left(\frac{\pt a}{a}\right) \right\}=\{ ce^{bt}, b \in \C, c\in \C^{*} \}.$$
Conversely, the map that sends~$z^{t}$ to~$ce^{bt}z^{t}$ is an element of~$Aut_{\pz}^{\dt}\left(\widetilde{\mathcal{M}_{U}(z)}\Big|\mathcal{M}_{U}(z)\right)$. Finally, viewed as a linear differential algebraic subgroup of~$\mathrm{GL}_{1}(\mathcal{M}_{U})$, 
$$\begin{array}{lll}
Aut_{\pz}^{\dt}\left(\widetilde{\mathcal{M}_{U}(z)}\Big|\mathcal{M}_{U}(z)\right) & \simeq & \left\{ a\in \mathcal{M}_{U} \Big| \pt \left(\frac{\pt a}{a}\right)=0\right\}\\\\
&= &\{a\in \mathcal{M}_{U} |   a\neq 0 \hbox{ and } a\pt^{2}a -(\pt a)^{2}=0\}\\\\
&\subseteq & \mathrm{GL}_{1}(\mathcal{M}_{U}).\\
\end{array}$$
\end{ex}
\pagebreak[3]
\begin{ex}[Parameterized Euler equation]
Let~$f(t)$ be an analytic function different from~$0$, and let us consider:~$$ 
\pz^{2} Y(z,t)+\left(\frac{1}{z}-\frac{1}{f(t)z^{2}}\right)\pz Y(z,t)+\frac{1}{f(t)z^{3}} Y(z,t)=0,$$
which can be seen as a system:
$$ \pz\begin{pmatrix}
 Y(z,t)   \\ 
\pz Y(z,t)
\end{pmatrix}=
\begin{pmatrix}
0 & 1 \\ 
\frac{-1}{f(t)z^{3}} &\frac{1}{f(t)z^{2}}-\frac{1}{z}
                              \end{pmatrix} 
\begin{pmatrix}
Y(z,t)   \\ 
\pz Y(z,t)
\end{pmatrix}.$$
If~$f\equiv 1$, we recognize the Euler equation. A fundamental solution is:~$$\begin{pmatrix}
1 & \hat{F}(z,t) \\ 
\frac{1}{f(t)z^{2}} &\pz \hat{F}(z,t) \end{pmatrix}\begin{pmatrix}
e\left(\frac{-1}{f(t)z}\right) &0 \\ 
0 &1 \end{pmatrix},$$ where~$\hat{F}(z,t) =-\displaystyle \sum_{n \geq 0} n! (f(t)z)^{n+1}$. The only  singularity is~$0$. The monodromy is trivial. Let~$\tau$ be an element of the exponential torus. Then, the image of the fundamental solution under~$\tau$ is:
$$\begin{pmatrix}
1 & \hat{F}(z,t) \\ 
\frac{1}{f(t)z^{2}} &\pz \hat{F}(z,t) \end{pmatrix}\begin{pmatrix}
\a e\left(\frac{-1}{f(t)z}\right) &0 \\ 
0 &1 \end{pmatrix},$$
with~$\a \in \C^{*}$.
 Therefore, the matrices of the elements of the exponential torus are of the form~$\mathrm{Diag}(\a,1)$, with~$\a \in \C^{*}$.  The only level of the system is~$1$ and the singular directions are the~$\arg\left(f(t)^{-1}\right)+2k\pi$, with~$k\in\Z$.
As we have seen in Proposition \ref{3propo1}, we can compute the Stokes matrix with the Laplace and the Borel transforms. It follows from the definition of the formal Borel transform that~$$\hat{\mathcal{B}}_{1}\Big(\hat{F}(z,t)\Big)\equiv\log(1-f(t)z).$$ Let~$0<\e <\frac{\pi}{2}$ be such that there are no singular directions in:~$$\Big[\arg\left(f(t)^{-1}\right)-\e,\arg\left(f(t)^{-1}\right)\Big[\quad \bigcup \quad \Big]\arg\left(f(t)^{-1}\right),\arg\left(f(t)^{-1}\right)+\e\Big].$$
Then, the following matrices are fundamental solutions: 
$$\begin{pmatrix}
1 & \mathcal{L}_{1,\arg\left(f(t)^{-1}\right)+\e}(\log(1-f(t)z)) \\ 
\frac{1}{f(t)z^{2}} &\pz \mathcal{L}_{1,\arg\left(f(t)^{-1}\right)+\e}(\log(1-f(t)z))  \end{pmatrix}\begin{pmatrix}
e^{\frac{-1}{f(t)z}}&0 \\ 
0 &1 \end{pmatrix},$$
$$\begin{pmatrix}
1 & \mathcal{L}_{1,\arg\left(f(t)^{-1}\right)-\e}(\log(1-f(t)z)) \\ 
\frac{1}{f(t)z^{2}} &\pz \mathcal{L}_{1,\arg\left(f(t)^{-1}\right)-\e}(\log(1-f(t)z))  \end{pmatrix}\begin{pmatrix}
e^{\frac{-1}{f(t)z}}&0 \\ 
0 &1 \end{pmatrix}.$$
To compute the Stokes matrix in the direction~$\arg\left(f(t)^{-1}\right)$, we have to compute:~$$\mathcal{L}_{1,\arg\left(f(t)^{-1}\right)+\e}( \log(1-f(t)z))-\mathcal{L}_{1,\arg\left(f(t)^{-1}\right)-\e}(\log(1-f(t)z)) .$$
We have
$$\begin{array}{ccc}
&&\mathcal{L}_{1,\arg\left(f(t)^{-1}\right)+\e}( \log(1-f(t)z))-\mathcal{L}_{1,\arg\left(f(t)^{-1}\right)-\e}(\log(1-f(t)z)) \\\\
& = &z^{-1} \int_{0}^{\infty i(\arg\left(f(t)^{-1}\right)+\e)} \log(1-f(t)u)e^{-(\frac{u}{z})} d(u) \\ \\
 & -  &z^{-1}\int_{0}^{\infty i(\arg\left(f(t)^{-1}\right)-\e)} \log(1-f(t)u)e^{-(\frac{u}{z})} d(u).
\end{array}~$$

Integration by parts and the residue theorem imply that:
$$\mathcal{L}_{1,\arg\left(f(t)^{-1}\right)+\e}( \log(1-f(t)z))-\mathcal{L}_{1,\arg\left(f(t)^{-1}\right)-\e}(\log(1-f(t)z))=2i\pi f(t) e^{-\left(\frac{1}{f(t)z}\right)}.$$ Therefore, the Stokes matrix in this direction is~$\begin{pmatrix}
1 & 2i\pi f(t) \\ 
0 &1 \end{pmatrix}~$. Finally we obtain:
$$\begin{array}{lcl}
Aut_{\pz}^{\dt}\left(\widetilde{\mathcal{M}_{U}(z)}\Big|\mathcal{M}_{U}(z)\right)&\simeq&\left\{ \begin{pmatrix}
\a & b f \\ 
0 &1 \end{pmatrix}, \hbox{ where } \a \in \C^{*}\hbox{ and } b\in \C \right\}\\\\
&\simeq& \left\{ \begin{pmatrix}
\a & \b \\ 
0 &1 \end{pmatrix}, \hbox{ where }\pt \a =0,\a \neq 0 \hbox{ and } \pt \left(\frac{\b}{f}\right)=0 \right\}.
\end{array}~$$
\end{ex}
\pagebreak[3]
\begin{ex}[Bessel equation]
We are interested in the parameterized linear differential equation:
$$ \pz\begin{pmatrix}
 Y(z,t)   \\ 
\pz Y(z,t)
\end{pmatrix}=
\begin{pmatrix}
0 & 1 \\ 
\frac{(t^{2}-z^{2})}{z^{2}} &\frac{-1}{z}
                              \end{pmatrix} 
\begin{pmatrix}
Y(z,t)   \\ 
\pz Y(z,t)
\end{pmatrix}.$$
This equation has two singularities:~$0$ and~$\infty$.  Let~$U$ be a non empty disc such that~${U\cap (1/2+\Z)=\varnothing}$. First, we will compute the local group at~$0$,
~$$Aut_{\pz}^{\dt}\left(\widetilde{\mathcal{M}_{U}(z)}\Big|K_{U,0}\cap\widetilde{\mathcal{M}_{U}(z)}\right).$$ If~$t+1/2 \notin \Z$, the two solutions:
$$J_{t}(z)=\left(\frac{z}{2}\right)^{t} \displaystyle \sum_{k=0}^{\infty} \dfrac{(-1)^{k}z^{2k}}{k!\G (t+k+1)2^{k}}~$$ 
$$J_{-t}(z)=\left(\frac{z}{2}\right)^{-t} \displaystyle \sum_{k=0}^{\infty} \dfrac{(-1)^{k}z^{2k}}{k!\G (-t+k+1)2^{k}},~$$ 
are linearly independent (see \cite{Wat} Page 43) and we have a fundamental solution of the specialized system. The equation is regular singular at~$z=0$, and therefore, the group generated by the monodromy~$\hat{m}$ is dense for Kolchin topology in the parameterized differential Galois group~$Aut_{\pz}^{\dt}\left(\widetilde{\mathcal{M}_{U}(z)}\Big|K_{U,0}\cap\widetilde{\mathcal{M}_{U}(z)}\right)$. By the same reasoning as in Example \ref{3ex1}:
$$Aut_{\pz}^{\dt}\left(\widetilde{\mathcal{M}_{U}(z)}\Big|K_{U,0}\cap\widetilde{\mathcal{M}_{U}(z)}\right) \simeq \left\{ \begin{pmatrix}
\a & 0 \\ 
0 &\a^{-1} \end{pmatrix}, \hbox{ where }\a \neq 0 \hbox{ and }\a\pt^{2}\a -(\pt \a)^{2}=0 \right\}.$$
We now turn to the singularity at infinity. We have:
$$ \pz \begin{pmatrix}
Y(z^{-1},t)   \\ 
\pz Y(z^{-1},t)
\end{pmatrix}=
\begin{pmatrix}
0 & 1 \\ 
\frac{t^{2}}{z^{2}}-\frac{1}{z^{4}} &\frac{-1}{z}
                              \end{pmatrix} 
\begin{pmatrix}
 Y(z^{-1},t)   \\ 
\pz  Y(z^{-1},t)
\end{pmatrix}.$$
In order to compute the matrices of the monodromy, the elements of the exponential torus, and the Stokes operators, we make use of another basis of solutions:
$$H_{t}^{(1)}(z^{-1})=\frac{J_{-t}(z^{-1})-e^{-it\pi}J_{t}(z^{-1})}{i \sin(t\pi)}$$
$$H_{t}^{(2)}(z^{-1})=\frac{J_{-t}(z^{-1})-e^{it\pi}J_{t}(z^{-1})}{-i \sin(t\pi)}.$$
In \cite{Wat} page 198, we find that on the sector~$]-\pi,2\pi[$,~$H_{t}^{(1)}(z^{-1})$ is asymptotic to:
$$\widetilde{H}_{t}^{(1)}(z^{-1})=\left(\frac{2z}{\pi}\right)^{1/2}e^{i(z^{-1}-t\pi/2-\pi/4)}\displaystyle \sum_{k=0}^{\infty} \frac{(-1)^{k}\G (t+k+1/2)z^{k}}{(2i)^{k}k!\G (t-k+1/2)}.$$
The same holds for~$H_{t}^{(2)}(z^{-1})$ on the sector~$]-2\pi,\pi[$:
$$\widetilde{H}_{t}^{(2)}(z^{-1})=\left(\frac{2z}{\pi}\right)^{1/2}e^{-i(z^{-1}-t\pi/2-\pi/4)}\displaystyle \sum_{k=0}^{\infty} \frac{\G (t+k+1/2)z^{k}}{(2i)^{k}k!\G (t-k+1/2)}.$$
It follows that in the basis~$\left(H_{t}^{(1)}(z^{-1}),H_{t}^{(2)}(z^{-1})\right)$, the matrix of the monodromy  is:~$$\begin{pmatrix}
-1&0 \\ 
0&-1
\end{pmatrix}$$ 
and the matrices of the elements of the exponential torus are of the form:
$$\begin{pmatrix}
\a & 0\\ 
0 &\a^{-1}
\end{pmatrix}, \hbox{ where } \a \in \C^{*}.$$
The only level is~$1$ and due to the expression of~$\widetilde{H}_{t}^{1}(z^{-1})$ and~$\widetilde{H}_{t}^{2}(z^{-1})$, the singular directions are the directions~$\frac{\pi}{2}+k\pi$, with~$k\in \Z$. 
By definition, the Stokes matrix in the direction~$\frac{\pi}{2}+k\pi$ is the matrix that sends the asymptotic representation defined on the sector~${](k-1)\pi,(k+1)\pi[}$ to the asymptotic representation defined on the sector~${]k\pi,(k+2)\pi[}$.
In \cite{RM1}, 3.4.12 (see also \cite{Ber}), we find that in the basis~$(H_{t}^{1}(z^{-1}),H_{t}^{2}(z^{-1}))$ the  Stokes matrix in the direction~$\frac{\pi}{2}+2k\pi$ is
$$\begin{pmatrix}
1& 0 \\ 
2e^{2i\pi t}\cos(\pi t)&1
\end{pmatrix},$$
and the Stokes matrix in the direction~$-\frac{\pi}{2}+2k\pi$ is
$$\begin{pmatrix}
1& -2e^{-2i\pi t}\cos(\pi t) \\ 
0&1
\end{pmatrix}.$$
An application of the local and global density theorems (Theorems \ref{3theo1} and \ref{3theo2}) gives that~$${Aut_{\pz}^{\dt}\left(\widetilde{\mathcal{M}_{U}(z)}\Big|K_{U,\infty}\cap\widetilde{\mathcal{M}_{U}(z)}\right)\hbox{ and  }Aut_{\pz}^{\dt}\left(\widetilde{\mathcal{M}_{U}(z)}\Big|\mathcal{M}_{U}(z)\right)}$$ are linear differential algebraic subgroups of~$\mathrm{SL}_{2}(\mathcal{M}_{U})$, because all the matrices we have computed are in~$\mathrm{SL}_{2}(\mathcal{M}_{U})$, which is closed in the Kolchin topology. \par
Let~$C$ be a differentially closed field that contains~$\mathcal{M}_{U}$ and consider~$Gal_{\pz}^{\dt}\left(\widetilde{C(z)}\Big|C(z)\right)$, the parameterized differential Galois group defined in Proposition \ref{3propo2}. First, we are going to compute the Zariski closure~$G$ of~$Gal_{\pz}^{\dt}\left(\widetilde{C(z)}\Big|C(z)\right)$. Let~$C^{*}=C\setminus \{0\}$. From the classification of linear algebraic subgroup of~$\mathrm{SL}_{2}(C)$ (see \cite{VdPS}, Theorem 4.29), there are four possibilities:
\begin{enumerate}
\item~$G$ is conjugate to a subgroup of~$B=\left\{ \begin{pmatrix}
a&b \\ 
0&a^{-1}
\end{pmatrix}\hbox{, where } a\in C^{*},b\in \C \right\}$.
\item~$G$ is conjugate to a subgroup of
$$D_{\infty}=\left\{ \begin{pmatrix}
a &0 \\ 
0&a^{-1}
\end{pmatrix} \bigcup \begin{pmatrix}
0 &b^{-1} \\ 
-b&0
\end{pmatrix}\hbox{, where } a,b\in C^{*} \right\}.$$ 
\item~$G$ is finite.
\item~$G=\rm SL_{2}(C)$.
\end{enumerate}
From Proposition \ref{3propo2}, every matrix that belongs to~$Aut_{\pz}^{\dt}\left(\widetilde{\mathcal{M}_{U}(z)}\Big|\mathcal{M}_{U}(z)\right)$ belongs also to~$Gal_{\pz}^{\dt}\left(\widetilde{C(z)}\Big|C(z)\right)$. Since~$G$ must contain~$$\begin{pmatrix}
1& 0 \\ 
2e^{2i\pi t}\cos(\pi t)&1
\end{pmatrix} \hbox{ and }\begin{pmatrix}
1& -2e^{-2i\pi t}\cos(\pi t) \\ 
0&1
\end{pmatrix},$$
we find that the only possibility is that~$Gal_{\pz}^{\dt}\left(\widetilde{C(z)}\Big|C(z)\right)$ is Zariski dense in ~$ SL_{2}(C)$.
In \cite{C72}, Proposition 42, Cassidy classifies the Zariski-dense differential algebraic subgroups of~$\mathrm{SL}_{2}(C)$. Finally, we have two possibilities:
\begin{itemize}
\item~$Gal_{\pz}^{\dt}\left(\widetilde{C(z)}\Big|C(z)\right)$ is conjugate to~$\mathrm{SL}_{2}(C_{0})$ over~$\mathrm{SL}_{2}(C)$, where $${C_{0}=\{ a\in C(z)| \pz a=\pt a=0\}}.$$
\item~$Gal_{\pz}^{\dt}\left(\widetilde{C(z)}\Big|C(z)\right)=\mathrm{SL}_{2}(C)$.
\end{itemize}
 If~$Gal_{\pz}^{\dt}\left(\widetilde{C(z)}\Big|C(z)\right)$ is conjugate to~$\mathrm{SL}_{2}(C_{0})$ over~$\mathrm{SL}_{2}(C)$, the matrix of the monodromy of the singularity~$0$ is conjugate to a matrix~$M\in \mathrm{SL}_{2}(C_{0})$ over~$\mathrm{SL}_{2}(C)$. Similar matrices have the same eigenvalues, so the eigenvalues of~$M$ are~$e^{2i\pi t}$ and~$e^{-2i\pi t}$, which is not possible if~$M$ belongs to~$\mathrm{SL}_{2}(C_{0})$. Because of Proposition \ref{3propo2}, we find that
$$Aut_{\pz}^{\dt}\left(\widetilde{\mathcal{M}_{U}(z)}\Big|\mathcal{M}_{U}(z)\right)=\mathrm{SL}_{2}(\mathcal{M}_{U}).$$
\end{ex}

\pagebreak[3]
\section{Applications.}

We now give three applications of the parameterized differential Galois theory. In~$\S \ref{3sec31}$, we deal with linear differential equations that are completely integrable (see Definition~\ref{3defi2}). It was proved in \cite{CS} that an equation is completely integrable if and only if its parameterized differential Galois group is conjugate over a differentially closed field to a group of constant matrices. We use the global density theorem (Theorem \ref{3theo1}) to prove that the equation is completely integrable if and only if there exists a fundamental solution such that the matrices of the  topological generators for the Galois group appearing in the global density theorem (Theorem \ref{3theo1}) are constant matrices. As a corollary, we deduce that the equation is completely integrable if and only if the matrices of the topological generators for the Galois group given in the parameterized density theorem are conjugate over~$\mathrm{GL}_{m} (\mathcal{M}_{U})$ to constant matrices. In~$\S \ref{3sec32}$, we study an entry of a Stokes operator at the singularity at infinity of the equation:
$$ \pz^{2} Y(z,t)=(z^{3}+t)Y(z,t).$$
In particular, we prove that it is not~$\pt$-finite: it satisfies no parameterized linear differential equation. This partially answers a question by Sibuya.  In~$\S \ref{3sec33}$, we deal with the inverse problem in the parameterized differential Galois theory. Let~$k$ be a so-called universal~$(\dt)$-field (see~$\S \ref{3sec32}$). We give a necessary condition for a linear differential algebraic subgroup of~$\mathrm{GL}_{m}(k)$ for being the global parameterized differential Galois group for some equation having coefficients in~$k(z)$. The corresponding sufficient condition was proved in \cite{MS12}, Corollary 5.2.

\pagebreak[3]
\subsection{Completely integrable equations.}\label{3sec31}

In this subsection, we study completely integrable equations. See also \cite{GO} for an approach from the point of view of differential Tannakian categories.
\pagebreak[1]
\begin{defi}\label{3defi2}
Let~$A_{0}\in \mathrm{M}_{m}\Big(\mathcal{M}_{U}(z)\Big)$. We say that the linear differential equation~$\partial_{0}Y=A_{0}Y$ is completely integrable if there exist~$A_{1},\dots,A_{n}\in \mathrm{M}_{m}\Big(\mathcal{M}_{U}(z)\Big)$ such that, for all~${0\leq i,j \leq n}$,

$$ \partial_{t_{i}}A_{j}-\partial_{t_{j}}A_{i}=A_{i}A_{j}-A_{j}A_{i},$$
with~$\partial_{t_{0}}=\pz$.
\end{defi}

Sibuya shows in \cite{Si90}, Theorem A.5.2.3, that if the parameterized linear differential equation~$\pz Y(z,t)=A(z,t)Y(z,t)$ is regular singular, then it is isomonodromic (see Page~\pageref{3p3} for the definition) if and only if it is completely integrable. This result is not true in the irregular case. The main reason is the fact that there are more topological generators in the parameterized differential Galois group.
\pagebreak[3]
\begin{propo}\label{3propo6}
Let~$A_{0}(z,t)\in \mathrm{M}_{m}\Big(\mathcal{M}_{U}(z)\Big)$ and let~$\widetilde{ \mathcal{M}_{U}(z)}\Big| \mathcal{M}_{U}(z)$ be the parameterized Picard-Vessiot extension for~$\pz Y(z,t)=A_{0}(z,t)Y(z,t)$ defined in~$\S \ref{3sec24}$. The linear differential equation~${\pz Y(z,t)=A_{0}(z,t)Y(z,t)}$ 
is completely integrable if and only if there exists a fundamental solution~$F(z,t)$ in~$\widetilde{\mathcal{M}_{U}(z)}$ such that the images of the topological generators of~$Aut_{\pz}^{\dt}\left(\widetilde{ \mathcal{M}_{U}(z)}\Big| \mathcal{M}_{U}(z)\right)$ (see Theorem \ref{3theo1}), with respect to the representation associated to~$F(z,t)$, belong to~$\mathrm{GL}_{m}(\C)$.
\end{propo}

\begin{proof}
Let~$C$ be a differentially closed field that contains~$\mathcal{M}_{U}$ and let us consider~$C(z)$ as in~$\S\ref{3sec24}$. Let~$\widetilde{C(z)}\Big| C(z)$ be the parameterized Picard-Vessiot extension for~${\pz Y(z,t)=A_{0}(z,t)Y(z,t)}$, and let~$Gal_{\pz}^{\dt}\left(\widetilde{C(z)}\Big| C(z)\right)$ be the parameterized differential Galois group defined in~$\S \ref{3sec21}$. We recall that if we take a different fundamental solution in~$\widetilde{ \mathcal{M}_{U}(z)}$ to compute the Galois group, we obtain a conjugate linear differential algebraic subgroup of~$\mathrm{GL}_{m}(\mathcal{M}_{U})$.\par
 Using the global density theorem (Theorem \ref{3theo1}), we find that there exists a fundamental solution such that the matrices of the topological generators for the Galois group appearing in Theorem \ref{3theo1} are constant if and only if~$Aut_{\pz}^{\dt}\left(\widetilde{ \mathcal{M}_{U}(z)}\Big| \mathcal{M}_{U}(z)\right)$ is conjugate over~$\mathrm{GL}_{m}(\mathcal{M}_{U})$ to a subgroup of~$\mathrm{GL}_{m}(\C)$. Using Proposition \ref{3propo2}, we find that~$Aut_{\pz}^{\dt}\left(\widetilde{ \mathcal{M}_{U}(z)}\Big| \mathcal{M}_{U}(z)\right)$ is conjugate over~$\mathrm{GL}_{m}(\mathcal{M}_{U})$ to a subgroup of~$\mathrm{GL}_{m}(\C)$ if and only if~$Gal_{\pz}^{\dt}\left(\widetilde{
C(z)}\Big|C(z)\right)$ is conjugate over~$\mathrm{GL}_{m}(C)$ to a subgroup of~$\mathrm{GL}_{m}(C_{0})$, where~$$C_{0}=\left\{ a\in C(z) | \pz a=\partial_{t_{1}}a=\dots= \partial_{t_{n}}a=0\right\}.$$
Proposition 3.9, \cite{CS} says that this occurs if and only if there exist~${A_{1},\dots,A_{n}\in \mathrm{M}_{m}(C(z))}$ such that, for all~${0\leq i,j \leq n}$,
$$ \partial_{t_{i}}A_{j}-\partial_{t_{j}}A_{i}=A_{i}A_{j}-A_{j}A_{i},$$
with~$\partial_{t_{0}}=\pz$.
To finish, we follow the proof of Proposition 1.24 in \cite{DVH}. Let~$0< i \leq n$ and let us consider~$$ \pz A_{i}-\partial_{t_{i}}A_{0}=A_{0}A_{i}-A_{i}A_{0}.$$  By clearing the denominators, we obtain that every entry of every~$z$-coefficient of~$A_{i}$ satisfies a finite set of polynomial equations with coefficients in~$\mathcal{M}_{U}$. Since the polynomial equations have a solution in~$C$, they must have a solution in the algebraic closure of~$\mathcal{M}_{U}$. Using Remark \ref{3rem4}, we find the existence of a non empty polydisc~$U'\subset U$ such that all the~$A_{i}$ belong to~$\mathrm{M}_{m}\Big(\mathcal{M}_{U'}(z)\Big)$. This concludes the proof.
\end{proof}

In the proof of Proposition \ref{3propo6}, we have proved:
\pagebreak[3]
\begin{coro}
 Let~$A(z,t)\in \mathrm{M}_{m}\Big(\mathcal{M}_{U}(z)\Big)$. The equation  ~${\pz Y(z,t)=A(z,t)Y(z,t)}$, is completely integrable if and only if the matrices of the topological generators for the Galois group appearing in Theorem \ref{3theo1} are conjugate over~$\mathrm{GL}_{m}(\mathcal{M}_{U})$ to constant matrices.
\end{coro}
\pagebreak[3]
\begin{rem}\label{3rem1}
This corollary improves Proposition 3.9 in \cite{CS}. The conjugation occurs in a field that is not differentially closed. Furthermore, we do not need the entire parameterized differential Galois group to be conjugate to a group of constant matrices in order to deduce that the equation~${\pz Y(z,t)=A(z,t)Y(z,t)}$ is completely integrable. 
\end{rem} 
In \cite{GO}, the authors study completely integrable parameterized linear differential equations using differential Tannakian categories. In particular, they prove that the notion of integrability with respect to all the parameters is equivalent to the notion of integrability with respect to each parameter separately, which generalizes \cite{D1}, Proposition~9. Furthermore, they improve Proposition 3.9 in \cite{CS} by avoiding the assumption that the field of constants is differentially closed.

\pagebreak[3]
\subsection{On the hyper transcendence of a Stokes matrix.}\label{3sec32}

In this subsection, we will study the parameterized linear differential equation:
\begin{equation}\label{3eq4}
\pz^{2} Y(z,t)=(z^{3}+t)Y(z,t).
\end{equation}
Sibuya, in Chapter 2 of \cite{Si}, shows that there exists a formal solution~$y_{0}(z,t)$  which admits an asymptotic representation~$\widetilde{y}_{0}(z,t)$ on the sector (see Theorem 6.1 in \cite{Si}):~$$\left\{ z\in \widetilde{\C} \Big| \arg(z) \in \left]\frac{-3\pi}{5},\frac{3\pi}{5}\right[ \right\} .$$ 
We easily check that for~$k\in \Z$,
$$y_{k}(z,t)=y_{0}\left(e^{\frac{-2ki\pi}{5}}z,e^{\frac{-6ki\pi}{5}}t\right)$$
is a solution of the differential equation (\ref{3eq4}) which has the asymptotic representation~${\widetilde{y}_{k}(z,t)=\widetilde{y}_{0}\Big(e^{\frac{-2ki\pi}{5}}z,e^{\frac{-6ki\pi}{5}}t\Big)}$ on the sector~$S_{k-1}\cup \bar{S}_{k}\cup S_{k+1},$ where
$$S_{k}=\left\{ z \in \widetilde{\C}\Big| \arg(z) \in \left]\frac{(2k-1)\pi}{5},\frac{(2k+1)\pi}{5}\right[ \right\} ,$$
and~$\bar{S}_{k}$ is its closure. 

$$\begin{tabular}{ccc}

\includegraphics[width=0.3\linewidth]{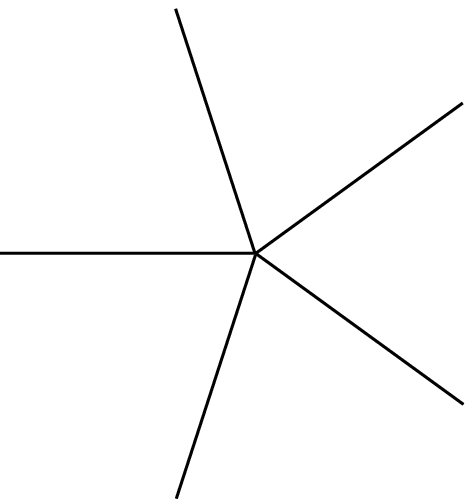}

&
\includegraphics[width=0.3\linewidth]{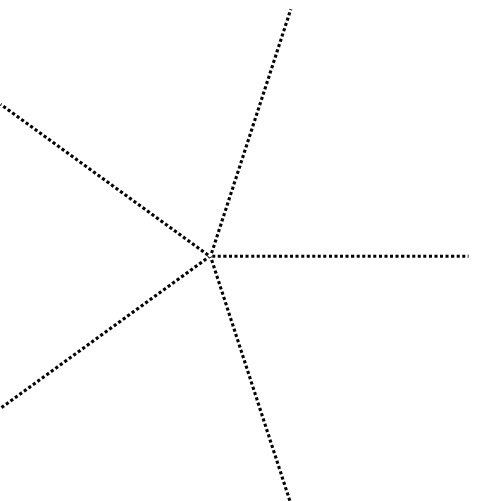}

&
\includegraphics[width=0.3\linewidth]{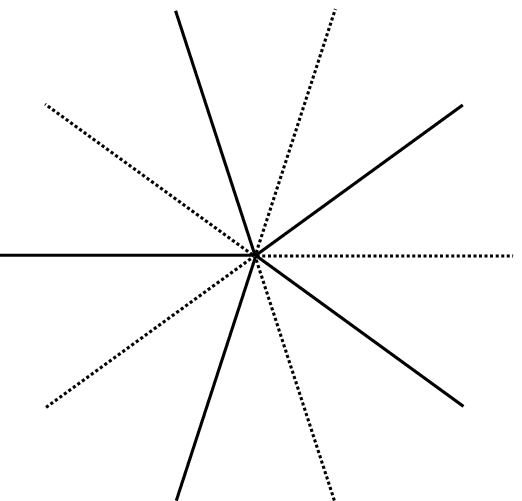}
 \\
\\
The sectors~$S_{k}$& The singular directions& Sectors~$S_{k}$ and singular directions
\end{tabular}$$

 The asymptotic representation~$\widetilde{y}_{k}(z,t)$ is bounded uniformly on each compact set in the~$t$-plane as~$|z|$ tends to infinity on the sector~$S_{k}$, and tends to infinity uniformly on each compact set in the~$t$-plane as~$|z|$ tends to infinity on the sectors~$S_{k-1}$ and~$S_{k+1}$. As we see in  \cite{Si}, page 83,~$y_{k+1}(z,t)$ and~$y_{k+2}(z,t)$ are linearly independent and we can write
$y_{k}(z,t)$ as a~$\mathcal{M}_{\C}$-linear combination  of ~$y_{k+1}(z,t)$ and~$y_{k+2}(z,t)$: \begin{eqnarray}\label{3eq6}\forall k\in \N, \forall z,t\in \C, \; y_{k}(z,t)=C_{k}(t)y_{k+1}(z,t)+\widetilde{C}_{k}(t)y_{k+2}(z,t),\end{eqnarray}
where~$\widetilde{C}_{k}(t),C_{k}(t) \in \mathcal{M}_{\C}$.
By Theorem 21.1 in \cite{Si}, we obtain that~$$\widetilde{C}_{k}(t)=-e^{\frac{2i\pi}{5}} \hbox{ and } C_{k}(t)=C_{0}\left(e^{\frac{-6ki\pi}{5}}t\right).$$
In \cite{Si}, the author asks if~$C_{0}(t)$ is differentially transcendental, i.e., satisfies no differential polynomial equations. We will use Galois theory to prove that for every non empty polydisc~$U$,~$C_{0}(t)$ is not~$\pt$-finite over~$\mathcal{M}_{U}$, i.e., satisfies no linear differential equations in coefficients in~$\mathcal{M}_{U}$. \\
The singularity of the system is at infinity. Let~$W(z,t)=zY(z^{-1},t)$. We obtain the parameterized linear differential equation 
\begin{eqnarray}\label{3eq5}z^{7} \pz^{2}W(z,t)=(1+tz^{3})W(z,t),\end{eqnarray}
which can be written in the form
$$\pz \begin{pmatrix}
 W(z,t)   \\ 
\pz W(z,t)
\end{pmatrix}=
\begin{pmatrix}
0 & 1 \\ 
\frac{1+tz^{3}}{z^{7}} &0
                              \end{pmatrix}\begin{pmatrix}
W(z,t)   \\ 
\pz W(z,t)
\end{pmatrix}.$$
Let~$k$ be a so-called universal~$(\dt)$-field of characteristic~$0$: for any~$(\dt)$-field~$k_{0}\subset k$,~$(\dt)$-finitely generated over~$\Q$, and any~$(\dt)$-finitely generated extension~$k_{1}$ of~$k_{0}$, there is a~$(\dt)$-differential~$k_{0}$-isomorphism of~$k_{1}$ into~$k$. See Chapter~$3$, Section~$7$ of \cite{Kol73} for more details. In particular,~$k$ is~$(\dt)$-differentially closed. Let~$k(z)$ denotes the~$(\pz,\dt)$-differential field of rational functions in the indeterminate~$z$, with coefficients in~$k$, where~$z$ is a~$(\dt)$-constant with~$\pz z =1$,~$k$ is the field of constants with respect to~$\pz$, and~$\pz$ commutes with all the derivations.\par
Let~$A(z,t)=\begin{pmatrix} 
0 & 1 \\ 
\frac{1+tz^{3}}{z^{7}} &0
                              \end{pmatrix}$. The two solutions~$zy_{1}(z^{-1},t)$,~$zy_{2}(z^{-1},t)$ admit asymptotic representations and the only singularity is~$0$. Therefore,~$$\mathcal{M}_{U}(z)\langle y_{1}(z^{-1},t),y_{2}(z^{-1},t)\rangle_{\pz,\pt}\Big|\mathcal{M}_{U}(z)=\widetilde{\mathcal{M}_{U}(z)}\Big|\mathcal{M}_{U}(z)$$ is a parameterized Picard-Vessiot extension for~$\pz W(z,t)=A(z,t)W(z,t)$. Because of Proposition~\ref{3propo2},~$\widetilde{k(z)}\Big|k(z)=k(z)\Big\langle y_{1}(z^{-1},t),y_{2}(z^{-1},t)\Big\rangle_{\pz,\pt}\Big|k(z)$ is a parameterized Picard-Vessiot extension.
\pagebreak[3]
\begin{lem}\label{3lem4}
$Gal_{\pz}^{\dt}\left(\widetilde{k(z)}\Big|k(z)\right)=\mathrm{SL}_{2}(k)$.
\end{lem}

Notice that the differential equation is of the form~$\pz^{2} W(z,t)=r(z,t)W(z,t)$, where~$r(z,t)\in k(z)$. In this case, we can compute the Galois group using a parameterized version of Kovacic's algorithm, see \cite{Ar,D1}. In order to have a self contained proof, we will perform the calculations explicitly.  

\begin{proof}
 If we apply Kovacic's algorithm (see \cite{Kov}), we find that the unparameterized differential Galois group~$Gal_{\pz}\left(\widetilde{k(z)} \Big| k(z)\right)$ is equal to~$\mathrm{SL}_{2}(k)$.
We apply Proposition 6.26 in \cite{HS}, to deduce that~$Gal_{\pz}^{\dt}\left(\widetilde{k(z)}\Big| k(z)\right)$ is Zariski-dense in~$\mathrm{SL}_{2}(k)$. By Proposition 42 in \cite{C72}, we deduce that there are two possibilities:
\begin{itemize}
\item ~$Gal_{\pz}^{\dt}\left(\widetilde{k(z)}\Big| k(z)\right)=\mathrm{SL}_{2}(k)$
\item~$Gal_{\pz}^{\dt}\left(\widetilde{k(z)}\Big|k(z)\right)$ is conjugate to~$\mathrm{SL}_{2}(k_{0})$ over~$\mathrm{SL}_{2}(k)$, where $$k_{0}=\left\{ a\in k(z) | \pz a=\pt a=0\right\}.$$
\end{itemize}
  We see in \cite{D1}, Remark 4.4, that the last case occurs if and only if the following parameterized differential equation has a solution in~$\mathcal{M}_{U}(z)$, for some non empty polydisc~$U$ in~$\C^{n}$:
$$\pz^{3} y(z,t)=\pz y(z,t)\frac{4+4tz^{3}}{z^{7}}+y(z,t)\pz \frac{4+4tz^{3}}{z^{7}}-\pt \frac{4+4tz^{3}}{z^{7}}.$$
With the algorithm presented in \cite{VdPS} p.100, we find that this does not happen and then:
$$Gal_{\pz}^{\dt}\left(\widetilde{k(z)}\Big| k(z)\right)=\mathrm{SL}_{2}(k).$$
\end{proof}
\pagebreak[3]
\begin{lem}
The singular directions of the equation (\ref{3eq5}) are:
$$ \frac{2k\pi}{5}\hbox{,with } k\in \Z.$$
\end{lem}

\begin{proof}
Let~$k\in \Z$. The matrix~$$\begin{pmatrix} zy_{k}(z^{-1},t) & zy_{k+1}(z^{-1},t) \\ 
\pz zy_{k}(z^{-1},t) & \pz zy_{k+1}(z^{-1},t) 
\end{pmatrix},$$ is a fundamental solution for the equation~$ \pz\begin{pmatrix}
 W(z,t)   \\ 
\pz W(z,t)
\end{pmatrix}=
\begin{pmatrix}
0 & 1 \\ 
\frac{1+tz^{3}}{z^{7}} &0
                              \end{pmatrix}\begin{pmatrix}
W(z,t)   \\ 
\pz W(z,t)
\end{pmatrix}.$
  The fundamental solution admits an asymptotic representation on the sectors:~$$\left\{ z\in \widetilde{\C}  \Big| \arg(z) \in \left]\frac{(2k-1)\pi}{5},\frac{(2k+3)\pi}{5}\right[ \right\} .$$
 The only level is~$\frac{5}{2}$. From Proposition \ref{3propo3} and the construction of the singular directions, we find that the singular directions are
$$ \frac{2k\pi}{5}\hbox{, with } k\in \Z.$$
\end{proof}
\pagebreak[3]
\begin{ex}\label{3ex2}
We want to compute the Stokes matrix in the direction~$\frac{8\pi}{5}$ for the fundamental solution:
$$\begin{pmatrix} zy_{1}(z^{-1},t) & zy_{2}(z^{-1},t) \\ 
\pz zy_{1}(z^{-1},t) & \pz zy_{2}(z^{-1},t) 
\end{pmatrix}.~$$
We recall the construction of the Stokes matrices. See~$\S \ref{3sec13}$ for the notations. Let~$\hat{H}(z,t)z^{L(t)}e\Big(Q(z,t)\Big)$ be a fundamental solution in the parameterized Hukuhara-Turrittin canonical form.
Let~$H^{-}(z,t)$ (resp.~$H^{+}(z,t)$) be the matrix such that
$$\begin{array}{ll}
H^{-}(z,t)e^{L(t)\log (z)}e^{Q(z,t)}&\left(\hbox{resp. }  H^{+}(z,t)e^{L(t)\log (z)}e^{Q(z,t)}\right)
\end{array}~$$
 is the germ of an asymptotic solution on the sector 
$$\begin{array}{ll}
\left\{ z\in \widetilde{\C}  \Big|\arg(z)\in \left]\pi,\frac{9\pi}{5}\right[\right\} &\left(\hbox{resp. }  \left\{z\in \widetilde{\C}  \Big|\arg(z)\in \left]\frac{7\pi}{5},\frac{11\pi}{5}\right[\right\}\right) 
\end{array}~$$
The Stokes matrix in the direction~$\frac{8\pi}{5}$ is the matrix that sends~$$H^{-}(z,t)e^{L(t)\log (z)}e^{Q(z,t)} \hbox{ to } H^{+}(z,t)e^{L(t)\log(z)}e^{Q(z,t)}.$$  
With the domain of definition of the asymptotic representation of~$z\widetilde{y}_{1}(z^{-1},t)$, we deduce from the definition of the Stokes operators that:
\begin{equation}\label{3eq8}
St_{\frac{8\pi}{5}}\Big(zy_{1}(z^{-1},t)\Big)=zy_{1}(z^{-1},t).
\end{equation}
We first write~$St_{\frac{8\pi}{5}}\Big(zy_{2}(z^{-1},t)\Big)$ in the basis~$$\Big(zy_{0}(z^{-1},t),zy_{1}(z^{-1},t)\Big).$$ There exist~$a(t)$ and~$b(t)\in \mathcal{M}_{U}$ such that:~$$St_{\frac{8\pi}{5}}\Big(zy_{2}(z^{-1},t)\Big)=a(t)zy_{0}(z^{-1},t)+b(t)zy_{1}(z^{-1},t).$$ 
By the construction of the asymptotic solutions with Laplace and Borel transforms (see Proposition~\ref{3propo1}), the asymptotic representation of~$St_{\frac{8\pi}{5}}\Big(zy_{2}(z^{-1},t)\Big)$ has to be bounded in some sector of~$\left]\frac{7\pi}{5},\frac{11\pi}{5}\right[$, which means that there exist~$\frac{7\pi}{5}<\a<\b<\frac{11\pi}{5}~$ and~$\e>0$ such that~$St_{\frac{8\pi}{5}}\Big(zy_{2}(z^{-1},t)\Big)$ is uniformly bounded for~$\arg(z)\in ]\a,\b[$ and~$z<|\e|$. Therefore,~$a(t)=0$ or~$b(t)=0$. Since the Stokes operators are automorphisms, we get~$b(t)=0$. Lemma \ref{3lem4} says that the parameterized differential Galois group is~$\mathrm{SL}_{2}(k)$. Therefore, because of Proposition \ref{3propo2} and Lemma \ref{3lem5}, the determinant of the matrix has to be~$1$. Thus by (\ref{3eq6}), we get that the Stokes matrix in direction~$\frac{8\pi}{5}$ is:~$$St_{\frac{8\pi}{5}}=\begin{pmatrix} 1 & -C_{0}(t)e^{\frac{3i\pi}{5}} \\ 
0 & 1 
\end{pmatrix}.$$ 
\end{ex}
\pagebreak[3]
\begin{lem}\label{3lem7}
Let~$C_{0}(t)$ be defined as above. Assume that~$C_{0}(t)$ is~$\pt$-finite over~$k$. Then, the~$\pt$-differential transcendence degree (see~$\S \ref{3sec21}$ for definition) of~$\widetilde{k(z)}$ over~$k(z)$ is at most ~$2$. 
\end{lem}

\begin{proof}
 The extension field~$\widetilde{k(z)}$ is generated over~$k(z)$ by~$ y_{1}(z^{-1},t)$ and~$y_{2}(z^{-1},t)$. 
By the parameterized differential Galois correspondence (see Theorem 9.5 in \cite {CS}), the Kolchin closure of the group generated by~$St_{\frac{8\pi}{5}}$ is equal to 
$$Gal_{\pz}^{\dt}\left(\widetilde{k(z)}\Big| F\right),$$
where~$F$ is the subfield of~$\widetilde{k(z)}$ fixed by~$St_{\frac{8\pi}{5}}$. Using (\ref{3eq8}), we deduce that~$F$ contains~$$k(z) \Big\langle y_{1}(z^{-1},t) \Big\rangle_{\pz,\pt}.$$
Because ~$C_{0}(t)$ satisfies a linear differential equation with coefficients in~$k$, there exists~$P$, a linear differential polynomial such that this group is of the form
$$ \left\{ \begin{pmatrix}
1 & \a \\ 
0 &1
\end{pmatrix}\hbox{, with } P(\a)=0=P(C_{0}(t)) \right\},$$ and has~$\pt$-differential dimension over~$k$ equal to~$0$. Therefore by Proposition \ref{3propo9}, the~$\pt$-differential transcendence degree of~$\widetilde{k(z)}$ over~$F$ is equal to~$0$. Because of the fact that~$F$ contains~$k(z) \Big\langle y_{1}(z^{-1},t) \Big\rangle_{\pz,\pt},$ there exists a differential polynomial~$Q$ with coefficients in~$k(z)$ such that:~$$Q\Big(y_{1}(z^{-1},t),y_{2}(z^{-1},t)\Big)=0=Q\Big(\pz ( y_{1}(z^{-1},t)),\pz ( y_{2}(z^{-1},t))\Big).$$ 
Therefore, the~$\pt$-differential transcendence degree of~$\widetilde{k(z)}$ over~$k(z)$ is at most~$2$, because~$\widetilde{k(z)}$ is generated as a~$\pt$-differential field over~$k(z)$ by~$$\Big\{y_{1}(z^{-1},t),y_{2}(z^{-1},t),\pz ( y_{1}(z^{-1},t)),\pz (y_{2}(z^{-1},t))\Big\}.$$
\end{proof}
\pagebreak[3]
\begin{theo}
The function~$C_{0}(t)$ is not~$\pt$-finite over~$k$.
\end{theo}

\begin{proof}
As we see from Lemma \ref{3lem4},~$$Gal_{\pz}^{\dt}\left(\widetilde{k(z)} \Big| k(z)\right)=\mathrm{SL}_{2}(k).$$
Therefore, by Proposition \ref{3propo9}, the~$\pt$-differential transcendence degree of~$\widetilde{k(z)}$ over~$k(z)$ is~$3$. If~$C_{0}(t)$ was~$\pt$-finite over~$k$, because of Lemma  \ref{3lem7}, the~$\pt$-differential transcendence degree of~$\widetilde{k(z)}$ over~$k(z)$ would be smaller than~$3$. Therefore,~$C_{0}(t)$ is not~$\pt$-finite over~$k$.
\end{proof}

\pagebreak[3]
\subsection{Which linear differential algebraic groups are parameterized differential Galois groups?}\label{3sec33}

As in~$\S \ref{3sec32}$, let~$k$ be a so-called universal~$(\dt)$-field of characteristic~$0$. Let us consider an equation~$\pz Y(z,t)=A(z,t)Y(z,t)$, with~$A(z,t)\in \mathrm{M}_{m}(k(z))$, let~$\widetilde{k(z)}\Big| k(z)$ be the parameterized Picard-Vessiot extension, and let~$G=Gal_{\pz}^{\dt}\left(\widetilde{k(z)}\Big| k(z)\right)\subset \mathrm{GL}_{m}(k)$ be the parameterized differential Galois group defined in~$\S \ref{3sec21}$. The following theorem of Seidenberg, applied with~$K_{0}=\Q$ and~$K_{1}$, the~$(\dt)$-field generated by~$\Q$ and the~$z$-coefficients of~$A(z,t)$, tells us that there exists a non empty polydisc~$U$ such that~$A(z,t)$ may be seen as an element of~$\mathrm{M}_{m}(\mathcal{M}_{U}(z))$. 
\pagebreak[3]
\begin{theo}[Seidenberg,\cite{Sei58,Sei69}]
Let~$\Q\subset K_{0} \subset K_{1}$ be finitely generated~$(\dt)$-differential extensions of~$\Q$, and assume that~$K_{0}$ consists of meromorphic functions on some domain~$U$ of~$\C^{n}$. Then,~$K_{1}$ is isomorphic to the field~$K_{1}^{*}$ of meromorphic functions on  a non empty polydisc~$U'\subset U$ such that~$K_{0}|_{U'}\subset K_{1}^{*}$, and the derivations in~$\dt$ can be identified with the derivations with respect to the coordinates on~$U'$.
\end{theo}

Let~$\widetilde{ \mathcal{M}_{U}(z)}\Big| \mathcal{M}_{U}(z)$ be the parameterized Picard-Vessiot extension defined in~$\S \ref{3sec24}$ and let~$Aut_{\pz}^{\dt}\left(\widetilde{ \mathcal{M}_{U}(z)}\Big| \mathcal{M}_{U}(z)\right)$ be the parameterized differential Galois group. Using Corollary \ref{3coro1}, we find that~$Aut_{\pz}^{\dt}\left(\widetilde{ \mathcal{M}_{U}(z)}\Big| \mathcal{M}_{U}(z)\right)$ contains a finitely generated subgroup that is Kolchin-dense in~$Aut_{\pz}^{\dt}\left(\widetilde{ \mathcal{M}_{U}(z)}\Big| \mathcal{M}_{U}(z)\right)$. With Proposition \ref{3propo2}, we find that~$G$ contains a finitely generated subgroup that is Kolchin-dense in~$G$. Combined with Corollary 5.2 in \cite{MS12}, which gives the sufficiency of the condition, this yields the following result:
\pagebreak[3]
\begin{theo}[Inverse problem]\label{3theo3}
Let~$G$ be a  linear differential algebraic subgroup of~$\mathrm{GL}_{m}(k)$. Then,~$G$ is the global parameterized differential Galois group of some equation having coefficients in~$k(z)$ if and only if~$G$ contains a finitely generated subgroup that is Kolchin-dense in~$G$.
\end{theo}

In the unparameterized case, any linear algebraic group defined over~$\C$ is a Galois group of a Picard-Vessiot extension (see \cite{Tre}). In fact, every linear algebraic group defined over~$\C$ contains a finitely generated subgroup that is Zariski-dense, which means that Theorem \ref{3theo3}  is a generalization of the result in \cite{Tre}. 
\par The situation is more complicated in the parameterized case. For example, the additive group:~$$ \left\{ \begin{pmatrix}
1 & \a \\ 
0 &1
\end{pmatrix}\hbox{, with } \a\in k \right\},$$ is not the global parameterized differential Galois group of any equation having coefficients in~$k(z)$ (see Section 7 of \cite{CS}). 
In the parameterized case with only regular singular poles, the problem was solved in \cite{MS12}, Corollary 5.2: they obtain the same necessary and sufficient condition on the group as in Theorem \ref{3theo3}. In \cite{S13}, the author characterizes the linear algebraic subgroups of~$\mathrm{GL}_{m}(k)$ that appear as the global parameterized differential Galois groups of some equation having coefficients in~$k(z)$: they are the groups such that the identity component has no quotient isomorphic to the additive group~$(k,+)$ or multiplicative group~$(k^{*},\times)$ of~$k$.

\appendix

\pagebreak[3]
\section{}
Let us keep the same notations as in~$\S \ref{3sec11}$ and~$\S \ref{3sec12}$. The goal of the appendix is to prove the following theorem. Notice that our proof closely follows the unparameterized case, see \cite{BJL,LR01}. See Remark \ref{3rem2} for a discussion of another similar result.
\pagebreak[3]
\begin{theo}\label{3theo4}
 Let~$\pz Y(z,t)=A(z,t)Y(z,t)$, with~$A(z,t)\in \mathrm{M}_{m}\Big(\hat{K}_{U}\Big)$. 
There exists a non empty polydisc~$U'\subset U$ such that we have a fundamental solution  of the form 
$$\hat{P}(z,t) z^{C(t)}e\Big(Q(z,t)\Big)\in \mathrm{GL}_{m}\left(\widehat{\textbf{K}_{U'}}\right),$$ 
with:
\begin{itemize}
 \item ~$\hat{P}(z,t)\in \mathrm{GL}_{m}\Big(\hat{K}_{U'}\Big)$.
 \item ~$C(t) \in \mathrm{M}_{m}(\mathcal{M}_{U'})$.
 \item ~$e\Big(Q(z,t)\Big)=\mathrm{Diag}\Big(e(q_{i}(z,t))\Big)$, with~$q_{i}(z,t) \in \textbf{E}_{U'}$.
\end{itemize}
Moreover, we may choose the same non empty polydisc~$U'$ as in Proposition \ref{3propo4}. Combined with Remark \ref{3rem2}, if~$A(z,t)\in \mathrm{M}_{m}\Big(\mathcal{O}_{U}(\{z\})\Big)$, this gives a sufficient condition on~$t_{0}\in U$, to have a fundamental solution~$\hat{P}(z,t) z^{C(t)}e\Big(Q(z,t)\Big)\in \mathrm{GL}_{m}\Big(\widehat{\textbf{K}_{U'}}\Big)$ in the same form as above with~$t_{0}\in U'$.
\end{theo}
Remark that contrary to Proposition \ref{3propo4},~$\hat{H}(z,t)\in \mathrm{GL}_{m}\Big(\hat{K}_{U'}\Big)$. On the other hand, we lose the commutation between~$ z^{C(t)}$ and~$e\Big(Q(z,t)\Big)$. Before giving the proof of the theorem, we state and prove two lemmas.
\pagebreak[3]
\begin{lem}\label{3lem1}
Let~$U'\subset U$ be a non empty polydisc. Let~$a(t)\in \mathcal{M}_{U'}$ and~$\a(z,t) \in \hat{K}_{F,U'}$ such that ~$\hat{m}(\a(z,t))=a(t)\a(z,t)$. Then there exist~$\hat{h}(z,t) \in \hat{K}_{U'}$ and~$b(t)\in \mathcal{M}_{U'}$ such that~$\a(z,t)=\hat{h}(z,t)z^{b(t)}$.
\end{lem}

\begin{proof}
 Let~$\a(z,t)\in  \hat{K}_{F,U'}$ such that~$\hat{m}(\a(z,t))=a(t)\a(z,t)$. The element~$\a(z,t)$ belongs to the fraction field of a free polynomial ring:
$$P=\hat{K}_{U'}\left[\log,z^{b_{1}(t)},\dots,z^{b_{k}(t)}\right].$$
Write~$\a(z,t)=\a_{1}(z,t)/\a_{2}(z,t)$ with~$\gcd$ in~$P$ equals to 1.
Using the relations in~$\hat{K}_{F,U'}$, and applying~$\hat{m}$ to~$\a_{1}(z,t)/\a_{2}(z,t)$, we find that~$\a(z,t)$ contains no terms in~$\log$. 
One can normalize~$\a_{2}(z,t)$ such that it contains a term of the form~$z^{ n_{1}b_{1}(t)+\dots+n_{k}b_{k}(t)}$ with coefficient~$1$ and~$n_{i}\in \Z$. Using~${\hat{m}(\a_{1}(z,t)/\a_{2}(z,t))=a(t)\a_{1}(z,t)/\a_{2}(z,t)}$, we find that~${\hat{m}(\a_{2}(z,t))=e^{2i\pi\left(n_{1}b_{1}(t)+\dots+n_{k}b_{k}(t)\right)}\a_{2}(z,t)}$ and~$\hat{m}(\a_{1}(z,t))$ erquals to $a(t)e^{2i\pi\left(n_{1}b_{1}(t)+\dots+n_{1}b_{1}(t)\right)}\a_{1}(z,t)$,
which is impossible, unless~$$e^{2i\pi\left(n_{1}b_{1}(t)+\dots+n_{k}b_{k}(t)\right)}=1.$$
This means that~$\a_{2}(z,t)\in \hat{K}_{U'}$ and we may assume~$\a_{2}(z,t)=1$. Applying~$\hat{m}$ to~$\a_{1}(z,t)$, one finds that~$\a_{1}(z,t)$ contains at most one term, that is~$\a(z,t)=\hat{h}(z,t)z^{b(t)}$, with ~$\hat{h}(z,t)\in \hat{K}_{U'}$ and~$b(t)\in \mathcal{M}_{U'}$ satisfying~$e^{2i\pi b(t)}=a(t)$.
\end{proof}
\pagebreak[3]
\begin{lem}\label{3lem6}
Let~$U'\subset U$ be a non empty polydisc. Let~$A(z,t)\in \mathrm{M}_{m}\Big(\hat{K}_{U'}\Big)$. Let~$F_{1}(z,t)e\Big(Q_{1}(z,t)\Big)$ and  ~$F_{2}(z,t)e\Big(Q_{2}(z,t)\Big)$ be two fundamental solutions of:
$$\pz Y(z,t)=A(z,t)Y(z,t),$$
satisfying, for~$i\in \{ 1;2 \}$,~$F_{i}(z,t)\in \mathrm{GL}_{m}\Big(\hat{K}_{F,U'}\Big)$ and~$Q_{i}(z,t)=\mathrm{Diag}[q_{i,j}(z,t)]$ such that~$q_{i,j}(z,t)$ belongs to~$\textbf{E}_{U'}$.
Then,~$F_{1}(z,t)^{-1}F_{2}(z,t)\in \mathrm{GL}_{m}(\mathcal{M}_{U'})$.
\end{lem}

\begin{proof}
 A straightforward computation shows that:
$$ \pz \Bigg(\Big(F_{1}(z,t)e(Q_{1}(z,t))\Big)^{-1}F_{2}(z,t)e(Q_{2}(z,t))\Bigg)=0.$$
By Proposition \ref{3propo5},~$$\Big(F_{1}(z,t)e(Q_{1}(z,t))\Big)^{-1}F_{2}(z,t)e(Q_{2}(z,t))=C(t)\in \mathrm{GL}_{m}(\mathcal{M}_{U'}).$$
Hence, we have the equality:
$$e(Q_{1}(z,t))C(t)e(-Q_{2}(z,t))=F_{1}(z,t)^{-1}F_{2}(z,t).~$$
The entries of~$e(Q_{1}(z,t))C(t)e(-Q_{2}(z,t))$ are of the form~$C_{i,j}(t)e(q_{1,j}(z,t)-q_{2,j}(z,t))$, with~$C_{i,j}(t)$ that belongs to~$\mathcal{M}_{U'}$,
and the matrix~$F_{1}(z,t)^{-1}F_{2}(z,t)$ belongs to~$\mathrm{GL}_{m}\Big(\hat{K}_{F,U'}\Big)$. By construction,~$$\hat{K}_{F,U'}\cap \mathcal{M}_{U'}\Big(\left(e(q(z,t))\right)_{q(z,t) \in \textbf{E}_{U'}}\Big)=\mathcal{M}_{U'},$$ and we obtain that:
$$F_{1}(z,t)^{-1}F_{2}(z,t)\in \mathrm{GL}_{m}(\mathcal{M}_{U'}).$$
\end{proof}

\begin{proof}[Proof of Theorem \ref{3theo4}]
By Proposition \ref{3propo4}, we know that we have a fundamental solution of the parameterized linear differential equation~$\pz Y(z,t)=A(z,t)Y(z,t)$ of the form: 
$$\hat{H}(z,t) z^{L(t)}e\Big(Q(z,t)\Big),$$ with~$\hat{H}(z,t)\in \mathrm{GL}_{m}\Big(\hat{K}_{U'}[z^{1/ \nu}]\Big)$ and~$\nu\in \N^{*}$. 
From Definition \ref{3defi1},~$\hat{m}$ commutes with the derivation~$\pz$, and therefore~$\hat{m}\left(\hat{H}(z,t) z^{L(t)}e(Q(z,t))\right)$ is another  fundamental solution. From the construction of~$\hat{m}$, we deduce that~$\hat{m}\left(\hat{H}(z,t) z^{L(t)}\right)\in \mathrm{GL}_{m}\Big(\hat{K}_{F,U'}\Big)$, and we can apply Lemma \ref{3lem6} to 
 deduce the existence of~$\hat{M}(t)\in \mathrm{GL}_{m}(\mathcal{M}_{U'})$ such that: 
\begin{equation}\label{3eq1}
\hat{m}\left(\hat{H}(z,t) z^{L(t)}\right)=\hat{H}(z,t) z^{L(t)}\hat{M}(t).
\end{equation}
 Let us consider~$\hat{M}(t)=D(t)U(t)$, with~$D(t)$ diagonalizable and~$U(t)$ unipotent such that~${D(t)U(t)=U(t)D(t)}$ is the multiplicative analogue of the Jordan decomposition of~$\hat{M}(t)$. If~$a (t)$ is an eigenvalue of~$D(t)$ (and therefore an eigenvalue of~$\hat{M}(t)$), then
there exists~$0\neq \a(z,t) \in \hat{K}_{F,U'}$ such that~${\hat{m}(\a(z,t))=a (t) \a (z,t)}$, because of the relation (\ref{3eq1}). By Lemma \ref{3lem1}, ~$\a (z,t)$ is equal to~$\hat{h}(z,t)z^{b(t)}$, with~${b(t) \in \mathcal{M}_{U'}}$ satisfying~${e^{2i\pi b (t)}=a (t)}$ and~$\hat{h}(z,t)\in \hat{K}_{U'}$. This implies that~$a(t)$ and all the eigenvalues of~$D(t)$ are of the form~$e^{\b (t)}$, with~${\b (t) \in \mathcal{M}_{U'}}$. So we have proved the existence of~$C(t)\in \mathrm{M}_{m}(\mathcal{M}_{U'})$ such that~$e^{2i\pi C(t)}=\hat{M}(t)$. Let:~$$\hat{P}(z,t)=\hat{H}(z,t)z^{L(t)}z^{-C(t)}.$$ A computation shows that the monodromy of~$z^{C(t)}$ is:
$$\hat{m}\left(z^{C(t)}\right)=e^{2i\pi C(t)}z^{C(t)}=z^{C(t)}e^{2i\pi C(t)}.$$ The matrix~$\hat{P}(z,t)$ is fixed by the monodromy and therefore belongs to~$\mathrm{GL}_{m}\Big(\hat{K}_{U'}\Big)$, because of Proposition \ref{3propo5}. Finally,~$$\hat{P}(z,t)z^{C(t)}e\Big(Q(z,t)\Big)$$ is a fundamental solution of the parameterized linear differential equation~${\partial_{z}Y(z,t)=A(z,t)Y(z,t)}$ that has the required property.
\end{proof}
\fussy 
\pagebreak[3]
\nocite{A,AMW}
\bibliographystyle{alpha}
\bibliography{C:/Users/thomas.dreyfus/Dropbox/Maths/biblio}
\end{document}